\documentclass[final,12pt]{colt2026} 

\usepackage{mathrsfs}
\usepackage{hyperref}
\usepackage{xcolor}
\usepackage{nicefrac}
\usepackage[margin=1cm]{caption}
\usepackage{subcaption}
\usepackage[makeroom]{cancel}
\usepackage{bbm}
\usepackage{comment}
\usepackage{algorithm}
\usepackage{algpseudocode}
\usepackage{enumitem}

\algrenewcommand\algorithmicrequire{\textbf{Input:}}
\algrenewcommand\algorithmicensure{\textbf{Output:}}
\usepackage{times}
\usepackage{subcaption}
\usepackage{wrapfig}
\usepackage{mathtools}
\usepackage{mathabx}
\usepackage{multirow} 
\usepackage{booktabs}

\title[Dimension-free Bounds for Covariance Estimation with Tensor-Train Structure]{Dimension-free Bounds for Covariance Estimation\\ with Tensor-Train Structure}

\coltauthor{%
 \Name{Artsiom Patarusau} \Email{apotarusov@hse.ru}\\
 \addr HSE University
 \AND
 \Name{Nikita Puchkin} \Email{npuchkin@hse.ru}\\
 \addr HSE University
 \AND
 \Name{Maxim Rakhuba} \Email{mrakhuba@hse.ru}\\
 \addr HSE University
 \AND
 \Name{Fedor Noskov} \Email{fnoskov@hse.ru}\\
 \addr HSE University%
}

\renewcommand{\le}{\leqslant}
\renewcommand{\leq}{\leqslant}
\renewcommand{\ge}{\geqslant}

\newcommand{\eps}{\varepsilon}

\newcommand{\rmd}{\mathrm{d}}
\newcommand{\F}{\mathrm{F}}
\newcommand{\diag}{\operatorname{diag}}
\newcommand{\Tr}{\operatorname{Tr}}

\newcommand{\rmvec}{\operatorname{vec}}
\newcommand{\rank}{\operatorname{rank}}

\newcommand{\E}{\mathbb E}
\newcommand{\event}{\boldsymbol{\mathcal{E}}}

\newcommand{\R}{\mathbb R}
\newcommand{\bbS}{\mathbb S}
\newcommand{\1}{\mathbbm 1}
\newcommand{\bbO}{\mathbb{O}}

\newcommand{\KL}{\mathcal{KL}}
\newcommand{\btheta}{\boldsymbol{\theta}}

\newcommand{\ttm}{\mathtt{m}}
\newcommand{\ttr}{\mathtt{r}}
\newcommand{\ttR}{\mathtt{R}}

\def\argmin{\operatornamewithlimits{argmin}}
\newcommand{\Img}{\operatorname{Im}}
\newcommand{\SVD}{\operatorname{SVD}}

\newcommand{\bb}{\mathbf{b}}

\newcommand{\bx}{\mathbf{x}}
\newcommand{\by}{\mathbf{y}}
\newcommand{\bX}{\mathbf{X}}
\newcommand{\bxi}{\boldsymbol{\xi}}
\newcommand{\bfeta}{\boldsymbol{\eta}}

\newcommand{\cT}{\mathcal{T}}
\newcommand{\cA}{\mathcal{A}}
\newcommand{\cY}{\mathcal{Y}}
\newcommand{\cE}{\mathcal{E}}
\newcommand{\bW}{\mathbf{W}}
\newcommand{\cN}{\mathcal{N}}

\newcommand{\cW}{\mathcal{W}}
\newcommand{\cP}{\mathcal{P}}
\newcommand{\cR}{\mathcal{R}}
\newcommand{\cH}{\mathcal{H}}
\newcommand{\cX}{\mathcal{X}}
\newcommand{\nS}{\Vert \Sigma \Vert}
\newcommand{\cS}{\mathcal{S}}
\newcommand{\tty}{\mathtt{y}}
\newcommand{\ttx}{\mathtt{x}}
\newcommand{\cG}{\mathcal{G}}
\newcommand{\cD}{\mathcal{D}}
\newcommand{\bbL}{\mathbb{L}}
\newcommand{\ttA}{\mathtt{A}}

\newcommand{\bv}{\mathbf{v}}

\def\red#1{{\color{red} #1}}

\newtheorem{Th}{Theorem}[section]
\newtheorem{Lem}[Th]{Lemma}

\newtheorem{Prop}[Th]{Proposition}

\newtheorem{As}[Th]{Assumption}

\let\oldmax\max
\renewcommand{\max}{\oldmax\limits}
\let\oldmin\min
\renewcommand{\min}{\oldmin\limits}

\begin{document}

\maketitle

\begin{abstract}
    We consider a problem of covariance estimation from a sample of i.i.d. high-dimensional random vectors. To avoid the curse of dimensionality, we impose an additional assumption on the structure of the covariance matrix $\Sigma$. To be more precise, we study the case when $\Sigma$ can be approximated by a sum of double Kronecker products of smaller matrices in a tensor train (TT) format. Our setup naturally extends widely known Kronecker sum and CANDECOMP/PARAFAC models but admits richer interaction across modes. We suggest an iterative polynomial time algorithm based on TT-SVD and higher-order orthogonal iteration (HOOI) adapted to Tucker‑2 hybrid structure. We derive non-asymptotic dimension-free bounds on the accuracy of covariance estimation taking into account hidden Kronecker product and tensor train structures. The efficiency of our approach is illustrated with numerical experiments.
\end{abstract}



\section{Introduction}

Given $\bX, \bX_1, \dots, \bX_n \in \R^d$ i.i.d. centered random vectors, we are interested in estimation of their covariance matrix $\Sigma = \E \bX \bX^\top \in \R^{d \times d}$. Despite its long history, this classical problem still gets considerable attention of statistical and machine learning communities. The reason is that in modern data mining tasks researchers often have to deal with high-dimensional observations. In such scenarios they cannot rely on classical estimates, for instance, sample covariance
\[
    \widehat\Sigma = \frac1n \sum\limits_{i = 1}^n \bX_i \bX_i^\top,
\]
suffering from the curse of dimensionality. To overcome this issue, statisticians impose additional assumptions on $\Sigma$ in order to exploit the data structure and reduce the total number of unknown parameters. Some recent methodological and theoretical advances in covariance estimation are related with Kronecker product models, which are particularly useful for analysis of multiway or tensor-valued data \citep{werner08, allen10, greenewald13, sun18, guggenberger23}. For example, motivated by multiple input multiple output (MIMO) wireless communications channels, \citet*{werner08} assumed that $\Sigma$ can be represented as a Kronecker product of two smaller matrices $\Phi \in \R^{p \times p}$ and $\Psi \in \R^{q \times q}$, such that $pq = d$:
\begin{equation}
    \label{eq:kronecker_product_model}
    \Sigma
    = \Phi \otimes \Psi =
    \begin{pmatrix}
        \varphi_{11} \Psi & \dots & \varphi_{1p} \Psi \\
        \vdots & \ddots & \vdots \\
        \varphi_{p1} \Psi & \dots & \varphi_{pp} \Psi 
    \end{pmatrix}.
\end{equation}
It is known that (see, for instance, the proof of Theorem 1 in \citep{vanloan93}) $\Sigma$ of form \eqref{eq:kronecker_product_model} can be reshaped into a rank-one matrix using an isometric rearrangement (or permutation) operator $\cP : \R^{pq \times pq} \rightarrow \R^{p^2 \times q^2}$ (see \cite[Definition 2.1]{puchkin24a}). Based on this fact, \citeauthor*{werner08} suggested to estimate $\cP(\Sigma)$ applying singular value decomposition to $\cP(\widehat \Sigma)$ and showed that this estimate is asymptotically efficient in the Gaussian case. They called this approach covariance matching. This idea was further developed by \citep{tsiligkaridis13, masak22, puchkin24a}, who considered the sum of Kronecker products model
\begin{equation}
    \label{eq:sum_of_kronecker_products_model}
    \Sigma = \sum\limits_{k = 1}^K \Phi_k \otimes \Psi_k,
\end{equation}
where $\Phi_1, \Psi_1, \dots, \Phi_K, \Psi_K$ are symmetric positive semidefinite matrices, such that $\Phi_j \in \R^{p \times p}$, $\Psi_j \in \R^{q \times q}$ for all $j \in \{1, \dots, K\}$ and $pq = d$.
They studied properties of the permuted regularized least squares (PRLS) estimates. In \citep{tsiligkaridis13, puchkin24a}, the authors regularized the loss function using the nuclear norm
\begin{equation}
    \label{eq:prls_estimate_nuclear}
    \widehat \Sigma^{\circ} = \cP^{-1}(\widetilde R),
    \quad \text{where} \quad
    \widetilde R \in \argmin\limits_{R \in \R^{p^2 \times q^2}} \left\{ \left\|R - \cP(\widehat \Sigma) \right\|_{\F}^2 + \lambda \|R\|_* \right\},
\end{equation}
while \cite{masak22} considered a rank-penalized estimate
\begin{equation}
    \label{eq:prls_estimate_rank}
    \widecheck \Sigma = \cP^{-1}(\widecheck R),
    \quad 
    \widecheck R \in \argmin\limits_{R \in \R^{p^2 \times q^2}} \left\|R - \cP(\widehat \Sigma) \right\|_{\F}^2 + \lambda \, \rank(R).
\end{equation}
Following the covariance matching approach of \cite{werner08}, both \eqref{eq:prls_estimate_nuclear} and \eqref{eq:prls_estimate_rank} reduce the problem of covariance estimation to recovering of a low-rank matrix $\cP(\widehat \Sigma)$ from noisy observations. We would like to note that the estimates $\widehat{\Sigma}^{\circ}$ and $\widecheck{\Sigma}$ admit explicit expressions based on the singular value decomposition of $\cP(\widehat \Sigma)$. For this reason, they can be computed in polynomial time.

In the present paper, we consider a covariance model combining Kronecker product and tensor train (TT) structure. To be more precise, we consider $\Sigma$ of the form
\begin{equation}
    \label{eq:kronecker_tt_model}
    \Sigma = \sum\limits_{j = 1}^J \sum\limits_{k = 1}^K U_j \otimes W_{jk} \otimes V_k,
\end{equation}
where $U_j \in \R^{p \times p}$, $W_{jk} \in \R^{q \times q}$, and $V_k \in \R^{r \times r}$ for any $j \in \{1, \dots, J\}$ and $k \in \{1, \dots, K\}$. The numbers $p$, $q$, and $r$ are assumed to be such that $pqr = d$. Let us note that \eqref{eq:kronecker_tt_model} naturally extends \eqref{eq:sum_of_kronecker_products_model} to the case of three-way data and coincides with it when $J = 1$
and $U_1 = 1$. The rationale for selecting our model is that the TT decomposition \citep{oseledets11} is recognized for its computational efficiency compared to the canonical polyadic (CP) decomposition, while providing a robust framework for representing higher-order tensors. For this reason, tensor trains have numerous applications in hyperspectral imaging \citep{solgi22}, high-dimensional Bayesian inference \citep{tao21}, dimension reduction of high-order Markov processes \citep{zhou2022optimal}, and LLM compression \citep{xu24, huang25}. We would like to note that the CANDECOMP/PARAFAC model
\begin{equation}
    \label{eq:candecomp}
    \Sigma = \sum\limits_{k = 1}^K \Phi_k \otimes \Psi_k \otimes \Omega_k,
\end{equation}
which has recently got considerable attention in the literature (see, for example, \citep{pouryazdian16, greenewald19, yu25} and the references therein), is a particular case of \eqref{eq:kronecker_tt_model} with $J = K$, $W_{jk} = \Psi_k \1(j = k)$, and $U_j = \Phi_j$. 
Following the covariance matching approach, we can reshape a matrix $\Sigma$ of the form \eqref{eq:kronecker_tt_model} into a third-order tensor with low canonical rank. Indeed, given a matrix $A \in \R^{pqr \times pqr}$, let us define a rearrangement operator $\cR : \R^{pqr \times pqr} \rightarrow \R^{p^2 \times q^2 \times r^2}$ componentwise: for any $1 \leq a \leq p^2$, $1 \leq b \leq q^2$, and $1 \leq c \leq r^2$
\begin{equation}
    \label{eq:rearrangement_operator}
    \cR(\Sigma)_{a, b, c} = \Sigma_{(\lceil a / p \rceil - 1) \cdot qr + (\lceil b/q \rceil - 1) \cdot r + \lceil c/r \rceil,((a - 1) \%p) \cdot qr + ((b - 1) \% q) \cdot r + (c - 1)\%r + 1},
\end{equation}
where $y \% x \in \{0, \ldots, x - 1\}$  stands for the residual of $y$ modulo $x$.
Then it is easy to check that
\begin{equation}
    \label{eq:kronecker_tt_rearranged}
    \cR(\Sigma) = \sum\limits_{j = 1}^J \sum\limits_{k = 1}^K \rmvec(U_j) \otimes \rmvec(W_{jk}) \otimes \rmvec(V_k),
\end{equation}
where, for any matrix $A$, $\rmvec(A)$ is a vector obtained by stacking the columns of $A$ together.
Unfortunately, a formal extension of the approach suggested by \cite{tsiligkaridis13} to the CANDECOMP/PARAFAC model will not result in a practical algorithm. The main obstacle is that approximation of the nuclear norm of a tensor is an NP-hard problem \cite{hillar2013most}. The statistical-computational gap was discussed in several papers including \citep{barak2016noisy, zhang2018tensor, han22b, luo22, luo2024tensor}. For this reason, when developing an algorithm for estimation of the covariance matrix \eqref{eq:kronecker_tt_model}, we must take into account both its computational and sample complexities. In the present paper, we extend the approach of \cite{zhang2018tensor} and suggest an iterative procedure similar to the higher-order orthogonal iteration (HOOI) with the notable distinction of utilizing the Tucker-2 representation of the tensor. Our algorithm successfully adapts to the structure \eqref{eq:kronecker_tt_model} but requires less time, than Tucker decomposition and HOOI.

While statisticians (see, for example, \citep{tsiligkaridis13, puchkin24a}) established rates of convergence of the PRLS estimate \eqref{eq:prls_estimate_nuclear}, the CANDECOMP/PARAFAC model \eqref{eq:candecomp} and the more general tensor train model \eqref{eq:kronecker_tt_model} remain underexplored. In Section \ref{section: Analusis of TT-SVD} (see \eqref{eq:tensor_tt_decomposition} below), we discuss that the tensor train model \eqref{eq:kronecker_tt_model} can be represented in a way, which is very similar to the low Tucker rank tensor model (see, for instance, \cite[Definition 2.1]{han22b}). The only difference is that \eqref{eq:tensor_tt_decomposition} includes two factors with orthogonal columns while in Tucker decomposition one has three such factors. For this reason, some bounds on the estimation accuracy of $\Sigma$ of the form \eqref{eq:kronecker_tt_model} with respect to the Frobenius norm follow from the results on tensor estimation \cite{zhang2018tensor, han2022optimal, kumar25}, scalar-on-tensor regression \cite{khavari21, wang25}, and tensor-on-tensor regression \cite{raskutti19, luo2024tensor} with constraints on Tucker ranks. However, these bounds are dimension dependent, while many recent results in covariance estimation establish dimension-free bounds (see, for instance, \cite{koltchinskii17, bunea2015sample, abdalla22, zhivotovskiy24, puchkin24a, puchkin2025sharper}). To our knowledge, the existing dimension-free results on tensor estimation only cover the case of simple rank-one tensors \citep{vershynin2020concentrationinequalitiesrandomtensors, zhivotovskiy24, alghattas25, chen25}.
In the present paper, we derive high-probability dimension-free bounds on the accuracy of estimation of third-order tensors with low TT-ranks and of the covariance matrices, which can be well approximated by \eqref{eq:kronecker_tt_model}.

\paragraph{Contribution.} Our main contribution is a comprehensive non-asymptotic analysis of this estimation procedure. We first derive a general deterministic perturbation bound for our TT-SVD-like algorithm, which may be of independent interest. We then leverage this result to establish a high-probability error bound for our covariance estimator. The final bound clearly decomposes the error into a bias term, related to how well the true $\Sigma$ can be approximated by our model, and a variance term. This variance term scales gracefully with the sample size n, the TT-ranks $(J,K)$, and data-dependent effective dimensions that capture the intrinsic complexity of the covariance structure. To the best of authors knowledge, this is the first work to prove dimension-free estimator for the broad class of tensors, which is parametrized only by ranks $J, K$ (\cite{zhivotovskiy24, alghattas25} considered rank-1 case of CANDECOMP-PARAFAC).

\paragraph{Paper structure.} The rest of the paper is organized as follows. In Section \ref{section: Analusis of TT-SVD}, we present our algorithm and main theoretical guarantees. We illustrate efficiency of the suggested approach with numerical experiments in Section \ref{section: experiments}. All proofs are deferred to Appendix.

\paragraph{Notation.} Given a matrix $M \in \R^{d_1 \times d_2}$, we define its vectorization as
\begin{align*}
    \rmvec(M)_{(a - 1) \cdot d_2 + b} = M_{a,b}, \quad a \le d_1, b \le d_2.
\end{align*}
For a tensor $\cT$ of order $k$ with dimensions $d_1, \ldots, d_k$, we define a multiplication $\times_i$ on mode $i$ by a matrix $M \in \R^{d' \times d_i}$ as follows:
\begin{align*}
    (\cT \times_i M)_{a_1 a_2 \ldots a_i a_{i + 1} \ldots a_k} = \sum_{a_i' = 1}^{d_i} \cT_{a_1 a_2 \ldots a_{i - 1} a_i' a_{i + 1} \ldots a_k} M_{a_i a_i'},
\end{align*}
where $a_j, j \neq i,$  takes values in $\{1, \ldots, d_j\}$ and $a_i$ takes values in $\{1, \ldots, d'\}$.

It will be convenient to assume that random vectors $\bX, \bX_1, \ldots, \bX_n$ lie in a tensor product space $\R^p \otimes \R^q \otimes \R^q$, so $\Sigma = \E \bX \bX^\top$ belongs to the space of SDP Hermitian operators $\cH_+(\R^p \otimes \R^q \otimes \R^r)$ from $\R^p \otimes \R^q \otimes \R^q$  to itself. Then, we will define partial traces of $\Sigma$ as follows. Given linear spaces $L_1, L_2$ and linear operators $X: L_1 \to L_1, Y : L_2 \to  L_2$, we define the partial trace $\Tr_{L_i}$, $i = 1, 2$, w.r.t. $L_i$ as follows:
\begin{align*}
    \Tr_{L_1}(X \otimes Y) = \Tr(X) \cdot Y, \quad \Tr_{L_2}(X \otimes Y) = X \cdot \Tr(Y).
\end{align*}
We extend $\Tr_{L_i}(\cdot)$ to all operators from $L_1 \otimes L_2 \to L_1 \otimes L_2$ by linearity. In our case, for operators from $\cH_+(\R^{p} \otimes \R^q \otimes \R^r)$, we define $\Tr_1(\cdot)$ as a partial trace w.r.t. $\R^p$, $\Tr_2(\cdot)$ as a partial trace w.r.t. $\R^q$ and $\Tr_3(\cdot)$ as a partial trace w.r.t. $\R^r$.
Partial traces will play in important role in our theoretical analysis. We define
\begin{align*}
    \ttr_1(\Sigma) & = \max \left \{ \frac{\Vert \Tr_1(\Sigma) \Vert}{\Vert \Sigma \Vert}, \frac{\Vert \Tr_{1,2}(\Sigma) \Vert}{\Vert \Tr_2(\Sigma) \Vert} \right \}, \qquad
    \ttr_2(\Sigma)  = \max \left \{ \frac{\Vert \Tr_2(\Sigma) \Vert}{\Vert \Sigma \Vert}, \frac{\Vert \Tr_{2,3}(\Sigma) \Vert}{\Vert \Tr_3(\Sigma) \Vert} \right \}, \\
    \ttr_3(\Sigma) & = \max \left \{ \frac{\Vert \Tr_3(\Sigma) \Vert}{\Vert \Sigma \Vert}, \frac{\Vert \Tr_{1,3} (\Sigma) \Vert}{\Vert \Tr_1(\Sigma) \Vert}, \frac{\Vert \Tr_{1,2,3} (\Sigma)\Vert}{\Vert \Tr_{1, 2}(\Sigma) \Vert} \right \},
\end{align*}
where $\Tr_{i_1 i_2 \ldots i_k}$ stands for the composition of the traces $\Tr_{i_1}, \Tr_{i_2}, \ldots, \Tr_{i_k}$.
Quantities $\ttr_1(\Sigma), \ttr_2(\Sigma), \ttr_3(\Sigma)$ play the role of effective dimensions. From \citep[display (23)]{Rastegin_2012}, we know that $\ttr_1(\Sigma) \le p, \ttr_2(\Sigma) \le q, \ttr_3(\Sigma) \le r$. We define them as maxima over ratios of some partial traces to ensure that for any non-empty set $S \subset \{1, 2, 3\}$ we have
\begin{align*}
    \frac{\Vert \Tr_S(\Sigma) \Vert}{\Vert \Sigma \Vert} \le \prod_{s \in S} \ttr_s(\Sigma).
\end{align*}
For a tensor $\cT \in \R^{p^2 \times q^2 \times r^2}$, we introduce the unfolding operator with respect to the first mode as
\[
    \ttm_1(\cT)_{x,y}  = \cT_{x, \lceil y / r^2 \rceil ,(y  - 1)\% r^2 + 1}.
\]
Similarly, the unfolding operators with respect to the second and the third modes are define as follows:
\[
    \ttm_2(\cT)_{x, y}  = \cT_{(y - 1)\% p^2 + 1, x, \lceil y / p^2 \rceil}, 
    \quad
    \ttm_3(\cT)_{x, y}  = \cT_{\lceil y / q^2 \rceil, (y- 1) \% q^2 + 1, x}.
\]

We denote the output of SVD algorithm with hard thresholding via rank $J$ as $SVD_J$. We denote matrices with orthonormal columns of size $\R^{d \times r}$ by $\bbO_{d, r}$.  In what follows, $[m]$ stands for the set of integers from $1$ to $m$.

\section{Main results}
\label{section: Analusis of TT-SVD}

Let us return to the estimation of the covariance matrix $\Sigma$ of the form \eqref{eq:kronecker_tt_model}. As discussed in the introduction, we can reshape $\Sigma$ into a third-order tensor $\cR(\Sigma)$ using the rearrangement operator \eqref{eq:rearrangement_operator}:
\[
    \cR(\Sigma) = \sum\limits_{j = 1}^J \sum\limits_{k = 1}^K \rmvec(U_j) \otimes \rmvec(W_{jk}) \otimes \rmvec(V_k) \in \R^{p^2 \times q^2 \times r^2},
\]
where vectors $\rmvec(U_j)$ are assumed to be linearly independent, as well as vectors $\rmvec(V_k)$. 
Stacking together vectors $\rmvec(U_j)$, $j = 1, \dots, J$ into a matrix $U \in \R^{p^2 \times J}$, vectors $\rmvec(V_k)$, $k = 1, \dots, K$ into a matrix $V \in \R^{r^2 \times K}$ and matrices $W_{jk}$, $j = 1, \dots, J$, $k = 1, \dots, K$ into a three-dimensional tensor $\cW \in \R^{J \times q^2 \times K}$, we can rewrite the above decomposition in the following compact form:
\begin{align}
    \label{eq:tensor_tt_decomposition}
    \cR(\Sigma) = \cW \times_3 V \times_1 U. 
\end{align}
Note that this decomposition is not unique. In particular, multiplying $U$ by an invertible matrix $Q_U \in \R_{J, J}$ from the right and $\cW$ by $Q_{U}^{-1}$ from the first mode does not change the right-hand side of \eqref{eq:tensor_tt_decomposition}. The same true for the factor $V$. Hence, one can assume that the columns of $U$ and $V$ are orthonormal, i.e. $U \in \bbO_{p^2, J}$ and $V \in \bbO_{r^2, K}$. In what follows, we always assume that this is the case. For brevity, we set $d_1 = p^2$, $d_2 = q^2$, and $d_3 = r^2$.

We extend the model~\eqref{eq:kronecker_tt_model} to the case when $\Sigma$ can be approximated by decomposition~\eqref{eq:kronecker_tt_model} up to some error. Then, it is naturally to consider the best $(J, K)$-TT-rank approximation of $\cR(\Sigma)$, which we denote by $\cT^*$.  We denote the misspecification shift  $\cR(\Sigma) - \cT^*$ by $\overline{\cE}$. To approximate $\Sigma$, we aim to recover its structured part $\cT^*$ from the noisy tensor $\cY = \cR(\widehat \Sigma)$, which can be represented as
\begin{align}
\label{eq: tensor noisy model}
    \cY = \cT^* + \cE \in \R^{d_1 \times d_2 \times d_3},  
\end{align}
where the error tensor $\cE$ consists of the approximation part $\overline{\cE}$ and the noise part $\widehat{\cE} = \cR(\widehat \Sigma) - \cR(\Sigma)$.

Since $\cT^*$ has TT-ranks $(J, K)$, it can be decomposed as $\cT^* = \cW^* \times_3 V^* \times_1 U^*$, where $U^* \in \bbO_{p^2, J}$, $V^* \in \bbO_{r^2, K}$ and $\cW^* \in \R^{J \times q^2 \times K}$. This decomposition suggests the following natural algorithm for estimating $\cT^*$ from $\cY$. Using truncated SVD, one estimates the image of $U^*$ which coincides with $\Img \ttm_1(\cT^*)$, then estimates the image of $V^*$ which coincides with $\Img \ttm_3(\cT^*)$, and then project $\cY$ onto the estimated spaces. However, this estimation is not straightforward, and one should apply truncated SVD iteratively to reach reasonable accuracy.  In Section~\ref{section: experiments}, we conduct numerical experiments illustrating that additional iterations indeed improve the estimation. We summarized the resulting procedure as Algorithm \ref{algo: order 3 TT-SVD}. We refer to it as the Hardtth algorithm where the abbreviation HardTTh stands for \textbf{Hard} \textbf{T}ensor \textbf{T}rain \textbf{Th}resholding.

\begin{algorithm}[htbp]
\caption{HardTTh}
\label{algo: order 3 TT-SVD}
\begin{algorithmic}
\Require Tensor $\cY \in \mathbb{R}^{d_1 \times d_2 \times d_3}$, TT-ranks $(J, K)$, number of steps $T$
\Ensure TT-approximation $\widehat{\cT} = \widehat{\cW} \times_3 \widehat{V} \times_1 \widehat{U}$, where $\widehat{U} \in \bbO_{d_1, J}$, $\widehat{V} \in \bbO_{d_2, K}$, $\widehat{\cW} \in \R^{J \times d_2 \times K}$
\State Compute truncated SVD of $\ttm_1(\cY)$ keeping the first $J$ singular values:
\Statex \hspace{\algorithmicindent}$\widehat{U}_0, \Sigma_{0,1}, \widetilde{U}_0 \gets \SVD_J(\ttm_1(\cY))$
\State Compute truncated SVD of $\ttm_3(\cY \times_1 \widehat{U}_0^\top)$ keeping the first $K$ singular values:
\Statex \hspace{\algorithmicindent}$\widehat{V}_0, \Sigma_{0,2}, \widetilde{V}_0 \gets \SVD_K(\ttm_3(\cY \times_1 \widehat{U}_0^\top))$
\State \textbf{for} $t \gets 1$ \textbf{to} $T$ \textbf{do}
\Statex \hspace{\algorithmicindent}$\widehat{U}_t, \Sigma_{t,1}, \widetilde{U}_t \gets \SVD_J(\ttm_1(\cY \times_3 \widehat{V}_{t-1}^\top))$
\Statex \hspace{\algorithmicindent}$\widehat{V}_t, \Sigma_{t,2}, \widetilde{V}_t \gets \SVD_K(\ttm_3(\cY \times_1 \widehat{U}_t^\top))$
\State \textbf{end for}
\State Set $\widehat{U} \gets \widehat{U}_T$, $\widehat{V} \gets \widehat{V}_T$
\State Set $\widehat{\cW} \gets \cY \times_3 \widehat{V}^\top \times_1 \widehat{U}^\top$
\end{algorithmic}
\end{algorithm}

Notice that computational complexity of Algorithm \ref{algo: order 3 TT-SVD} is determined by the complexity of truncated SVD applied to the matricizations. The truncated $SVD_J$ at the first step of HardTTh takes $O( d_1 d_2 d_3 \cdot \min\{d_1, d_2 d_3\})$. Other steps require either $O(Jd_3 d_2 \cdot \min\{d_3, Jd_2\} + Jd_1 d_2 d_3)$ or $O(K d_1 d_2 \min\{d_1, Kd_2\} + Kd_1 d_2 d_3)$ flops, so the overall complexity of the algorithm is
\begin{align*}
    & O((J + K) T d_1 d_2 d_3 + T  K d_1 d_2 \cdot \min\{d_1, Kd_2\} + T Jd_3 d_2 \cdot \min\{d_3, Jd_2\} \\
    & \qquad + d_1 d_2 d_3 \cdot \min\{d_1, d_2 d_3\}).
\end{align*}
If the misspecification is not too large, the number $T$ of iterations can be taken logarithmical in the ambient dimensions, see discussion below after Theorem~\ref{theorem: Sigma estimator performance}.

In practice, randomized truncated SVD could be used~\citep{halko2011finding} or other approximate algorithms~\citep{baglama2005augmented}.

Given the output $\widehat{\cT}$ of Algorithm~\ref{algo: order 3 TT-SVD} applied to $\cY = \cR(\widehat{\Sigma})$, define the estimator $\widetilde{\Sigma}$ of $\Sigma$ as $\widetilde{\Sigma} = \cR^{-1}(\widehat{\cT})$. To analyze rates of convergence for this estimator, we impose some assumption on the distribution of $\bX_i$. While $\widetilde{\Sigma}$ is not guaranteed to be symmetric and positive semidefinite, one can project $\widetilde{\Sigma}$ on the cone of PSD matrices, obtaining an even better estimator of $\Sigma$ in the Frobenius norm.

\begin{As}
    \label{as:orlicz}
    There exists $\omega > 0$, such that the standardized random vector $\Sigma^{-1/2} \bX$ satisfies the inequality
    \begin{equation}
        \label{eq:orlicz_norm_inequality}
        \log \E \exp \left\{ (\Sigma^{-1/2} \bX)^\top V (\Sigma^{-1/2} \bX) - \Tr(V) \right\} \leq \omega^2 \Vert V\Vert_{\F}^2 
    \end{equation}
    for all $V \in \R^{d \times d}$, such that $\Vert V\Vert_{\F} \leq 1 /  \omega$.
\end{As}

In~\citep{puchkin2025sharper}, the authors showed that Assumption~\ref{as:orlicz} holds for a large class of distribution. Indeed, Assumption~\ref{as:orlicz} is a weaker version of the Hanson--Wright inequality. In particular, if the Hanson--Wright inequality is fulfilled for $\Sigma^{-1/2} \bX$, then $\bX$ satisfies Assumption~\ref{as:orlicz}. Therefore, Assumption~\ref{as:orlicz} can be used when $\Sigma^{-1/2}\bX$ is multivariate standard Gaussian, consists of i.i.d. sub-Gaussian random variables, satisfies the logarithmic Sobolev inequality or the convex concentration property~\citep{adamczak_note_2015}. We remark that Assumption~\ref{as:orlicz} is more general than those typically considered in the existing tensor estimation literature, which  focuses on tensors with independent sub-Gaussian noisy entries \citep{zhang2018tensor, han2022optimal, luo2024tensor, xia2022inference, zhou2022optimal, tang2025revisit}.

Under Assumption~\ref{as:orlicz}, we establish the following theorem. We give its proof in Appendix~\ref{section: proof of main theorem}. The proof sketch is given in Appendix~\ref{section: proof sketch}.  
\begin{Th}
\label{theorem: Sigma estimator performance}
Fix $\delta \in (0, 1)$. Grant Assumption~\ref{as:orlicz}. Suppose that singular values $\sigma_J(\ttm_1(\cR(\Sigma))$, $\sigma_K(\ttm_3(\cR(\Sigma))$ satisfy
\begin{align*}
    \sigma_J(\ttm_1(\cR(\Sigma))) & \ge 25 \Vert \ttm_1(\overline{\cE}) \Vert + 768 \omega \Vert \Sigma \Vert \sqrt{\frac{\ttr_1^2(\Sigma) + \ttr_2^2(\Sigma) \ttr_3^2(\Sigma) + \log(6/\delta)}{n}}, \\
    \sigma_K(\ttm_3(\cR(\Sigma))) & \ge 25 \Vert \ttm_3(\overline{\cE}) \Vert + 768 \omega \Vert \Sigma \Vert \sqrt{  \frac{J \ttr_1^2(\Sigma) +  J \ttr_2^2(\Sigma) + \ttr_3^2(\Sigma) + \log (48/\delta)}{n}}.
\end{align*}
Then, we have
    \begin{align*}
        \Vert \widetilde{\Sigma} - \Sigma \Vert_{\F} \le \overline{\bb} + 96 \omega \Vert \Sigma \Vert \sqrt{\frac{J \ttr_1^2(\Sigma) + JK \ttr^2_2(\Sigma) + K \ttr^2_3(\Sigma) + \log(48/\delta)}{n}} + \widetilde{\diamondsuit}_2 + \widetilde{r}_T
    \end{align*}
    with probability at least $1 - \delta$, provided $n \ge \ttR_\delta$, where 
    \begin{align*}
        \overline{\bb} & = \Vert \overline{\cE} \Vert_{\F} + 5 \sqrt{J} \Vert \ttm_1(\overline{\cE}) \Vert + 5 \sqrt{K} \Vert \ttm_3(\overline{\cE}) \Vert,
    \end{align*}
    and $\ttR_\delta$ and remainder terms $\widetilde{\diamondsuit}_2, \widetilde{r}_T$ are defined in Table~\ref{tab: ancillary variables}.
\end{Th}

\begin{table}[!ht]
    \centering
    \begin{tabular}{|c|c|}
    \hline 
    Variable & Expression \\
    \hline
         $\widetilde{\alpha}_U$ & $ \Vert \ttm_1(\overline{\cE} \times_3 (V^*)^\top) \Vert + 32 \omega \Vert \Sigma \Vert \sqrt{\frac{\ttr_1^2(\Sigma) + K \ttr_2^2(\Sigma) + \log(48/\delta)}{n}}$ \\
    $\widetilde{\beta}_U$ & $ \sup_{\substack{V \in \R^{d_2 \times K} \\ \Vert V \Vert \le 1}} \Vert \ttm_1(\overline{\cE} \times_3 V^\top) \Vert + 32 \omega \Vert \Sigma \Vert \sqrt{\frac{\ttr_1^2(\Sigma) + K \ttr_2^2(\Sigma) + K \ttr_3^2(\Sigma) + \log(48/\delta)}{n}} $
    \\
    $\widetilde{\alpha}_V$ & $ \Vert \ttm_3(\overline{\cE} \times_1 (U^*)^\top) \Vert + 32 \omega \Vert \Sigma \Vert \sqrt{\frac{\ttr_3^2(\Sigma) + J \ttr_2^2(\Sigma) + \log(48/\delta)}{n}}$ \\ 
    $\widetilde{\beta}_V$ & $ \sup_{\substack{U \in \R^{d_1 \times J} \\ \Vert U \Vert \le 1}} \Vert \ttm_3(\overline{\cE} \times_1 U^\top) \Vert + 32 \omega \Vert \Sigma \Vert \sqrt{\frac{\ttr_2^2(\Sigma) + J \ttr_1^2(\Sigma) + J \ttr_3^2(\Sigma) + \log(48/\delta)}{n}}$ \\
    $\widetilde{\diamondsuit}_2$ & $96 \left ( \frac{ \sqrt{K} \widetilde{\beta}_V \widetilde{\alpha}_U}{\sigma_J(\ttm_1(\cR(\Sigma)))} + \frac{\sqrt{J} \widetilde{\beta}_U \widetilde{\alpha}_V}{\sigma_K(\ttm_3(\cR(\Sigma)))}\right )$ \\
    & \\
        $\widetilde{r}_T$ & $ (\sqrt{J} + \sqrt{K}) \cdot \left ( \frac{200 \widetilde{\beta}_V \widetilde{\beta}_U}{\sigma_J(\ttm_1(\cR(\Sigma))) \sigma_K(\ttm_3(\cR(\Sigma)))}\right )^T \times$ \\
    & $ \qquad \quad  \times \left ( \Vert \ttm_1(\overline{\cE}) \Vert + 32 \omega \sqrt{\frac{\ttr_1^2(\Sigma) + \ttr_2^2(\Sigma) \ttr_3^2(\Sigma) + \log(6/\delta)}{n}} \right )$ \\
    & \\ 
    $\ttR_{\delta}$ & $J\ttr_1^2(\Sigma) + JK\ttr_2^2(\Sigma) + K\ttr_3^2(\Sigma) + \ttr_2^2(\Sigma) \ttr_3^2(\Sigma) + \log(48/\delta) $ \\
    \hline
    \end{tabular}
    \caption{List of ancillary variables}
    \label{tab: ancillary variables}
\end{table}

The upper bound on $\Vert \widetilde{\Sigma} - \Sigma \Vert_{\F}$ provided by the above theorem can be decomposed into the bias term $\overline{\bb}$ due to model misspecification, the leading variance term
\begin{align*}
    \widehat{\bv} =  96 \omega \Vert \Sigma \Vert \sqrt{\frac{J \ttr_1^2(\Sigma) + JK \ttr^2_2(\Sigma) + K \ttr^2_3(\Sigma) + \log(48/\delta)}{n}},
\end{align*}
and remainder terms $\widetilde{\diamondsuit}_2, \widetilde{r}_T$. Note that after $T = O(\log (JK\ttr_2(\Sigma)))$ iterations, the variance part 
\begin{align*}
    \widetilde{r}_T^v & = (\sqrt{J} + \sqrt{K}) \left ( \frac{200 \widetilde{\beta}_V \widetilde{\beta}_U}{\sigma_J(\ttm_1(\cR(\Sigma))) \sigma_K(\ttm_3(\cR(\Sigma)))}\right )^T
    \cdot 32 \omega \sqrt{\frac{\ttr_1^2(\Sigma) + \ttr_2^2(\Sigma) \ttr_3^2(\Sigma) + \log(6/\delta)}{n}},
\end{align*}
of $\widetilde{r}_T$
will be dominated by $\widehat{\bv}$.

Compared to the known results in the literature, Theorem~\ref{theorem: Sigma estimator performance} has several advantages. First, it provides dimension-free bounds based on the effective dimensions $\ttr_i(\Sigma) \le d_i$ instead of bounds involving ambient dimensions $d_1, d_2, d_3$ as in vast of literature on high-dimensional tensor estimation (cf.~\citep{zhang2018tensor,qin2025scalable,han2022optimal, tang2025revisit,luo2024tensor}). Moreover, in the case when $\cE$ consists of i.i.d. Gaussian random entries $\cN(0, \sigma^2)$, then, for any estimator $\widehat{\cX}$ based on $\cY = \cX^* + \cE$, there exists $\cX^* = \cW^* \times_1 U^* \times_3 V^*$, $U^* \in \bbO_{d_1, J}$, $V^* \in \bbO_{d_3, K}$, such that
\begin{align*}
    \E \Vert \widehat{\cX} - \cX^* \Vert_{\F}^2 \ge c \cdot \sigma^2 (J d_1 + JK d_2 + K d_3)
\end{align*}
holds for some small enough constant $c$ (see \citep[Theorem~4.2]{zhou2022optimal}), so Theorem~\ref{theorem: Sigma estimator performance} replaces ambient dimensions $d_1 = p^2, d_2 = q^2, d_3 = r^2$ in the minimax optimal rates of tensor estimation with their effective counterparts. Unfortunately, there is strong evidence that, for such convergence rates and any polynomial-time estimator, the spectral gap assumption is unavoidable~\citep{barak2016noisy,hopkins2015tensor,zhang2018tensor,luo2024tensor,diakonikolas2023statistical}.

Second, we point out the following. Set $\ttr(\Sigma) = \Tr(\Sigma) / \Vert \Sigma \Vert$. It is known that, under some assumptions, the sample covariance matrix $\widehat{\Sigma}$ satisfies concentration inequalities
\begin{align*}
    \Vert \widehat{\Sigma} - \Sigma \Vert & \lesssim  \Vert \Sigma \Vert \sqrt{\frac{\ttr(\Sigma) + \log(1/\delta)}{n}}, \qquad 
    \Vert \widehat{\Sigma} - \Sigma \Vert_{\F} \lesssim  \Vert \Sigma \Vert \sqrt{\frac{\ttr^2(\Sigma) + \log(1/\delta)}{n}}
\end{align*}
with probability at least $1 -\delta$  (see~\citep{zhivotovskiy24, bunea2015sample, hsu2012random, puchkin2025sharper}), where $\lesssim$ hides some distribution-dependent constant. Hence, our effective dimensions $\ttr_i(\Sigma)$ naturally extends the effective dimension $\ttr(\Sigma)$  of sample covariance concentration in the unstructured case. Third, while~\citet{puchkin24a} prove dimension-free bounds for the model~\eqref{eq:sum_of_kronecker_products_model} and the estimator $\widehat{\Sigma}^\circ = \cP^{-1}(\widetilde{R})$ defined by~\eqref{eq:prls_estimate_nuclear}, they do not analyze the misspecification case and bound the variance term with probability at least $1 - \delta$ as follows:
\begin{align*}
   \Vert \widehat{\Sigma}^\circ - \Sigma \Vert_{\F}  \lesssim \sqrt{K} \omega \sum_{k = 1}^K \Vert \Phi_k \Vert \Vert \Psi_k \Vert \sqrt{\frac{\max_k \ttr^2(\Psi_k) + \max_k \ttr^2(\Phi_k) + \log(1/\delta)}{n}},
\end{align*}
\textcolor{blue}{so they have} rough variance proxy factor $\sum_{k = 1}^K \Vert \Phi_k \Vert \Vert \Psi_k \Vert$ instead of $\Vert \Sigma \Vert = \Vert \sum_{k = 1}^K \Phi_k \otimes \Psi_k \Vert$. We improve their analysis to establish bounds on the variance involving variance proxy factor $\Vert \Sigma \Vert$ which seems to be tight. 

To highlight the advances of Theorem~\ref{theorem: Sigma estimator performance}, let us discuss how effective dimensions could be small compared to the ambient dimensions. In Appendix~\ref{section: proof of effective dimensions bounds}, we prove the following proposition.
\begin{Prop}
\label{proposition: bounds on effective dimensions}
Suppose that a covariance matrix $\Sigma \in \cH_+(\R^p \otimes \R^q \otimes \R^r)$ can be represented in the form
\begin{align*}
    \Sigma = \sum_{j = 1}^J \sum_{k = 1}^K U_j \otimes W_{jk} \otimes V_k
\end{align*}
for some symmetric positive semidefinite matrices $U_j, W_{jk}, V_k$. Then, we have
\begin{align*}
    \ttr_1(\Sigma) \le J \cdot \max_j \ttr(U_j), \quad \ttr_2(\Sigma) \le JK \cdot \max_{jk} \ttr(W_{jk}), \quad \ttr_3(\Sigma) \le K \cdot \max_k \ttr(V_k).
\end{align*}
\end{Prop}
For example, Proposition~\ref{proposition: bounds on effective dimensions} implies that if the spectra of matrices $U_j, W_{jk}$ and $V_k$ decay quadratically, i.e. if $\max_{jk} \{\sigma_i(U_j)/\Vert U_j \Vert, \sigma_i(W_{jk}) / \Vert W_{jk} \Vert, \sigma_i(V_k) / \Vert V_k \Vert \} \le C_{\sigma} i^{-2}$, then $\ttr_1(\Sigma) \le C_\sigma \pi^2/6 \cdot J, \ttr_2(\Sigma) \le C_{\sigma} \pi^2/6 \cdot JK$ and $\ttr_3(\Sigma) \le C_\sigma \pi^2/6 \cdot K$.

To further illustrate the difference between algebraic and effective ranks and highlight the importance of dimension-free bounds, we conducted the following toy experiment. We trained a 3-layer model with ReLU activation and 5070 parameters on MNIST for 10 epochs and computed its Fisher matrix $I \in \R^{5070 \times 5070}$ over the entire training dataset, which consists of 60000 objects. While the dimensions  $p$, $q$, and $r$ were set to $13$, $15$, and $26$, respectively, the corresponding effective ranks were significantly smaller: $\ttr_1(I) = 0.67$, $\ttr_2(I) = 5$ and $\ttr_3(I) = 20$.

The main drawback of Theorem~\ref{theorem: Sigma estimator performance} is the requirements $\sigma_J(\ttm_1(\cR(\Sigma))) \gtrsim\Vert \Sigma \Vert \sqrt{\ttr_2^2(\Sigma) \ttr_3^2(\Sigma) / n}$ and $n \gtrsim \ttr_2^2(\Sigma) \ttr_3^2(\Sigma)$. Indeed, the theory of tensor estimation by SVD-based algorithms developed in~\citep{zhang2018tensor, tang2025revisit} suggests that the minimax error can be achieved under condition
\begin{align}
\label{eq: zhang conditions}
    \sigma_J(\ttm_1(\cR(\Sigma))) \gtrsim \Vert \Sigma \Vert /n^{1/2} \cdot \left ( d_2 d_3 \right )^{3/8},
\end{align}
and there is strong evidence that the power $3/8$ in the above inequality can not be taken smaller for any polynomial-time algorithm~\citep{barak2016noisy,hopkins2015tensor,zhang2018tensor,luo2024tensor,diakonikolas2023statistical}. However, minimax bounds under conditions of the type~\eqref{eq: zhang conditions} were established  when entries of $\widehat{\cE}$ are i.i.d. Roughly speaking, the estimation error of the singular subspaces corresponds to the impact of the term $\ttm_1(\cE) \ttm_1(\cE)^\top$ in the decomposition
\begin{align*}
    \ttm_1(\cY) \ttm_1(\cY)^\top = \ttm_1(\cT^*) \ttm_1(\cT^*)^\top + \ttm_1(\cT^*) \ttm_1(\cE)^\top + \ttm_1(\cE) \ttm_1(\cT^*)^\top + \ttm_1(\cE) \ttm_1(\cE)^\top
\end{align*}
on the perturbation of eigenspace of $\ttm_1(\cT^*)\ttm_1(\cT^*)^\top$, see~\citep{cai2018rate}. When entries of $\widehat{\cE}$ are i.i.d., we have $\E \ttm_1(\widehat{\cE}) \ttm_1(\widehat{\cE})^\top = \alpha I_{d_1}$ for some scalar $\alpha$, so the error of singular subspaces estimation is determined by deviations of $\ttm_1(\widehat{\cE})^\top \ttm_1(\widehat{\cE})^\top$ from its mean, which can be controlled under conditions like~\eqref{eq: zhang conditions}. This is clearly not the case of our setup, so Algorithm~\ref{algo: order 3 TT-SVD} requires debiasing before applying SVD, which needs extra assumptions on the distribution of $\bX_i$ and is left for future work. 

Comparing Theorem~\ref{theorem: Sigma estimator performance} with results of~\citet{zhang2018tensor}, one can note that, in their paper, upper bounds on the tensor estimation error do not involve second-order terms like $\widetilde{\diamondsuit}_2$. The reason is that their work imposes an assumption $\max\{d_1, d_2, d_3\} \le C \min\{d_1, d_2, d_3\}$ for some absolute constant $C$. Translated to our setup, it means that, assuming $\max_i \ttr_i(\Sigma) \le C \min_{i} \ttr_i(\Sigma)$, the term $\widetilde{\diamondsuit}_2$ is dominated by the leading variance term $\widehat{\bv}$, which is exactly the case.

Finally, we briefly comment on the choice of $J$ and $K$. If $\Sigma$ can be represented by~\eqref{eq:kronecker_tt_model} for some $J, K$, such that
\begin{align*}
    \sigma_J(\ttm_1(\cR(\Sigma)) & \ge C \omega \Vert \Sigma \Vert \sqrt{\frac{\ttr_1^2(\Sigma) + \ttr_2^2(\Sigma) \ttr_3^2(\Sigma) + \log(6/\delta)}{n}}, \\
    \sigma_K(\ttm_3(\cR(\Sigma)) & \ge C \omega \Vert \Sigma \Vert \sqrt{\frac{J \ttr_2^2(\Sigma) + J \ttr_2^2(\Sigma) + \ttr_3^2(\Sigma) + \log(48/\delta)}{n}}
\end{align*}
for some large enough absolute constant $C$, and \textcolor{blue}{the following bounds hold}
\begin{align}
    & \Vert \Sigma \Vert/2 \le \Vert \widehat{\Sigma} \Vert \le 3\Vert \Sigma \Vert/2, \nonumber \\
    & \Vert \Tr_S(\widehat{\Sigma}) - \Tr_S(\Sigma) \Vert \le \frac{1}{2}\Vert \Tr_{S}(\Sigma) \Vert \text{ for all non-empty } S \subseteq [3] \label{eq: partial traces concentration}
\end{align}
with probability at least $1 - \delta/6$, then one can define estimators $\widehat{J}, \widehat{K}$ of $J, K$ as
\begin{align}
    \widehat{J} & = \max \left \{ J' \mid \sigma_{J'}(\ttm_1(\cR(\widehat{\Sigma})) \ge C' \omega \Vert \widehat{\Sigma} \Vert \sqrt{\frac{\ttr_1^2(\widehat{\Sigma}) + \ttr_2^2(\widehat{\Sigma}) \ttr_3^2(\widehat{\Sigma}) + \log (6/\delta)}{n}} \right \}, \label{eq: definition of J estimator} \\
    \widehat{K} & = \max \left \{ K' \mid \sigma_{K'}(\ttm_3(\cR(\widehat \Sigma)) \ge C' \omega \Vert \widehat{\Sigma} \Vert \sqrt{\frac{\widehat{J} \ttr_1^2(\widehat{\Sigma}) + \widehat{J} \ttr_2^2(\widehat{\Sigma}) + \ttr_3^2(\widehat \Sigma) + \log(48/\delta)}{n}} \right \}, \nonumber
\end{align}
where $C'$ is some other absolute constant and $\omega$ is assumed to be known. For example, one can compute $\omega$ explicitly when $\bX_i$ are linear transform of Gaussian random variables. For such $\widehat{J}$, we will have
\begin{align*}
    \sigma_{\widehat{J}}(\ttm_1(\cR(\Sigma))) > 768 \omega \Vert \Sigma \Vert \sqrt{\frac{\ttr_1^2(\Sigma) + \ttr_2^2(\Sigma) \ttr_3^2(\Sigma) + \log(6/\delta)}{n}} \ge   \Vert \ttm_1(\widehat{\cE}) \Vert, 
\end{align*}
with probability $1 - \delta/6$ (see Lemma~\ref{lemma: m1 norm upper bound} in Appendix), implying $\widehat{J} \le J$. If $C$ is significantly larger than $C'$, then the singular value  $\sigma_J(\ttm_1(\cR(\widehat{\Sigma}))) \ge \sigma_J(\ttm_1(\cR(\Sigma))) - \Vert \ttm_1(\widehat{\cE}) \Vert$ satisfies the inequality of the definition~\eqref{eq: definition of J estimator} with probability at least $1 - \delta/6$, so $J \le \widehat{J}$, and we conclude $J = \widehat{J}$ with probability at least $1 - \delta/2$. Analogously, one can show that $K = \widehat{K}$ for suitable choice of $C, C'$ with probability at least $1 - \delta/2$, yielding $J = \widehat{J}$ and $K = \widehat{K}$ with probability at least $1 - \delta$. 

    Then, while applying Algorithm~\ref{algo: order 3 TT-SVD} with $J \le \widehat{J}, K \le \widehat{K}$ could lead to better bias-variance tradeoff, using $J > \widehat{J}$ will result in much worse convergence rate in our model.

However, this holds assuming that~\eqref{eq: partial traces concentration} is fulfilled, so concentration bounds should be established for the norms of partial traces, which we left for future research. 

\section{Experiments}
\label{section: experiments}

In the present section, we illustrate that additional iterations $T$ of HardTTh indeed improve the estimation of the covariance matrix $\Sigma$ provided singular values of matricizations satisfy conditions of Theorem~\ref{theorem: Sigma estimator performance} up to some constant. We also compare HardTTh with several other algorithms.

To illustrate our theory, we construct a sampling model with the covariance matrix $\Sigma$ satisfying~\eqref{eq:kronecker_tt_model} as follows. Set $J = 7, K = 9$ and $p = q = r = 10$. Let $\cE^{ijk}$, $i \in [n], j \in [J], k \in [K]$ be $n \cdot JK$ tensors of shape $(p, q, r)$ consisting of i.i.d. standard Gaussian entries. Let $A_j \in \R^{p \times p}, B_{jk} \in \R^{q \times q}, C_k \in \R^{r \times r}$ be random symmetric matrices with diagonal and upper diagonal entries being i.i.d. Gaussian as well. Then,  random vectors $\bX_1, \ldots, \bX_n$ are defined as vectorized tensors
\begin{align*}
    \sum_{j = 1}^J \sum_{k = 1}^K \cE^{ijk} \times_3 C_k \times_2 B_{jk} \times_1 A_j \in \R^{p \times q \times r},
\end{align*}
conditioned on $A_j, B_{jk}, C_k$. The covariance matrix $\Sigma$ of $\bX_i$ satisfies (see \cite{puchkin24a})
\[
    \Sigma = \sum_{j = 1}^J \sum_{k = 1}^K A_j^2 \otimes B_{jk}^2 \otimes C_k^2.
\]

We propose several algorithms for comparative analysis with HardTTh. Specifically, we consider a version of Algorithm~\ref{algo: order 3 TT-SVD} with $T = 0$ additional steps, to which we refer as TT-HOSVD. This algorithm computes an approximate Tucker-2 decomposition of a noisy tensor $\cR(\widehat{\Sigma}) \approx \widehat{\cW} \times_3 \widehat{V}_0 \times_1 \widehat{U}_0$, and output the estimatior $\widehat{\cW} \times_3 \widehat{V}_0 \times_1 \widehat{U}_0$ of $\cR(\Sigma)$. We use this comparison to justify whether additional iterations are indeed necessary.

Furthermore, we modify the algorithm proposed in \cite{tsiligkaridis13} for use in our context. Instead of a single parameter $\lambda$ to control soft-thresholding, two distinct parameters are passed for each of the first and third matricizations of $\cR(\widehat{\Sigma})$. Using the first one, soft-thresholding upon first matricization is applied, then tensor is reshaped and soft-thresholding with another parameter upon third matricization is used. Then, we reshape the obtained tensor $\widehat{\cX}$ back into a matrix $\cR^{-1}(\widehat{\cX})$ of size $pqr \times pqr$. The pseudocode is given in Algorithm~\ref{algo:tsilikagridis_thresholding} in Appendix~\ref{section: tensor PRLS pseudocode}.

Finally, we compare HardTTh with the approximate Tucker decomposition with the Tucker ranks $(J, JK, K)$ using HOOI (Higher Order Orthogonal Iterations) algorithm of~\citet{zhang2018tensor}. If no additional iterations in this algorithm were applied, we refer to it as ``Tucker'' in our tables. Otherwise, we refer to it as ``Tucker+HOOI''.

We also include the sample covariance estimator into our comparative analysis.

We conduct several experiments varying the number of samples $n$. For $n = 500$, the result is given in Table~\ref{tab:covariance_algorithm_comparison-500 samples}. For $n = 2000$, the result is given in Table~\ref{tab:covariance_algorithm_comparison-2000 samples}. Other values of $n$ are studied in Appendix~\ref{section: additional experiments}. For each estimator $\widehat{S}$ of $\Sigma$, we compute the relative error $\Vert \widehat{S} - \Sigma \Vert_{\F}/ \Vert \Sigma \Vert_{\F}$ in the Frobenius norm. For each $n$, we tune parameters $\lambda_1, \lambda_2$ of the PRLS algorithm over a log-scale grid. We fix the number of iterations $T$ of HardTTh to $10$. Additional experiments with up to $10^9$ parameters confirm our method scalability and can be found in~\ref{subsection:extra_experiments}. We also refer the reader to this section for the rank ablation study. 

\begin{table}[ht]
\centering
\caption{Performance comparison of tensor decomposition algorithms for $n = 500$. Relative errors were averaged over 32 repeats of the experiment, empirical standard deviation is given after $\pm$ sign. The best results are boldfaced.}
\label{tab:covariance_algorithm_comparison-500 samples}
\begin{tabular}{lccc}
\toprule
\multirow{2}{*}{Metric} & \multicolumn{3}{c}{Algorithm} \\
\cmidrule(lr){2-4}
 & Sample Mean & TT-HOSVD & HardTTh \\
\midrule
Relative Error & $1.22 \pm 0.02$ & $0.269 \pm 0.008$ & $\mathbf{0.238 \pm 0.013}$ \\
Time (seconds) & $0.007\pm 0.003$ & $1.9\pm0.8$ & $2.7\pm0.8$ \\
\bottomrule
\end{tabular}

\begin{tabular}{lccc}
\toprule
\multirow{2}{*}{Metric} & \multicolumn{3}{c}{Algorithm} \\
\cmidrule(lr){2-4}
 & Tucker & Tucker+HOOI & PRLS \\
\midrule
Relative Error & $0.252 \pm 0.007$ & $0.240 \pm 0.013$ & $\mathbf{0.238\pm 0.017}$ \\
Time (seconds) & $41.3\pm1.7$ & $81.6\pm3.5$ & $0.7\pm0.3$ \\
\bottomrule
\end{tabular}
\end{table}

\begin{table}[ht]
\centering
\caption{Performance comparison of tensor decomposition algorithms for $n = 2000$. Relative errors were averaged over 16 repeats of the experiment, empirical standard deviation is given after $\pm$ sign. The best results are boldfaced.}
\label{tab:covariance_algorithm_comparison-2000 samples}
\begin{tabular}{lccc}
\toprule
\multirow{2}{*}{Metric} & \multicolumn{3}{c}{Algorithm} \\
\cmidrule(lr){2-4}
 & Sample Mean & TT-HOSVD & HardTTh \\
\midrule
Relative Error & $0.611 \pm 0.009$ & $0.154 \pm 0.006$ & $\mathbf{0.082 \pm 0.005}$ \\
Time (seconds) & $0.010 \pm 0.007$ & $1.7 \pm 0.6$ & $4.1\pm 1.1$ \\
\bottomrule
\end{tabular}

\begin{tabular}{lccc}
\toprule
\multirow{2}{*}{Metric} & \multicolumn{3}{c}{Algorithm} \\
\cmidrule(lr){2-4}
 & Tucker & Tucker+HOOI & PRLS \\
\midrule
Relative Error &$0.150 \pm 0.005$ & $\mathbf{0.082 \pm 0.005}$ & $0.216 \pm 0.012$ \\
Time (seconds) & $39.9 \pm 5.2$ & $74.2 \pm 8.1$ & $0.6 \pm 0.3$ \\
\bottomrule
\end{tabular}
\end{table}

Note that while the sample size increases by $4$, the relative error of HardTTh decreases by $3$, contradicting the $1/\sqrt{n}$ dependence between estimation error and the sample size. The reason is that for $n = 500$ neither TT-HOSVD nor HardTTh is able to reconstruct bases of $\Img \ttm_1(\cR(\Sigma))$ and $\Img \ttm_3(\cR(\Sigma))$, so the leading  error is determined by the lost components of these bases. Hence, one indeed needs some condition on the least singular values of matricizations of $\cR(\Sigma)$. When $n = 2000$, HardTTh is able to approximate these bases, yielding a much better performance, while TT-HOSVD cannot approximate them. It is instructive to look at $\sin \Theta$-distance between $\Img \widehat{U}_0, \Img \widehat{U}_T$ and $\Img U^*$. If $n = 500$, then both $\Img \widehat{U}_0, \Img \widehat{U}_T$ have $\sin \Theta$-distance to $\Img U^*$ around 1. But for $n = 2000$, while $\sin \Theta(\Img \widehat{U}_0, \Img U^*)$ is still around $1$, we have $\sin \Theta(\Img \widehat{U}_T, \Img U^*) = 0.33\pm 0.08$. Therefore, additional iterations of HardTTh indeed help. 

The fact that noise in singular values is larger than the estimation error is illustrated by the fact that PRLS performs worse than TT-HOSVD. Indeed, to remove noise in singular values, PRLS applies soft-thresholding with $\lambda_1, \lambda_2$ being around the noise level in singular values of matricizations. Then, soft-thresholded SVD has each singular value decreased by either $\lambda_1/2$ or $\lambda_2/2$. This yields the estimation error around the maximum of $\lambda_1$ and $\lambda_2$, which dramatically affects the algorithm performance. This highlights the difference between low-rank tensor estimation problem and low-rank matrix estimation problem, since for the latter there is no significant difference between soft-thresholding and hard-thresholding estimation.



\bibliography{references}

\newpage
\appendix

\section{Additional notations and basic tools}

For proofs, we need some extra notation. First, we adapt the Einstein notation for tensors, omitting the summation symbol and assuming that the summation holds across repeated indices, e.g. for the matrix product
\begin{align*}
    (A B)_{ab} = \sum_{c} A_{ac} B_{cb},
\end{align*}
we will write as
\begin{align*}
    (AB)_{ab} = A_{ac} B_{cb}.
\end{align*}

Second, we will widely use the following identities for a tensor $\cT \in \R^{d_1 \times d_2 \times d_3}$ and a matrix $X$ of suitable shape
\begin{equation}
\begin{aligned}
    \ttm_1(\cT \times_3 X) & = \ttm_1(\cT)(I_{d_2} \otimes X^\top), \\
    \ttm_1(\cT \times_1 X) & = X \cdot \ttm_1(\cT), \\
    \ttm_3( \cT X \times_1 X) & = \ttm_3(\cT) (X^\top \otimes I_{d_2}), \\
    \ttm_3(\cT \times_3 X) & = X \cdot \ttm_3(\cT).
\end{aligned} \label{eq: matricization-tensor idenities}
\end{equation}
While the second and the fourth identities are straightforward, the first and the last one should be verified. Let us prove the first identity for $X \in \R^{d' \times d_3}$. Choosing indices $a \in [d_1], b \in [d_2], c \in [d']$, we obtain
\begin{align*}
    \left ( \ttm_1(\cT \times_3 X) \right )_{a, (b - 1) \cdot d_3 + c} & = (\cT \times_3 X)_{abc} = X_{c c'} \cT_{a bc'} \\
    & = \ttm_1(\cT)_{a, (b' - 1) d_3 + c'} (I_{d_2} \otimes X^\top)_{(b' - 1) d_3 + c', (b - 1) d_3 + c}.
\end{align*}
The third idenitty of~\eqref{eq: matricization-tensor idenities} can be checked analogously.

For a matrix $U \in \bbO_{d, r}$, we denote the projector $U U^\top$ on $\Img U$ by $\Pi_{U}$.

\section{Proof of Proposition~\ref{proposition: bounds on effective dimensions}}
\label{section: proof of effective dimensions bounds}

\begin{proof}
The proposition follows from the following bound on the partial trace. Let $\Psi_g : L_1 \to L_1, \Phi_g: L_2 \to L_2, g = 1, \ldots, G,$ be positive semidefinite operators. Define
\begin{align*}
    H = \sum_{g = 1}^G \Psi_g \otimes \Phi_g.
\end{align*}
Then, we have
\begin{align*}
    \Vert \Tr_{L_1}(H) \Vert & = \Vert \sum_{g = 1}^G \Tr(\Psi_g) \Phi_g \Vert \le \sum_{g = 1}^G \frac{\Tr(\Psi_g)}{\Vert \Psi_g \Vert}  \Vert \Psi_g \Vert \Vert \Phi_g \Vert \le \max_g \ttr(\Psi_g) \sum_{g = 1}^G \Vert \Psi_g \Vert \Vert \Phi_g \Vert\\
    & \le G \cdot \max_g \ttr(\Psi_g) \cdot \max_g \Vert \Psi_g \Vert \Vert \Phi_g \Vert \le G \cdot \max_g \ttr(\Psi_g) \cdot \Vert H \Vert.
\end{align*}
The result follows by applying the above to each partial trace $\Tr_S(\Sigma), S \subseteq [3],$ with a proper choice of $L_1 ,\Psi_g$ and $\Phi_g$.
\end{proof}

\section{Proof sketch for Theorem~\ref{theorem: Sigma estimator performance}}
\label{section: proof sketch}

In this section, we provide the sketch of the proof of Theorem~\ref{theorem: Sigma estimator performance}. The proof develops the ideas of~\citet{zhang2018tensor} and~\citet{puchkin24a}. First, we consider the problem of estimating a tensor $\cT^* = \cW^* \times_3 V^*\times_1 U^*$ from a noisy observations $\cY = \cT^* + \cE$, without any assumptions on the error term $\cE$. Let $\widehat{\cT}$ be the estimator obtained by Algorithm~\ref{algo: order 3 TT-SVD} on the input $\cY$. The noise $\cE$ influence the estimation of $\widehat{\cT}$ in several ways. First, one need to impose some assumptions depending on the norms of $\ttm_1(\cE)$ and $\ttm_3(\cE \times_1 \widehat{U}_0)$ on the singular values of matricizations $\ttm_1(\cT^*), \ttm_3(\cT^*)$  to be able to recover  left singular subspaces of these matricizations up to a $\sin \Theta$-error at most $1/4$. Second, we show by induction on $t = 1, \ldots, T$ that $\Img \widehat{U}_t, \Img \widehat{V}_t$ improves the estimation of singular subspaces and establish the dependence of the estimation error on $\cE$ at step $T$. Finally, we decompose the error $\Vert \widehat{\cT} - \cT^* \Vert_{\F}$ into terms depending on the singular subspaces estimation and the error of estimating $\cW^*$. Combining all types of errors, we obtain the following theorem. Its proof if postponed to Section~\ref{section:proof of sensitivity analysis theorem}.

\begin{Th}
\label{theorem: sensetivity analysis}
    Given model~\eqref{eq: tensor model}, suppose that singular values $\sigma_J(\ttm_1(\cT^*)), \sigma_K(\ttm_3(\cT^*))$ satisfy
    \begin{align}
    \label{eq: sensititivty analysis -- sigma inequalities}
        \sigma_J(\ttm_1(\cT^*)) \ge 24 \Vert \ttm_1(\cE) \Vert \quad \text{and} \quad \sigma_K(\ttm_3(\cT^*)) \ge 24 \sup_{\substack{U \in \R^{d_1 \times J} \\ \Vert U \Vert \le 1}}\Vert \ttm_3(\cE) (U \otimes I_{d_2})\Vert.
    \end{align}
    Put 
    \begin{align*}
        \alpha_U & = \Vert \ttm_1(\cE \times_3 (V^*)^\top) \Vert, \qquad & \beta_U = \sup_{\substack{V \in \R^{d_2 \times K} \\ \Vert V \Vert \le 1}} \Vert \ttm_1(\cE \times_3 V^\top) \Vert, \\
        \alpha_V & = \Vert \ttm_3(\cE \times_1 (U^*)^\top) \Vert, \qquad & \beta_V = \sup_{\substack{U \in \R^{d_1 \times J} \\ \Vert U \Vert \le 1}} \Vert \ttm_3(\cE \times_1 U^\top) \Vert.
    \end{align*}
    Then, we have
    \begin{align*}
     \Vert \widehat{\cT} - \cT^* \Vert_{\F} & \le \sup_{U \in \bbO_{d_1, J}, V \in \bbO_{d_2, K}} \Vert \cE \times_3 V^\top \times_1 U^\top \Vert_{\F} + 4 \sqrt{K} \alpha_V+ 4 \sqrt{J} \alpha_U + \diamondsuit_2 + r_T,
\end{align*}
where 
\begin{align*}
\diamondsuit_2 & = 48 \cdot \left ( \frac{ \sqrt{K}\beta_V \alpha_U}{\sigma_J(\ttm_1(\cT^*))} + \frac{\sqrt{J}\beta_U \alpha_V}{\sigma_K(\ttm_3(\cT^*))}\right ), \\
r_T & = 3 (\sqrt{J} + \sqrt{K}) \cdot \left ( \frac{64\beta_V \beta_U}{\sigma_J(\ttm_1(\cT^*)) \sigma_K(\ttm_3(\cT^*))}\right )^T \Vert \ttm_1(\cE) \Vert.
\end{align*}
\end{Th}

Then, we decompose the error $\cE$ into the bias part $\overline{\cE}$ and the varaince part $\widehat{\cE}$. Using the triangle inequality, we bound each error term appearing in Theorem~\ref{theorem: sensetivity analysis} into the bias and variance parts, and bound the variance parts with high probability using the variational PAC--Bayes approach (see~\citep{catoni17,zhivotovskiy24,abdalla22,puchkin24a} for other applications of this technique).

\section{Proof of Theorem~\ref{theorem: Sigma estimator performance}}
\label{section: proof of main theorem}

\begin{proof}[Proof of Theorem~\ref{theorem: Sigma estimator performance}] For clarity, we divide the proof into several steps. For brevity, we denote $\cR(\ttm_i(\cdot))$, $i = 1, 3$, by $\cR_i(\cdot)$.

\noindent \textbf{Step 1. Sensititivty analysis of Algorithm~\ref{algo: order 3 TT-SVD}.} First, we establish deterministic bounds on the reconstruction of the tensor $\cT^*$ from a noisy observation $\cY$ by Algorithm~\ref{algo: order 3 TT-SVD}, denoting
\begin{align}
    \cY = \cT^* + \cE, \label{eq:  tensor model}
\end{align}
where $\cT^* = \cW^* \times_3 V^* \times_1 U^*$ is the best $(J,K)$-TT-rank approximation of $\cR(\Sigma)$, $U^* \in \bbO_{d_1, J}$, $V^* \in \bbO_{d_3, K}$, $\cW^* \in \R^{J \times d_2 \times K}$, and $\cY = \cR(\widehat{\Sigma})$. Let $\widehat{\cT}$ be the output of Algorithm~\ref{algo: order 3 TT-SVD} with input $\cY$. Then, Theorem~\ref{theorem: sensetivity analysis} is applicable. But we need first to check its conditions.

\noindent \textbf{Step 2. Checking conditions of Theorem~\ref{theorem: sensetivity analysis}.} We deduce Theorem~\ref{theorem: Sigma estimator performance} from Theorem~\ref{theorem: sensetivity analysis}. Let us start with conditions of Theorem~\ref{theorem: sensetivity analysis}, and bound right-hand sides of inequalities~\eqref{eq: sensititivty analysis -- sigma inequalities} from above. Consider the lower bound on $\sigma_J(\ttm_1(\cT^*))$. By the triangle inequality, we have
\begin{align*}
    \Vert \ttm_1(\cE) \Vert \le \Vert \ttm_1(\overline{\cE}) \Vert + \Vert \ttm_1(\widehat{\cE}) \Vert.
\end{align*}
The second term of the above can be upper bounded using the following lemma.

\begin{Lem}
\label{lemma: m1 norm upper bound}
Fix $\delta \in (0, 1)$. Suppose that $n \ge \ttr_1^2(\Sigma) + \ttr_2^2(\Sigma)\ttr_3^2(\Sigma) + \log(4/\delta)$. Then, under Assumption~\ref{as:orlicz}, we have
\begin{align*}
    \Vert \ttm_1(\widehat{\cE}) \Vert \le 32 \omega \Vert \Sigma \Vert \sqrt{\frac{\ttr_1^2(\Sigma) +  \ttr_2^2(\Sigma) \ttr_3^2(\Sigma) + \log(1/\delta)}{n}}
\end{align*}
with probability at least $1 - \delta$.
\end{Lem}
Define the event
\begin{align}
\label{eq: event_1 definition}
    \event_1 = \left \{ \Vert \ttm_1(\widehat{\cE}) \Vert \le 32 \omega \Vert \Sigma \Vert \sqrt{\frac{\ttr_1^2(\Sigma) + \ttr_2^2(\Sigma) \ttr_3^2(\Sigma) + \log(6/\delta)}{n}} \right \}.
\end{align}
Since $n \ge \ttR_{\delta} \ge  \ttr_1^2(\Sigma) + \ttr_2^2(\Sigma)\ttr_3^2(\Sigma) + \log(24/\delta)$, due to Lemma~\ref{lemma: m1 norm upper bound}, we have $\Pr(\event_1) \ge 1 - \delta/6$. Hence, if 
\begin{align*}
    \sigma_J(\ttm_1(\cT^*)) \ge 24 \Vert \ttm_1(\overline{\cE}) \Vert + 768 \omega \Vert \Sigma \Vert \sqrt{\frac{\ttr_1^2(\Sigma) + \ttr_2^2(\Sigma) \ttr_3^2(\Sigma) + \log(6/\delta)}{n}},
\end{align*}
the first inequality of~\eqref{eq: sensititivty analysis -- sigma inequalities} is fulfilled on the event $\event_1$. Since $\sigma_J(\ttm_1(\cT^*)) \ge \sigma_J(\cR_1(\Sigma)) - \Vert \ttm_1(\overline{\cE}) \Vert$, on $\event_1$, to fulfill the first inequality of~\eqref{eq: sensititivty analysis -- sigma inequalities}, it is enough to ensure that
\begin{align*}
    \sigma_J(\cR_1(\Sigma)) \ge 25 \Vert \ttm_1(\overline{\cE}) \Vert + 768 \omega \Vert \Sigma \Vert \sqrt{\frac{\ttr_1^2(\Sigma) + \ttr_2^2(\Sigma) \ttr_3^2(\Sigma) + \log(6/\delta)}{n}},
\end{align*}
as guaranteed by the conditions of the theorem.

To satisfy the second inequality of~\eqref{eq: sensititivty analysis -- sigma inequalities}, we use the triangle inequality again and obtain
\begin{align*}
    \sup_{\substack{U \in \R^{d_1 \times J} \\ \Vert U \Vert \le 1}}\Vert \ttm_3(\cE) (U \otimes I_{d_2})\Vert \le  \sup_{\substack{U \in \R^{d_1 \times J} \\ \Vert U \Vert \le 1}}\Vert \ttm_3(\overline{\cE}) (U \otimes I_{d_2})\Vert + \sup_{\substack{U \in \R^{d_1 \times J} \\ \Vert U \Vert \le 1}}\Vert \ttm_3(\widehat{\cE}) (U \otimes I_{d_2})\Vert.
\end{align*}
We bound the second term, using the following lemma. Its proof is given in Section~\ref{section: proof of lemma sup U times noise}.
\begin{Lem}
\label{lemma: sup U times noise}
Fix $\delta \in (0, 1)$. Suppose that $n \ge J \ttr_1^2(\Sigma) + J \ttr_2^2(\Sigma) +  \ttr_3^2(\Sigma) + \log (8/\delta)$. Then, with probability at least $1 - \delta$, we have
\begin{align*}
    \sup_{\substack{U \in \R^{d_1 \times J} \\ \Vert U \Vert \le 1}} \Vert \ttm_3(\widehat{\cE}) (U \otimes I_{d_2})\Vert \le 32 \omega \Vert \Sigma \Vert \sqrt{\frac{J \ttr_1^2(\Sigma) + J \ttr_2^2(\Sigma) + \ttr_3^2(\Sigma) + \log(8/\delta)}{n}}.
\end{align*}
    Analogously, if $n \ge \ttr_1^2(\Sigma) + K \ttr_2^2(\Sigma) + K \ttr_3^2(\Sigma) + \log(8/\delta)$, then, with probability at least $1 - \delta$, it holds that
\begin{align*}
    \sup_{V \in\R^{d_3 \times K}, \Vert V \Vert \le 1} \Vert \ttm_1(\cE) (I_{d_2} \otimes V) \Vert \le 32 \omega \Vert \Sigma \Vert \sqrt{\frac{\ttr_1^2(\Sigma) + K \ttr_2^2(\Sigma) + K \ttr_3^2(\Sigma) + \log(8/\delta)}{n}}.
\end{align*}
\end{Lem}
Define the event
\begin{align*}
    \event_2 = \left \{ \sup_{\substack{U \in \R^{d_1 \times J} \\ \Vert U \Vert \le 1}}\Vert \ttm_3(\widehat{\cE}) (U \otimes I_{d_2})\Vert \le 32 \omega \Vert \Sigma \Vert \sqrt{\frac{\ttr_3^2(\Sigma) + J \ttr_1^2(\Sigma) + J \ttr_2^2(\Sigma) + \log(48/\delta)}{n}} \right \}.
\end{align*}
It has probability $\Pr(\event_2) \ge 1 - \delta/6$, since  $n \ge \ttR_{\delta}$ satisfies conditions of Lemma~\ref{lemma: sup U times noise} with $\delta/6$ in place of $\delta$. Due to conditions of the theorem, we have
\begin{align*}
    \sigma_K(\cR_3(\Sigma)) \ge 25 \Vert \ttm_3(\overline{\cE}) \Vert + 768 \omega \Vert \Sigma \Vert \sqrt{\frac{\ttr_3^2(\Sigma) + J \ttr_1^2(\Sigma) + J \ttr_2^2(\Sigma) + \log(48/\delta)}{n}},
\end{align*}
so conditions of Theorem~\ref{theorem: sensetivity analysis} is satisfied on $\event_1 \cap \event_2$.  

\noindent \textbf{Step 3. Bounding $\alpha_U, \alpha_V, \beta_U, \beta_V$.} Then, we bound $\alpha_U, \alpha_V, \beta_U, \beta_V$. We start by the former two quantities. By the triangle inequality, we have
\begin{align*}
    \alpha_U & \le \Vert \ttm_1(\overline{\cE} \times_3 (V^*)^\top \Vert + \Vert \ttm_1(\widehat{\cE} \times_3 (V^*)^\top \Vert, \\
    \alpha_V & \le \Vert \ttm_3(\overline{\cE} \times_1 (U^*)^\top \Vert + \Vert \ttm_3(\widehat{\cE} \times_3 (U^*)^\top \Vert.
\end{align*}
To bound the second terms of the right-hand sides of the above, we use the following lemma. Its proof is given in Section~\ref{section: proof of lemma U star times noise bound}.
\begin{Lem}
\label{lemma: U star times noise bound}
Fix $\delta \in (0, 1)$. Suppose that $n \ge \ttr_1^2(\Sigma) + K \ttr^2_2(\Sigma) + \log(8/\delta)$. Then, with probability at least $1 - \delta$, we have
\begin{align*}
    \Vert \ttm_1(\widehat{\cE} \times_3 (V^*)^\top) \Vert \le 32 \omega \Vert \Sigma \Vert \sqrt{\frac{\ttr_1^2(\Sigma) + K \ttr^2_2(\Sigma) + \log(8/\delta)}{n}}.   
\end{align*}
Analogously, if $n \ge \ttr_3^2(\Sigma) + J \ttr^2_2(\Sigma) + \log(8/\delta)$, then, with probability at least $1 - \delta$, we have
\begin{align*}
    \Vert \ttm_3(\widehat{\cE} \times_3 (U^*)^\top) \Vert \le 32 \omega \Vert \Sigma \Vert \sqrt{\frac{\ttr_3^2(\Sigma) + J \ttr^2_2(\Sigma) + \log(8/\delta)}{n}}.
\end{align*}
\end{Lem}

Define events
\begin{align*}
    \event_3 & = \left \{ \Vert \ttm_1(\widehat{\cE} \times_3 (V^*)^\top \Vert \lesssim \omega \Vert \Sigma \Vert \sqrt{\frac{\ttr_1^2(\Sigma) + K \ttr^2_2(\Sigma) + \log(6/\delta)}{n}} \right \}, \\
    \event_4 & = \left \{ \Vert \ttm_3(\widehat{\cE} \times_3 (U^*)^\top \Vert \lesssim \omega \Vert \Sigma \Vert \sqrt{\frac{\ttr_3^2(\Sigma) + J \ttr^2_2(\Sigma) + \log(6/\delta)}{n}}  \right \}.
\end{align*}
Since $n \ge \ttR_{\delta}$ satisfies the conditions of Lemma~\ref{lemma: U star times noise bound} with $\delta/6$ in place of $\delta$, the lemma and the union bound imply $\Pr(\event_3 \cap \event_4) \ge 1 - \delta/3$. On the event $\event_3 \cap \event_4$, we have
\begin{align*}
    \alpha_U \le \widetilde{\alpha}_U \quad \text{and} \quad  \alpha_V \le \widetilde{\alpha}_V,
\end{align*}
where $\widetilde{\alpha}_U, \widetilde{\alpha}_V$ are defined in Table~\ref{tab: ancillary variables}. 

Next, we bound $\beta_U, \beta_V$. Applying the triangle inequality, we get
\begin{align*}
    \beta_U & \le \sup_{\substack{V \in \R^{d_2 \times K} \\ \Vert V \Vert \le 1}} \Vert \ttm_1(\overline{\cE} \times_3 V^\top) \Vert +  \sup_{\substack{V \in \R^{d_2 \times K} \\ \Vert V \Vert \le 1}} \Vert \ttm_1(\widehat{\cE} \times_3 V^\top) \Vert, \\
    \beta_V & \le  \sup_{\substack{U \in \R^{d_1 \times J} \\ \Vert U \Vert \le 1}}\Vert \ttm_3(\overline{\cE}) (U \otimes I_{d_2})\Vert + \sup_{\substack{U \in \R^{d_1 \times J} \\ \Vert U \Vert \le 1}}\Vert \ttm_3(\widehat{\cE}) (U \otimes I_{d_2})\Vert.
\end{align*}
Note that on the event $\event_2$, we have $\beta_V \le \widetilde{\beta}_V$, where $\widetilde{\beta}_V$ is defined in Table~\ref{tab: ancillary variables}. To bound $\beta_U$, we use Lemma~\ref{lemma: sup U times noise} again. Define an event
\begin{align*}
    \event_5 = \left \{\sup_{\substack{V \in \R^{d_2 \times K} \\ \Vert V \Vert \le 1}} \Vert \ttm_1(\widehat{\cE} \times_3 V^\top) \Vert \le 32 \omega \Vert \Sigma \Vert \sqrt{\frac{\ttr_1^2(\Sigma) + K \ttr_2^2(\Sigma) + K \ttr_3^2(\Sigma) + \log(48/\delta)}{n}} \right \}.
\end{align*}
Since $n \ge \ttR_{\delta}$ satisfies the conditions of the lemma with $\delta/6$ in place of $\delta$, we have $\Pr(\event_5) \ge 1 - \delta/6$, and on this event $\beta_U \le \widetilde{\beta}_U$. 

\noindent \textbf{Step 4. Bounding $\sup_{U \in \bbO_{d_1, J}, V \in \bbO_{d_2, K}} \Vert \cE \times_3 V^\top \times_1 U^\top \Vert_{\F}$.} Using the triangle inequality again, we get
\begin{align*}
    \sup_{U \in \bbO_{d_1, J}, V \in \bbO_{d_2, K}} \Vert \cE \times_3 V^\top \times_1 U^\top \Vert_{\F} & \le \sup_{U \in \bbO_{d_1, J}, V \in \bbO_{d_2, K}} \Vert \overline{\cE} \times_3 V^\top \times_1 U^\top \Vert_{\F} \\
    & \quad + \sup_{U \in \bbO_{d_1, J}, V \in \bbO_{d_2, K}} \Vert \widehat{\cE} \times_3 V^\top \times_1 U^\top \Vert_{\F}.
\end{align*}
We bound the second term of the right-hand side using the following lemma. Its proof is given in Section~\ref{section: proof of lemma leading term bound}.
\begin{Lem}
\label{lemma: leading term bound}
    Fix $\delta \in (0, 1)$. Suppose that $n \ge J \ttr_1^2(\Sigma) + JK \ttr_2^2(\Sigma) + K \ttr^2_3(\Sigma) + \log(8/\delta)$. Then, with probability at least $1 - \delta$, we have
    \begin{align*}
        \sup_{U \in \bbO_{d_1, J}, V \in \bbO_{d_2, K}} \Vert \widehat{\cE} \times_3 V^\top \times_1 U^\top \Vert_{\F} \le 32 \omega \Vert \Sigma \Vert \sqrt{\frac{J\ttr_1^2(\Sigma) + JK \ttr_2^2(\Sigma) + K \ttr^2_3(\Sigma) + \log(8/\delta)}{n}}.
    \end{align*}
\end{Lem}
Define the event
\begin{align*}
    \event_6 & = \left \{\sup_{U \in \bbO_{d_1, J}, V \in \bbO_{d_2, K}} \Vert \widehat{\cE} \times_3 V^\top \times_1 U^\top \Vert_{\F} \right. \\
     & \left. \qquad \qquad \qquad  \le 32 \Vert \Sigma \Vert \sqrt{\frac{J\ttr_1^2(\Sigma) + JK \ttr_2^2(\Sigma) + K \ttr^2_3(\Sigma) + \log(48/\delta)}{n}} \right \}.
\end{align*}
Since $n \ge \ttR_{\delta}$ satisfies the conditions of Lemma~\ref{lemma: leading term bound} with $\delta/6$ in place of $\delta$, it implies $\Pr(\event_6) \ge 1 - \delta/6$.

\noindent \textbf{Step 5. Establishing bias and variance leading terms.} The event $\event_0 = \bigcap_{i = 1}^6 \event_i$ has probability at least $1 - \delta$ due to the union bound. On the event $\event_0$, conditions of Theorem~\ref{theorem: sensetivity analysis} are satisfied, so we have
\begin{align*}
    \alpha_U \le \widetilde{\alpha}_U, \quad \alpha_V \le \widetilde{\alpha}_V, \quad \beta_U \le \widetilde{\beta}_U, \quad \beta_V \le \widetilde{\beta}_V
\end{align*}
and
\begin{align*}
    \sup_{U \in \bbO_{d_1, J}, V \in \bbO_{d_2, K}} \Vert \widehat{\cE} \times_3 V^\top \times_1 U^\top \Vert_{\F} \le 32 \omega \Vert \Sigma \Vert \sqrt{\frac{J\ttr_1^2(\Sigma) + JK \ttr_2^2(\Sigma) + K \ttr^2_3(\Sigma) + \log(48/\delta)}{n}}.
\end{align*}
The conclusion of Theorem~\ref{theorem: sensetivity analysis} yields
\begin{align*}
    \Vert \widehat{\cT} - \cT^* \Vert_{\F} & \le \sup_{U \in \bbO_{d_1, J}, V \in \bbO_{d_2, K}} \Vert \overline{\cE} \times_3 V^\top \times_1 U^\top \Vert_{\F} \\
    & \quad + \omega \Vert \Sigma \Vert \sqrt{\frac{J\ttr_1^2(\Sigma) + JK \ttr_2^2(\Sigma) + K \ttr^2_3(\Sigma) + \log(6/\delta)}{n}} \\
    & \quad + 4 \sqrt{K} \widetilde{\alpha}_U + 4 \sqrt{J} \widetilde{\alpha}_U + \diamondsuit_2 + r_T \\
\end{align*}
Substituting expressions for $\widetilde{\alpha}_U, \widetilde{\alpha}_V$ from Table~\ref{tab: ancillary variables}, we obtain
\begin{align*}
    \Vert \widehat{\cT} - \cT^* \Vert_{\F} & \le \sup_{U \in \bbO_{d_1, J}, V \in \bbO_{d_2, K}} \Vert \overline{\cE} \times_3 V^\top \times_1 U^\top \Vert_{\F} + 4 \sqrt{K} \Vert \ttm_1(\overline{\cE} \times_3 (V^*)^\top) \Vert \\
    & \quad + 4 \sqrt{J} \Vert \ttm_3(\overline{\cE} \times_1 (U^*)^\top) \Vert \\
    & \quad + 32 \omega \Vert \Sigma \Vert \sqrt{\frac{J\ttr_1^2(\Sigma) + JK \ttr_2^2(\Sigma) + K \ttr^2_3(\Sigma) + \log(48/\delta)}{n}} \\
     & \quad + 32 \sqrt{J} \omega  \Vert \Sigma \Vert \sqrt{\frac{\ttr_1^2(\Sigma) + K \ttr^2_2(\Sigma) + \log(48/\delta)}{n}} \\
     & \quad + 32 \sqrt{K} \omega \Vert \Sigma \Vert \sqrt{\frac{\ttr_3^2(\Sigma) + J \ttr^2_2(\Sigma) + \log(48/\delta)}{n}} + \diamondsuit_2 + r_T.
\end{align*}
Note that the fifth and sixth terms of the right-hand side are dominated by the fourth term.
Using 
\begin{align*}
    \Vert \widetilde{\Sigma} - \Sigma \Vert_{\F}  & = \Vert \widehat{\cT} - \cT^* + \cT^* - \cR^{-1}(\Sigma) \Vert_{\F} \le \Vert \widehat{\cT} - \cT^* \Vert_{\F} + \Vert \overline{\cE} \Vert_{\F}, \\
    \sup_{U \in \bbO_{d_1, J}, V \in \bbO_{d_2, K}} \Vert \overline{\cE} \times_3 V^\top \times_1 U^\top \Vert_{\F} & \le \sup_{U \in \bbO_{d_1, J}} \sup_{V \in \bbO_{d_3,K}} \Vert U^\top \ttm_1(\overline{\cE})(I_{d_2}\otimes V) \Vert_{\F} \\
    & \le \sqrt{J} \sup_{V \in \bbO_{d_3,K}} \Vert \ttm_1(\overline{\cE})(I_{d_2} \otimes V) \Vert \le \sqrt{J} \Vert \ttm_1(\overline{\cE}) \Vert, \\
    \Vert \ttm_3(\overline{\cE} \times_1 (U^*)^\top) \Vert & \le \Vert \ttm_3(\overline{\cE}) \Vert, \\
    \Vert \ttm_1(\overline{\cE} \times_3 (V^*)^\top) \Vert & \le \Vert \ttm_1(\overline{\cE}) \Vert,
\end{align*}
we derive
\begin{align}
\label{eq: established leading terms of the complete error}
    \Vert \widetilde{\Sigma} - \Sigma \Vert_{\F} \le \overline{\bb} + 96 \omega \Vert \Sigma \Vert \sqrt{\frac{J\ttr_1^2(\Sigma) + JK \ttr_2^2(\Sigma) + K \ttr^2_3(\Sigma) + \log(48/\delta)}{n}} + \diamondsuit_2 + r_T
\end{align}
on $\event_0$.

\noindent \textbf{Step 6. Bounding the remainder terms.} Since $\diamondsuit_2, r_T$ depend on $1/\sigma_J(\ttm_1(\cT^*))$ and $1/\sigma_K(\ttm_3(\cT^*))$, we will bound singular values $\sigma_J(\ttm_1(\cT^*)), \sigma_K(\ttm_3(\cT^*))$ below using $\sigma_J(\cR_1(\Sigma)), \sigma_K(\cR_3(\Sigma))$. By the conditions of the theorem, we have $\sigma_J(\cR_1(\Sigma)) \ge 25 \Vert \ttm_1(\overline{\cE}) \Vert$ and $\sigma_K(\cR_3(\Sigma)) \ge \Vert \ttm_3(\overline{\cE}) \Vert$, so, by the Weyl inequality, we deduce
\begin{align*}
    \sigma_J(\ttm_1(\cT^*)) & \ge \sigma_J(\cR_1(\Sigma)) - \Vert \ttm_1(\overline{\cE}) \Vert \ge \frac{24}{25} \cdot \sigma_J(\cR_1(\Sigma)), \\
    \sigma_K(\ttm_3(\cT^*)) & \ge \sigma_K(\cR_3(\Sigma)) - \Vert \ttm_3(\overline{\cE}) \Vert \ge \frac{24}{25} \cdot \sigma_K(\cR_3(\Sigma)).
\end{align*}
On the event $\event_0$, it implies
\begin{align*}
    \diamondsuit_2 & = 48 \cdot \left ( \frac{\sqrt{K} \beta_V \alpha_U}{\sigma_J(\ttm_1(\cT^*))} + \frac{\sqrt{J} \beta_U \alpha_V}{\sigma_K(\ttm_3(\cT^*))} \right ) \\
    & \le 50 \cdot \left ( \frac{\sqrt{K} \widetilde{\beta}_V \widetilde{\alpha}_U}{\sigma_J(\cR_1(\cT^*))} + \frac{\sqrt{J} \widetilde{\beta}_U \widetilde{\alpha}_V}{\sigma_K(\cR_3(\Sigma))} \right) = \widetilde{\diamondsuit}_2,
\end{align*}
and 
\begin{align*}
    r_T & = 3 (\sqrt{J} + \sqrt{K}) \cdot \left ( \frac{64 \beta_V \beta_U}{\sigma_J(\ttm_1(\cT^*)) \sigma_K(\ttm_3(\cT^*))} \right )^T \Vert \ttm_1(\cE) \Vert \\
    & \le (\sqrt{J} + \sqrt{K}) \left ( \frac{200 \widetilde{\beta}_V \widetilde{\beta}_U}{\sigma_J(\cR_1(\Sigma)) \sigma_K(\cR_3(\Sigma))} \right )^T \Vert \ttm_1(\cE) \Vert.
\end{align*}
Using definition~\eqref{eq: event_1 definition} of the event $\event_1$, $\event_0 \subset \event_1$, and the trinagle inequality $\Vert \ttm_1(\cE) \Vert \le \Vert \ttm_1(\overline{\cE}) \Vert + \Vert \ttm_1(\widehat{\cE}) \Vert$, we obtain
\begin{align*}
    r_T \le \widetilde{r}_T,
\end{align*}
where $\widetilde{r}_T$ is defined in Table~\ref{tab: ancillary variables}. Substituting the above bounds on $\diamondsuit_2, r_T$ into~\eqref{eq: established leading terms of the complete error} finishes the proof.
\end{proof}

\subsection{Proof of Lemma~\ref{lemma: m1 norm upper bound}}

\begin{proof}
\noindent \textbf{Step 1. Reduction to the PAC-bayes inequality.} The analysis will be based the following lemma, which is known as the PAC-Bayes inequality (see, e.g., \cite{catoni17}).
\begin{Lem}
    \label{lem:pac-bayes}
    Let $ \bX, \bX_1, \dots, \bX_n$ be i.i.d. random elements on a measurable space $\cX$. Let $\Theta$ be a parameter space equipped with a measure $\mu$ (which is also referred to as prior). Let $f : \cX \times \Theta \rightarrow \R$. Then, with probability at least $1 - \delta$, it holds that
    \[
        \E_{\btheta \sim \rho} \frac1n \sum\limits_{i = 1}^n f(\bX_i, \btheta)
        \leq \E_{\btheta \sim \rho} \log \E_\bX e^{f(\bX, \btheta)} + \frac{\KL(\rho, \mu) + \log(1 / \delta)}n
    \]
    simultaneously for all $\rho \ll \mu$.
\end{Lem}
Let us rewrite $\Vert \ttm_1(\widehat{\cE}) \Vert$ as the supremum of a certain empirical process. We have
\begin{align*}
    \Vert \ttm_1(\widehat{\cE}) \Vert = & \sup_{\bx \in \bbS^{d_1 - 1}, \by \in \bbS^{d_2 d_3  - 1}} \bx^\top \ttm_1(\widehat{\cE}) \by = \sup_{\bx \in \bbS^{d_1 - 1}, \by \in \bbS^{d_2 d_3 - 1}} \langle \ttm_1(\widehat{\cE}), \bx \by^{\top} \rangle \\
    & = \sup_{\bx \in \bbS^{d_1 - 1}, \by \in \bbS^{d_2 d_3 - 1}} \langle \widehat{\Sigma} - \Sigma, \cR_1^{-1}(\bx \by^\top) \rangle \\
    &  = \sup_{\bx \in \bbS^{d_1 - 1}, \by \in \bbS^{d_2 d_3 - 1}} \frac{1}{n} \sum_{i = 1}^n \langle \bX_i \bX_i^\top, \cR_1^{-1}(\bx \by^\top) \rangle - \E \langle \bX_i \bX_i^\top, \cR_1^{-1}(\bx \by^\top) \rangle \\
    & = \sup_{\bx  \in \bbS^{d_1 - 1}, \by \in \bbS^{d_2 d_3 - 1}} \frac{1}{n} \sum_{i = 1}^n \bX_i^\top \cR_1^{-1}(\bx \by^\top) \bX_i - \E \bX_i^\top \cR_1^{-1}(\bx \by^\top) \bX_i.
\end{align*}
Define the following functions:
\begin{align*}
    f_i(\bx, \by) & = \lambda \left \{ \bX_i^\top \cR_1^{-1}(\bx \by^\top) \bX_i - \E \bX_i^\top \cR_1^{-1}(\bx \by^\top) \bX_i \right \}, \\
    f_{\bX}(\bx, \by) & = \lambda \left \{ \bX^\top \cR_1^{-1}(\bx \by^\top) \bX - \E \bX^\top \cR_1^{-1}(\bx \by^\top) \bX \right \},
\end{align*}
where the positive factor $\lambda$ to be chosen later.
We will apply Lemma~\ref{lem:pac-bayes} to the empirical process
\begin{align*}
    \lambda  \Vert \ttm_1(\widehat{\cE}) \Vert = \sup_{\bx  \in \bbS^{d_1 - 1}, \by \in \bbS^{d_2 d_3 - 1}} \frac{1}{n} \sum_{i = 1}^n f_i(\bx, \by)
\end{align*}
with $\R^{d_1} \otimes \R^{d_2 d_3}$ as the parameter space and the centered Gaussian distribution $\cN(0, \sigma_1^2 I_{d_1}) \otimes \cN(0, \sigma_2^2 I_{d_2 d_3})$ as the prior $\mu$, where $\sigma_1, \sigma_2$ will be defined in the sequel. Consider random vectors $\bxi, \bfeta$ with mutual distribution $\rho_{\bx, \by}$ such that $\E \bxi \bfeta^\top = \bx \by^\top$. Since $f_{i}(\bx, \by), f_{\bX}(\bx, \by)$ are linear in $\bx\by^\top$, we have $\E_{\rho_{\bx, \by}} f_i(\bxi, \bfeta) = f_i(\bx, \by)$, so Lemma~\ref{lem:pac-bayes} yields
\begin{align}
\label{eq: PAC-Bayes upper bound}
     \sup_{\substack{\bx  \in \bbS^{d_1 - 1} \\ \by \in \bbS^{d_2 d_3 - 1}}} \frac{1}{n} \sum_{i = 1}^n f_i(\bx, \by) & \le \sup_{\substack{\bx  \in \bbS^{d_1 - 1} \\ \by \in \bbS^{d_2 d_3 - 1}}} \bigg \{ \E_{\rho_{\bx,\by}} \log \E_{\bX} \exp f_{\bX}(\bxi, \bfeta) \nonumber \\
     & \qquad \quad \qquad \quad +  \frac{\KL(\rho_{\bx, \by}, \mu) + \log(1/\delta)}{n} \bigg \}
\end{align} 
with probability at least $1 - \delta$. Then, we construct $\rho_{\bx,\by}$ such that the right-hand side of the above inequality can be controlled efficiently.

\noindent \textbf{Step 2. Constructing $\rho_{\bx, \by}$.} Suppose for a while that  $\rho_{\bx, \by}$-almost surely we have 
\begin{align}
\label{eq: lambda condition}
\lambda \Vert \Sigma^{1/2} \cR_1^{-1}(\bxi \bfeta^\top)  \Sigma^{1/2} \Vert_{\F} \le 1/\omega.
\end{align}
Then, Assumption~\ref{as:orlicz} implies
\begin{align}
    \E_{\rho_{\bx, \by}} \log \E_{\bX} \exp f_{\bX}(\bxi, \bfeta) & =  \E_{\rho_{\bx, \by}} \log \E_{\bX} \exp \left \{ \lambda \left ( \bX^\top \cR_1^{-1}(\bx \by^\top) \bX - \E \bX^\top \cR_1^{-1}(\bx \by^\top) \bX \right ) \right \} \nonumber \\
    & \le \lambda^2 \omega^2  \E_{\rho_{\bx, \by}} \Vert \Sigma^{1/2} \cR_1^{-1}(\bxi \bfeta^\top) \Sigma^{1/2} \Vert_{\F}^2. 
\label{eq: hanson-wright bound}
\end{align}
So, to control the above and keep the left-hand side of~\eqref{eq: lambda condition} bounded, we do the following. Define independent random vectors $G_1 \sim \cN(0, \sigma_1^2 I_{d_1}), G_2 \sim \cN(0, \sigma_2^2 I_{d_2 d_3})$, and consider a function
\begin{align}
\label{eq: definition of g}
    g(\bx', \by') = \Vert \Sigma^{1/2} \cR_1^{-1}(\bx' (\by')^\top)  \Sigma^{1/2} \Vert_{\F}.
\end{align}
By the triangle inequality, we have
\begin{align*}
    g(\bx + G_1, \by + G_2) \le g(\bx, \by) + g(\bx, G_2) + g(G_1, \by) + g(G_1, G_2),
\end{align*}
so
\begin{align*}
    g^2(\bx + G_1, \by + G_2) \le 4 g^2(\bx, \by) + 4 g^2(\bx, G_2) + 4 g^2 (G_1, \by) + 4 g^2(G_1, G_2).
\end{align*}
Then, the distribution $\rho_{\bx, \by}$ of the random vector $(\bxi, \bfeta)$ is equal to the distribution of $(\bx + G_1, \by + G_2)$ subject to the condition
\begin{align*}
    (G_1, G_2) \in \Upsilon = 
    \begin{Bmatrix} & 
    g^2(\bx, G_1) \le 4 \E g^2(\bx, G_1') \\ 
    (G_1, G_2): & g^2(G_1, \by) \le 4 \E g^2(G_1',\by) \\
    & g^2(G_1, G_2) \le 4 \E g^2(G_1', G_2'),
    \end{Bmatrix},
\end{align*}
where $G_1' \sim \cN(0, \sigma_1^2 I_{d_1}), G_2' \sim \cN(0, \sigma_2^2 I_{d_2})$ are independent copies of $G_1, G_2$.
Note that by the union bound and the Markov inequality, we have
\begin{align}
    \Pr \left ( (G_1, G_2) \not \in \Upsilon \right ) & \le \sum_{(a, b) \in (\{\bx, G_1\} \times \{\by, G_2\} )\setminus \{(\bx, \by)\} } \Pr \left (g^2(a, b) > 4 \E g^2(a, b) \right ) \nonumber \\
    & \le \sum_{(a, b) \in (\{\bx, G_1\} \times \{\by, G_2\} )\setminus \{(\bx, \by)\} } \frac{1}{4} = \frac{3}{4}. \label{eq: prob of Upsilon complement}
\end{align}
Let us check, that $\E_{\rho_{\bx, \by}} \bxi \bfeta^\top = \bx \by^\top$. Since the Gaussian distribution is centrally symmetric and the function $g$ does not change its value when multiplying any of its argument by $-1$, we have 
\begin{align}
(\bxi, \bfeta) \overset{d}{=} (\bx + \eps_1 (\bxi - \bx), \by + \eps_2 (\bfeta - \by)), \label{eq: Rademacher symmetrization}
\end{align}
where $\eps_1, \eps_2$ are i.i.d. Rademacher ramdom variables independent of $(\bxi, \bfeta)$. Then, we obtain
\begin{align*}
    \E \bxi \bfeta^\top = \bx \by^\top + \E \eps_1 \E (\bxi - \bx) \by^\top + \E \eps_2  \E \bx (\bfeta - \by)^\top + \E \eps_1 \E \eps_2 \E (\bxi - \bx) (\bfeta - \by)^\top = \bx \by^\top. 
\end{align*}
Hence, to satisfy the assumption~\eqref{eq: lambda condition} and use~\eqref{eq: hanson-wright bound}, it is enough to bound expectations $\E g^2(a, b)$ for $(a, b) \in \{\bx, G_1\} \times \{\by, G_2\}$.

\noindent \textbf{Step 3. Bounding expectations $\E g^2(\cdot, \cdot)$.} Let us start with $g^2(\bx, \by)$. From the definition~\eqref{eq: definition of g}, we have
\begin{align}
    g^2(\bx, \by) = \Vert \Sigma^{1/2} \cR_1^{-1}(\bx \by^\top) \Sigma^{1/2} \Vert_{\F}^2 & = \Tr(\Sigma^{1/2} \cR_1^{-1}(\bx \by^\top) \Sigma \cR_1^{-\top}(\bx \by^\top) \Sigma^{1/2}) \nonumber \\
    & = \Tr(\Sigma \cR_1^{-1}(\bx \by^\top) \Sigma \cR^{-\top}_1(\bx \by^\top)) \label{eq: trace form of g^2}
\end{align}
Since $\Tr(A B) \le \Vert A \Vert_{\F} \Vert B \Vert_{\F}$ for any matrices $A, B$, we have
\begin{align*}
    g^2(\bx, \by) \le \Vert \Sigma \cR^{-1}_1(\bx \by^\top) \Vert_{\F} \Vert \Sigma \cR^{-\top}_1(\bx \by^\top) \Vert \le \Vert \Sigma \Vert^2 \Vert \bx \by^\top \Vert^2_{\F} = \Vert \Sigma \Vert,
\end{align*}
where we used the fact that $\cR^{-1}_1(\cdot)$ does not change the Frobenius norm and that $\Vert \bx \by^\top \Vert_{\F} = \Vert \bx \Vert \Vert \by \Vert = 1$.

It will be convenient for future purposes to rewrite~\eqref{eq: trace form of g^2} in a slightly different form. We introduce the following tensors, that are reshapings of the matrix $\Sigma$ and vectors $\bx, \by, G_1, G_2$:
\begin{align*}
    \cS_{p_1q_1r_1p_2q_2r_2} = \Sigma_{(p_1 - 1) qr + (q_1 - 1) r + r_1, (p_2 - 1) qr + (q_2 - 1) r + r_2}, \\
    \cG^{(1)}_{p_2p_3} = (G_1)_{(p_2 - 1) \cdot p + p_3}, \quad \cG^{(2)}_{q_2 q_3 r_2 r_3} = (G_2)_{(q_2 - 1) qr^2 + (q_3 - 1) r^2 + (r_2 - 1) r + r_3}, \\
    \ttx_{p_2 p_3} = \bx_{(p_2 - 1) p + p_3}, \quad \tty_{q_2 q_3 r_2 r_3} = \by_{(q_2 - 1) q r^2 + (q_3 - 1)r^2 + (r_2 - 1)r + r_3}.
\end{align*}

Following the Einstein notation, we obtain
\begin{align}
    g^2(\bx, \by)  & =  \Tr(\Sigma \cR^{-1}_1(G_1 \by^\top) \Sigma \cR_1^{-\top}(\bx \by^\top)) \nonumber \\
    & =  \Sigma_{(p_1 - 1) qr + (r_1 - 1) r + r_1, (p_2 - 1) qr + (q_2 - 1) r + r_2} \nonumber \\
    & \quad \times (\bx \by)^{\top}_{(p_2 - 1) p + p_3, (q_2 -1 )qr^2 + (q_3 - 1) r^2 + (r_2 - 1)r + r_3} \nonumber \\
    & \quad \times \Sigma_{(p_3 - 1) qr + (q_3 - 1) r + r_3, (p_4 - 1) qr + (q_4 - 1) r + r_4} \nonumber \\
    & \quad \times (\bx \by)^\top_{(p_1 - 1) p + p_4, (q_1 - 1) qr^2 + (q_4 - 1)r^2 + (r_1 - 1) r + r_4}. \nonumber \\
    & = \cS_{p_1 q_1 r_1 p_2 q_2 r_2} \ttx_{p_2 p_3} \tty_{q_2 q_3 r_2 r_3} \cS_{p_3 q_3r_3 p_4 q_4 r_4} \ttx_{p_1 p_4} \tty_{q_1 q_4 r_1 r_4} \label{eq: einstein form of g^2}
\end{align}
Note that the above holds for any $\bx \in \R^{d_1}, \by \in \R^{d_2 d_3}$.

Then, we bound $\E g^2(G_1, \by)$. Following~\eqref{eq: einstein form of g^2}, we get
\begin{align*}
    \E g^2(G_1,\by) & = \E \cS_{p_1 q_1 r_1 p_2 q_2 r_2} \cG^{(1)}_{p_2 p_3} \tty_{q_2 q_3 r_2 r_3} \cS_{p_3q_3r_3p_4q_4r_4} \cG^{(1)}_{p_1 p_4} \tty_{q_1 q_4 r_1 r_4} \\
    & = \sigma_1^2 \delta_{p_2 p_1} \delta_{p_3 p_4} \cS_{p_1 q_1 r_1 p_2 q_2 r_2} \tty_{q_2 q_3 r_2 r_3} \cS_{p_3 q_3 r_3 p_4 q_4 r_4} \tty_{q_1 q_4 r_1 r_4} \\
    & = \sigma_1^2 \cS_{p_1 q_1 r_1 p_1 q_2 r_2} \tty_{q_2 q_3 r_2 r_3} \cS_{p_3 q_3 r_3 p_3 q_4 r_4} \tty_{q_1 q_4 r_1 r_4}
\end{align*}
where $\delta$ is the Kronecker delta symbol.
The above can be rewritten as the following trace:
\begin{align*}
    \E g^2(G_1, \by) = \sigma_1^2 \cdot \Tr(\Tr_1(\Sigma) Y \Tr_1(\Sigma) Y^\top),
\end{align*}
where entries of the matrix $Y$ are defined by $Y_{(q_2 - 1) r + r_2, (q_3 - 1) r + r_3} = \tty_{q_2q_3r_2r_3}$. Then, we have
\begin{align*}
    \E g^2(G_1, \by) \le \sigma_1^2 \Vert \Tr_1(\Sigma) Y \Vert_{\F} \cdot \Vert \Tr_1(\Sigma) Y^\top \Vert_{\F} \le \sigma_1^2 \Vert \Tr_1(\Sigma) \Vert^2 \cdot \Vert Y \Vert_{\F}^2 = \sigma_1^2 \Vert \Tr_1(\Sigma) \Vert.
\end{align*}

Next, we bound $\E g^2(\bx, G_2)$. Using~\eqref{eq: einstein form of g^2}, we derive
\begin{align*}
    \E g^2(\bx, G_2) & = \E \cS_{p_1 q_1 r_1 p_2 q_2 r_2} \ttx_{p_2 p_3} \cG^{(2)}_{q_2 q_3 r_2 r_3} \cS_{p_3 q_3 r_3 p_4 q_4 r_4} \ttx_{p_1 p_4} \cG^{(2)}_{q_1 q_4 r_1 r_4}\\
    & = \sigma_2^2  \delta_{q_2 q_1} \delta_{q_3 q_4} \delta_{r_2 r_1} \delta_{r_3 r_4} \cS_{p_1 q_1 r_1 p_2 q_2 r_2} \ttx_{p_2 p_3} \cS_{p_3 q_3 r_3 p_4 q_4 r_4} \ttx_{p_1 p_4} \\
    & = \sigma_2^2 \cdot \Tr (\Tr_{2,3}(\Sigma) X \Tr_{2,3}(\Sigma) X^\top),
\end{align*}
where entries of the matrix $X$ are defined by $X_{p_2,p_3} = \ttx_{p_2p_3}$. Then, we have
\begin{align*}
    \E g^2 (\bx, G_2) \le \sigma_2^2 \Vert \Tr_{2,3}(\Sigma) X \Vert_{\F} \cdot \Vert \Tr_{2,3}(\Sigma) X^\top \Vert_{\F} \le \sigma_2^2 \Vert \Tr_{2,3}(\Sigma) \Vert \cdot \Vert X \Vert_{\F}^2 = \sigma_2^2 \cdot \Vert \Tr_{2,3}(\Sigma) \Vert^2.
\end{align*}
Finally, we bound $\E g^2(G_1, G_2)$. Using~\eqref{eq: einstein form of g^2}, we get
\begin{align*}
    \E g^2(G_1, G_2) & = \E \cS_{p_1 q_1 r_1 p_2 q_2 r_2} \cG^{(1)}_{p_2 p_3} \cG^{(2)}_{q_2 q_3 r_2 r_3} \cS_{p_3 q_3 r_3 p_4 q_4 r_4} \cG^{(1)}_{p_1 p_4} \cG^{(2)}_{q_1 q_4 r_1 r_4} \\
    & = \sigma_1^2 \sigma_2^2 \delta_{p_1 p_2} \delta_{p_3 p_4} \delta_{q_1 q_2} \delta_{q_3 q_4} \delta_{r_1 r_2} \delta_{r_3 r_4} \cS_{p_1 q_1 r_1 p_2 q_2 r_2} \cS_{p_3 q_3 r_3 p_4 q_4 r_4}\\
    & = \sigma_1^2 \sigma_2^2 \cdot \Tr^2(\Sigma).
\end{align*}
Hence, we have $\rho_{\bx, \by}$-almost surely:
\begin{align*}
    g(\bxi, \bfeta) & \le 2 \sqrt{\Vert \Sigma \Vert^2 + \sigma_1^2 \Vert \Tr_1(\Sigma) \Vert^2 + \sigma_2^2 \Vert \Tr_{2,3}(\Sigma) \Vert^2 + \sigma_1^2 \sigma_2^2 \Tr^2(\Sigma)}.
\end{align*}
Set $\sigma_1^2 = \ttr_1^{-2}(\Sigma)$ and $\sigma_2^2 =\ttr_2^{-2}(\Sigma) \ttr_3^{-2}(\Sigma)$. By the definition of $\ttr_i(\Sigma)$, for this choice of $\sigma_1, \sigma_2$, the function $g(\bxi, \bfeta)$ is bounded by $4 \Vert \Sigma \Vert$ almost surely.
Thus, using~\eqref{eq: lambda condition} and~\eqref{eq: hanson-wright bound}, we deduce that for any $\lambda$ satisfying
\begin{align*}
    \lambda \le (4 \omega \Vert \Sigma \Vert)^{-1},
\end{align*}
we have
\begin{align}
    \E_{\rho_{\bx, \by}} \log \E_{\bX} \exp f_{\bX}(\bxi, \bfeta) & \le \lambda^2 \omega^2 \cdot \E_{\rho_{\bx, \by}} g^2(\bxi, \bfeta)  \le 16 \lambda^2 \omega^2 \Vert \Sigma \Vert^2. \label{eq: final bound on log E exp f}    
\end{align}

Due to~\eqref{eq: PAC-Bayes upper bound}, it remains to bound the Kullback-Leibler divergence $\KL(\rho_{\bx, \by}, \mu)$. 

\noindent \textbf{Step 4. Bounding the Kullback-Leibler divergence.} The density of $\rho_{\bx, \by}$ is given by
\begin{align*}
    \rho_{\bx, \by}(x, y) & = \frac{(2\pi)^{-(d_1 + d_2 d_3)/2} \sigma_1^{-d_1} \sigma_2^{-d_2 d_3}}{\Pr((G_1, G_2 \in \Upsilon)}  \exp \left \{ -\frac{1}{2\sigma_1^2} \Vert x - \bx \Vert^2 - \frac{1}{2\sigma_2^2} \Vert y - \by \Vert^2 \right \} \\
    & \qquad \times \1\{(x - \bx, y - \by) \in \Upsilon\}.
\end{align*}
The density of the prior $\mu$ is given by
\begin{align*}
    \mu(x, y) & = \frac{(2\pi)^{-(d_1 + d_2 d_3)/2}}{\sigma_1^{d_1} \sigma_2^{d_2 d_3}}  \exp \left \{ -\frac{1}{2\sigma_1^2} \Vert x \Vert^2 - \frac{1}{2\sigma_2^2} \Vert y \Vert^2 \right \}.
\end{align*}
Then, the KL-divergence can be computed as follows:
\begin{align*}
    \KL(\rho_{\bx, \by}, \mu) & = \int_{\R^{d_1 \times d_2 d_3}} \rho_{\bx, \by}(x, y) \log \frac{\rho_{\bx, \by}(x, y)}{\mu(x, y)} dx dy \\
    & = \log \frac{1}{\Pr((G_1, G_2) \in \Upsilon)} \\
    & \quad + \int_{\R^{d_1 \times d_2 d_3}} \rho_{\bx, \by}(x, y) \left \{ -\frac{1}{2\sigma_1^2} (\Vert x - \bx \Vert^2 - \Vert x \Vert^2) - \frac{1}{2\sigma_2^2} (\Vert y - \by \Vert^2 - \Vert y \Vert^2)  \right \} dx dy.
\end{align*}
Due to~\eqref{eq: prob of Upsilon complement}, the first term is bounded by $\log 4$. Note that the second term is equal to:
\begin{align*}
     - \frac{\Vert \bx \Vert^2}{2 \sigma_1^2} + \frac{2}{2\sigma_1^2}\langle \E_{\rho_{\bx, \by}} \bxi, \bx \rangle - \frac{\Vert \by \Vert^2}{2 \sigma_2^2} + \frac{2}{2\sigma_2^2} \langle \E_{\rho_{\bx, \by}} \bfeta, \by \rangle.
\end{align*}
Using~\eqref{eq: Rademacher symmetrization}, we get
\begin{align*}
    \E_{\rho_{\bx, \by}} \bxi & = \bx + \E \eps_1 \E (\bxi - \bx) = \bx, \\
    \E_{\rho_{\bx, \by}} \bfeta & = \by + \E \eps_2 \E (\bfeta - \by) = \by,
\end{align*}
so we have
\begin{align*}
    \KL(\rho_{\bx, \by}, \mu) & \le \log 4 + \frac{\Vert \bx \Vert_2^2}{2\sigma_1^2} + \frac{\Vert \by \Vert_2^2}{2 \sigma_2^2} = \log 4 +  \ttr_1^2(\Sigma) / 2 +  \ttr_2^2(\Sigma) \ttr_3^2(\Sigma) / 2.
\end{align*}
\noindent \textbf{Step 5. Final bound.} Substituting the above bound and bound~\eqref{eq: final bound on log E exp f} into~\eqref{eq: PAC-Bayes upper bound2} and using
\begin{align*}
    \Vert \ttm_1(\widehat{\cE}) \Vert = \frac{1}{\lambda} \sup_{\substack{\bx  \in \bbS^{d_1 - 1} \\ \by \in \bbS^{d_2 d_3 - 1}}} \frac{1}{n} \sum_{i = 1}^n f_i(\bx, \by),
\end{align*}
we get
\begin{align*}
    \Vert \ttm_1(\widehat{\cE}) \Vert & \le 16 \lambda \omega^2 \Vert \Sigma \Vert^2 + \frac{\ttr_1^2(\Sigma) / 2 + \ttr_2^2(\Sigma) \ttr_3^2(\Sigma) / 2 + \log(4/\delta)}{\lambda n}
\end{align*}
for any positive $\lambda \le (4 \omega \Vert \Sigma \Vert)^{-1}$ with probability at least $1 - \delta$. Since $n \ge \ttr_1^2(\Sigma) + \ttr_2^2(\Sigma)\ttr_3^2(\Sigma) + \log(4/\delta)$, we choose
\begin{align*}
    \lambda =  (4 \omega \Vert \Sigma \Vert)^{-1} \sqrt{\frac{\ttr_1^2(\Sigma) / 2 + \ttr_2^2(\Sigma) \ttr_3^2(\Sigma) / 2 + \log(4/\delta)}{n}},
\end{align*} 
and get
\begin{align*}
    \Vert \ttm_1(\widehat{\cE}) \Vert & \le 8 \omega \Vert \Sigma \Vert \sqrt{\frac{\ttr_1^2(\Sigma)/2+ \ttr_2^2(\Sigma) \ttr_3^2(\Sigma)/2 + \log(4/\delta)}{n}} \\
    & \le 32 \omega \Vert \Sigma \Vert \sqrt{\frac{\ttr_1^2(\Sigma) +  \ttr_2^2(\Sigma) \ttr_3^2(\Sigma) + \log(1/\delta)}{n}}.
\end{align*}
\end{proof}

\subsection{Proof of Lemma~\ref{lemma: sup U times noise}}
\label{section: proof of lemma sup U times noise}

\begin{proof}
We deduce Lemma~\ref{lemma: sup U times noise} from the following theorem. Its proof is posteponed to Section~\ref{section: proof of main concentration theorem}.

\begin{Th}
\label{theorem: concentration on the empirical process}
Let $\bbS_1, \bbS_2, \bbS_3$ be sets of linear operators 
\begin{align*}
    \bbS_i & \subset \left \{ A_i :  L_i \to \R^{d_i}, \text{ such that } \Vert A_i\Vert \le 1\right\}, i = 1, 3, \\
    \bbS_2 & \subset \left  \{ A \in L_1 \otimes \R^{d_2} \otimes L_3  \text{ such that } \Vert A \Vert_{\F} \le 1 \right \}.
\end{align*}
For brevity, put $L_2 = L_1 \otimes L_3$. Denote $\dim L_i ~\text{as}~ l_i$.
Then, we have
\begin{align*}
   \sup_{\substack{A_1 \in \bbS_1, \\ A_2 \in \mathbb{S}_{2}, A_3 \in \bbS_3}} \langle \widehat{\cE} \times_3 A_3^\top \times_1 A_1^\top, A_2 \rangle \le 2^7 \omega \Vert \Sigma \Vert \sqrt{\frac{\sum_{i = 1}^3 \min \{\ttr_i^2(\Sigma) \cdot l_i, \log |\bbS_i| \} + \log (8/\delta)}{n}}
\end{align*}
with probability at least $1 - \delta$, provided $n \ge \sum_{i = 1}^3 \min \{\ttr_i^2(\Sigma) \cdot l_i, \log |\bbS_i| \} + \log (8/\delta)$. Here we assume that $\min \{\ttr_i(\Sigma) \cdot l_i, \log |\bbS_i| \} = \ttr_i(\Sigma) \cdot l_i$ if $\bbS_i$ is infinite.
\end{Th}

Note that 
\begin{align*}
    \sup_{\substack{U \in \R^{d_1 \times J} \\ \Vert U \Vert \le 1}} \Vert \ttm_3(\widehat{\cE}) (U \otimes I_{d_2})\Vert & = \sup_{\substack{U \in \R^{d_1 \times J} \\ \Vert U \Vert \le 1}}\Vert \ttm_3(\widehat{\cE} \times_1 U^\top) \Vert \\
    & = \sup_{\substack{\bx \in \R^{d_3}, \by \in \R^{J d_2}, U \in \R^{d_1 \times J} \\ \Vert \bx \Vert \le 1, \Vert \by \Vert \le 1, \Vert U \Vert \le 1}} \bx^{\top} \ttm_3(\widehat{\cE} \times_1 U^\top) \by.
\end{align*}
can rewritten as the following supremum over scalar product:
\begin{align*}
   \sup_{\substack{A_1 \in \bbS_1, \\ A_2 \in \mathbb{S}_{2}, A_3 \in \bbS_3}} \langle \widehat{\cE} \times_3 A_3^\top \times_1 A_1^\top, A_2 \rangle,
\end{align*}
where 
\begin{align*}
    \bbS_1 & = \{A_1 : \R^{J} \to \R^{d_1} \mid \Vert A_1 \Vert \le 1\}, \\
    \bbS_2 & = \{A_2 \in \R^{J \times d_2 \times 1} \mid \Vert A_2 \Vert_{\F} \le 1\}, \\
    \bbS_3 & = \{A_3 : \R \to \R^{d_3} \mid \Vert A_3 \Vert \le 1 \}.
\end{align*}
Then, Theorem~\ref{theorem: concentration on the empirical process} implies that for any $\delta \in (0,1)$, with probability at least $1 - \delta$, we have
\begin{align*}
     \sup_{\substack{U \in \R^{d_1 \times J} \\ \Vert U \Vert \le 1}} \Vert \ttm_3(\widehat{\cE}) (U \otimes I_{d_2})\Vert \le 2^7 \omega \Vert \Sigma \Vert \sqrt{\frac{J\ttr_1^2(\Sigma) + J \ttr_2^2(\Sigma) + \ttr_3^2(\Sigma) + \log(8/\delta)}{n}},
\end{align*}
if $n \ge J\ttr_1^2(\Sigma) + J \ttr_2^2(\Sigma) + \ttr_3^2(\Sigma) + \log(8/\delta)$.

Analogously, we have 
\begin{align*}
    \sup_{V \in\R^{d_3 \times K}, \Vert V \Vert \le 1} \Vert \ttm_1(\cE) (I_{d_2} \otimes V) \Vert \le 32 \omega \Vert \Sigma \Vert \sqrt{\frac{\ttr_1^2(\Sigma) + K \ttr_2^2(\Sigma) + K\ttr_3^2(\Sigma) + \log(8/\delta)}{n}}
\end{align*}
with probability at least $1 - \delta$, if $n \ge \ttr_1^2(\Sigma) + K \ttr_2^2(\Sigma) + K\ttr_3^2(\Sigma) + \log(8/\delta)$. This completes the proof.
\end{proof}

\subsection{Proof of Lemma~\ref{lemma: U star times noise bound}}
\label{section: proof of lemma U star times noise bound}

\begin{proof}
Note that the norm
\begin{align*}
    \Vert \ttm_1(\widehat{\cE} \times_3 (V^*)^\top) \Vert & = \sup_{\substack{\bx \in \R^{d_1}, \by \in \R^{K d_2} \\ \Vert \bx \Vert \le 1, \Vert \by \Vert \le 1}} \bx^\top \ttm_1(\widehat{\cE} \times_3 (V^*)^\top) \by
\end{align*}
can be rewritten as the following supremum over scalar product:
\begin{align*}
   \sup_{\substack{A_1 \in \bbS_1, \\ A_2 \in \mathbb{S}_{2}, A_3 \in \bbS_3}} \langle \widehat{\cE} \times_3 A_3^\top \times_1 A_1^\top, A_2 \rangle,  
\end{align*}
where 
\begin{align*}
    \bbS_1 & = \{A_1 : \R \to \R^{d_1} \mid \Vert A_1 \Vert \le 1\}, \\
    \bbS_2 & = \{A_2 \in \R^{K \times d_2 \times 1} \mid \Vert A_2 \Vert_{\F} \le 1\}, \\
    \bbS_3 & = \{V^*\}.
\end{align*}
Hence, Theorem~\ref{theorem: concentration on the empirical process} implies that for any $\delta \in (0,1)$, with probability at least $1 - \delta$, we have
\begin{align*}
     \Vert \ttm_1( \widehat{\cE} \times_3 (V^*)^\top) \Vert \le 32 \omega \Vert \Sigma \Vert \sqrt{\frac{\ttr_1^2(\Sigma) + K \ttr_2^2(\Sigma) + \log(8/\delta)}{n}},
\end{align*}
if $n \ge \ttr_1^2(\Sigma) + K \ttr_2^2(\Sigma) + \log(8/\delta)$.
Analogously, we have
\begin{align*}
    \Vert \ttm_3( \widehat{\cE} \times_1 (U^*)^\top) \Vert \le 32 \omega \Vert \Sigma \Vert \sqrt{\frac{\ttr_3^2(\Sigma) + J \ttr^2_2(\Sigma) + \log(8/\delta)}{n}},
\end{align*}
with probability at least $1 - \delta$, if $n \ge   J \ttr_2^2(\Sigma) + \ttr_3^2(\Sigma)+ \log(8/\delta)$. This completes the proof.
\end{proof}

\subsection{Proof of Lemma~\ref{lemma: leading term bound}}
\label{section: proof of lemma leading term bound}

\begin{proof}
Using the variational representation of the Frobenius norm, we observe that 
\begin{align*}
    \sup_{U \in \bbO_{d_1, J}, V \in \bbO_{d_2, K}} \Vert \widehat{\cE} \times_3 V^\top \times_1 U^\top \Vert_{\F} = \sup_{\substack{U \in \bbO_{d_1, J}, V \in \bbO_{d_2, K} \\ W \in \R^{J \times d_2 \times K}, \Vert W \Vert_{\F} \le 1}} \langle \widehat{\cE} \times_3 V^\top \times_1 U^\top, W  \rangle.
\end{align*}
Then, we apply Theorem~\ref{theorem: concentration on the empirical process} with $\bbS_1 = \bbO_{d_1, J}, \bbS_2 = \{W \in \R^{J \times d_2 \times K} : \Vert W \Vert_{\F} \le 1\}, \bbS_3 = \bbO_{d_3, K}$ and get the desired result.
\end{proof}

\section{Proof of Theorem~\ref{theorem: sensetivity analysis}}
\label{section:proof of sensitivity analysis theorem}

\begin{proof}[Proof of Theorem~\ref{theorem: sensetivity analysis}]
The proof follows that of Theorem~1 by~\cite{zhang2018tensor}. For clarity, we divide it into several steps.

\noindent \textbf{Step 1. Reduction to spectral norm of random matrices.} We have
\begin{align}
    \Vert \widehat{\cT} - \cT^* \Vert_{\F}^2 & = \Vert \widehat{\cW} \times_3 \widehat{V} \times_1 \widehat{U} - \cW^* \times_3 V^* \times_1 U^* \Vert_{\F}^2 \nonumber \\
    & = \Vert \widehat{\cW} \times_3 \widehat{V} \times_1 \widehat{U} -  \cW^* \times_3 V^* \times_1 (\widehat{U} \widehat{U}^\top) U^* \Vert_{\F}^2 + \Vert \cW^* \times_3 V^* \times_1 (I - \Pi_{\widehat U}) U^* \Vert_{\F}^2 \nonumber \\
    & = \Vert \widehat{\cW} \times_3 \widehat{V} - \cW^* \times_3 V^* \times_1 (\widehat{U}^\top U^*) \Vert_{\F}^2 +  \Vert \cW^* \times_3 V^* \times_1 (I - \Pi_{\widehat U}) U^* \Vert_{\F}^2 \nonumber\\
    & = \Vert \widehat{\cW} - \cW^* \times_3 (\widehat{V}^\top V^*) \times_1 (\widehat{U}^\top U^*) \Vert_{\F}^2 + \Vert \cW^* \times_3 (I - \Pi_{\widehat{V}}) V^* \times_1 (\widehat{U}^{\top} U^*) \Vert_{\F}^2 \nonumber \\
    & \quad + \Vert \cW^* \times_3 V^* \times_1 (I - \Pi_{\widehat U}) U^* \Vert_{\F}^2. \label{eq: frobenius norm initial bound}
\end{align}
By the construction of $\widehat{\cW}$, the first term is equal to
\begin{align}
    \Vert \cY \times_3 \widehat{V}^\top \times_1 \widehat{U}^\top - \cT^* \times_3 \widehat{V}^\top \times_1 \widehat{U}^\top \Vert_{\F}^2 = \Vert \cE \times_3 \widehat{V}^\top \times_1 \widehat{U}^\top \Vert_{\F}^2. \label{eq: leading term withoput sup-out}
\end{align}
We rewrite the second term as follows:
\begin{align*}
    \Vert \cW^* \times_3 (I - \Pi_{\widehat{V}}) V^* \times_1 (\widehat{U}^\top U^*) \Vert_{\F} = \Vert(I - \Pi_{\widehat{V}}) \ttm_3(\cT^* \times_1 \widehat{U}^\top) \Vert_{\F}.
\end{align*}
 Due to~\eqref{eq: matricization-tensor idenities}, we have $\ttm_3(\cT^* \times_1 \widehat{U}^\top) = \ttm_3(\cT^*)(\widehat{U} \otimes I_{d_2})$, so $\ttm_3(\cT^* \times_1 \widehat{U}^\top)$ has rank at most $K$ and
\begin{align*}
    \Vert (I - \Pi_{\widehat{V}}) \ttm_3(\cT^*) (\widehat{U} \otimes I_{d_2}) \Vert_{\F} & \le \sqrt{K} \Vert (I - \Pi_{\widehat{V}}) \ttm_3(\cT^*) (\widehat{U} \otimes I_{d_2} )\Vert \\
    & = \sqrt{K} \Vert (I - \Pi_{\widehat{V}}) \ttm_3(\cT^* \times_1 \widehat{U}^\top) \Vert \\
    & \le \sqrt{K} \Vert (I - \Pi_{\widehat{V}}) \ttm_3(\cY \times \widehat{U}^\top) \Vert + \sqrt{K} \Vert (I - \Pi_{\widehat{V}}) \ttm_3(\cE \times_1 \widehat{U}_1^\top) \Vert.
\end{align*}
Since $\widehat{V}$ consists of $K$ leading left singular vectors of $\ttm_3(\cY \times_1 \widehat{U})$ and $\ttm_3(\cT^* \times_1 \widehat{U}_1^\top)$ has rank $K$, we have $\Vert (I - \Pi_{\widehat{V}}) \ttm_3(\cY \times \widehat{U}_1)\Vert = \sigma_{K + 1}(\ttm_3(\cY \times_1 \widehat{U}_1)) \le \Vert \ttm_3(\cE \times \widehat{U}_1) \Vert$ by the Weyl inequality . It yields
\begin{align}
    \Vert \cW^* \times_3 (I - \Pi_{\widehat{V}}) V^* \times_1 (\widehat{U}^\top U^*) \Vert_{\F} \le  2 \sqrt{K} \Vert \ttm_3(\cE \times_1 \widehat{U} ^\top) \Vert. \label{eq: second term of intial decomposition random sinfluar values bound}
\end{align}

Then, we bound the third term of~\eqref{eq: frobenius norm initial bound}. We have
\begin{align*}
    \Vert \cW^* \times_3 V^* \times_1 (I - \Pi_{\widehat{U}})U^* \Vert_{\F} & = \Vert \cW^* \times_1 (I - \Pi_{\widehat{U}}) U^* \Vert_{\F} \\
    & \le \sigma^{-1}_{\min}(\widehat{V}^\top_{T - 1} V^*) \Vert \cW^*) \times_3 (\widehat{V}^\top_{T - 1} V^*) \times_1 (I - \Pi_{\widehat{U}}) U^* \Vert_{\F} \\
    & = \sigma^{-1}_{\min}(\widehat{V}^\top_{T - 1} V^*) \Vert(I - \Pi_{\widehat{U}}) \ttm_1(\cT^* \times_3 \widehat{V}^\top_{T - 1}) \Vert_{\F}.
\end{align*}
The matrix $\ttm_1(\cT^* \times_3 \widehat{V}^\top_{T - 1}) = \ttm_1(\cT^*) (I_{d_2} \otimes \widehat{V}_{T - 1})$ has rank at most $J$, so
\begin{align*}
    \Vert(I - \Pi_{\widehat{U}}) \ttm_1(\cT^* \times_3 \widehat{V}^\top_{T - 1}) \Vert_{\F} & \le \sqrt{J} \Vert (I - \Pi_{\widehat{U}}) \ttm_1(\cT^* \times_3 \widehat{V}^\top_{T - 1}) \Vert \\
    & \le \sqrt{J} \Vert (I - \Pi_{\widehat{U}}) \ttm_1(\cY \times_3 \widehat{V}^\top_{T - 1}) \Vert + \sqrt{J} \Vert (I - \Pi_{\widehat{U}}) \ttm_1(\cE \times_3 \widehat{V}^\top_{T - 1}) \Vert.
\end{align*}
Since $\widehat{U}$ consists of $J$ leading left singular vectors of $\ttm_1(\cY \times_3 \widehat{V}^\top_{T - 1})$ and $\ttm_1(\cT^* \times_3 \widehat{V}^\top_{T - 1})$ has the rank at most $J$, we have $\Vert (I - \Pi_{\widehat{U}}) \ttm_1(\cY \times_3 \widehat{V}^\top_{T - 1}) \Vert = \sigma_{J + 1}(\ttm_1(\widehat{\cY} \times_3 \widehat{V}^\top_{T - 1})) \le \Vert \ttm_1(\cE) \times_3 \widehat{V}^\top_{T - 1})\Vert$ by the Weyl inequality. It implies
\begin{align*}
    \Vert \cW^* \times_3 V^* \times_1 (I - \Pi_{\widehat{U}})U^* \Vert_{\F} \le \frac{2 \sqrt{J}}{\sigma_{\min}(\widehat{V}^\top_{T - 1} V^*)} \Vert \ttm_1(\cE \times_3 \widehat{V}^\top_{T - 1}) \Vert.
\end{align*}
Combining~\eqref{eq: frobenius norm initial bound} with~\eqref{eq: leading term withoput sup-out},~\eqref{eq: second term of intial decomposition random sinfluar values bound} and the above display, we get
\begin{align}
    \Vert \widehat{\cT} - \cT^* \Vert_{\F}^2 & \le  \Vert \cE \times_3 \widehat{V}^\top \times_1 \widehat{U}^\top \Vert_{\F}^2 + 4K \Vert \ttm_3(\cE \times_1 \widehat{U} ^\top) \Vert^2 \nonumber \\
    & \quad + \frac{4 J}{\sigma_{\min}^2(\widehat{V}^\top_{T - 1} V^*)} \Vert \ttm_1(\cE \times_3 \widehat{V}^\top_{T - 1}) \Vert \nonumber \\
    & \le \sup_{U \in \bbO_{d_1, J}, V \in \bbO_{d_2,K}} \Vert \cE \times \times_3 V^\top \times_1 U^\top \Vert_{\F}^2 \nonumber \\
    & \quad + 4 K \Vert \ttm_3(\cE \times_1 \widehat{U} ^\top) \Vert^2  + \frac{4 J}{\sigma_{\min}^2(\widehat{V}^\top_{T - 1} V^*)} \Vert \ttm_1(\cE \times_3 \widehat{V}^\top_{T - 1}) \Vert^2.
    \label{eq: step 1 final bound}
\end{align}

\noindent \textbf{Step 2. Bounding $\sigma_{\min}(\widehat{V}_{T - 1}^\top V^*)$, $\Vert \ttm_1(\cE \times_3 \widehat{V}_{T - 1}^\top) \Vert$, $\Vert \ttm_3(\cE \times_1 \widehat{U}^\top) \Vert$.}  To obtain the theorem, we need to bound $\sigma_{\min}(\widehat{V}_{T - 1}^\top \times_3 \cE)$, $\Vert \ttm_1(\cE \times_3 \widehat{V}_{T - 1}^\top) \Vert$, $\Vert \ttm_3(\cE \times_1 \widehat{U}^\top) \Vert$. We start with the latter two norms. We have
\begin{align}
    \Vert \ttm_3(\cE \times_1 \widehat{U}^\top) \Vert = \Vert \ttm_3(\cE) (\widehat{U} \otimes I_{d_2}) \Vert \le \Vert \ttm_3(\cE) (\Pi_{U^*} \widehat{U} \otimes I_{d_2}) \Vert + \Vert \ttm_3(\cE) ((I - \Pi_{U^*}) \widehat{U} \otimes I_{d_2}) \Vert. \label{eq: m_3 bound}
\end{align}
Since $\Pi_{U^*} = U^* (U^*)^\top$, the first term of the above is at most 
\begin{align}
    \Vert \ttm_3(\cE) U^* ((U^*)^\top \widehat{U} \otimes I_{d_2}) \Vert & = \Vert \ttm_3(\cE) (U^* \otimes I_{d_2}) ((U^*)^\top \widehat{U} \otimes I_{d_2}) \Vert \nonumber \\
    & \le \Vert\ttm_3(\cE) (U^* \otimes I_{d_2}) \Vert  \Vert ((U^*)^\top \widehat{U} \otimes I_{d_2}) \Vert \nonumber \\
    & \le \Vert \ttm_3(\cE) (U^* \otimes I_{d_2}) \Vert. \label{eq: leading term m_3 bound}    
\end{align}
For the second term, we have
\begin{align*}
    \Vert \ttm_3(\cE) ((I - \Pi_{U^*}) \widehat{U} \otimes I_{d_2}) \Vert & \le \Vert \ttm_3(\cE) (\frac{(I - \Pi_{U^*})}{\Vert (I - \Pi_{U^*}) \widehat{U} \Vert} \otimes I_{d_2}) \Vert \cdot \Vert (I - \Pi_{U^*}) \widehat{U} \Vert \\
    & \le \sup_{\substack{V \in \R^{d_1 \times J}, \\ \Vert V \Vert = 1} } \Vert \ttm_3(\cE) (V \otimes I_{d_2}) \Vert \cdot \Vert (I - \Pi_{U^*}) \widehat{U} \Vert.
\end{align*}
Then, we have 
$$\Vert (I - \Pi_{U^*})\widehat{U} \Vert = \Vert (I - \Pi_{U^*}) \Pi_{\widehat{U}} \Vert = \Vert (\Pi_{\widehat{U}} - \Pi_{U^*}) \Pi_{\widehat{U}} \Vert \le \Vert \Pi_{\widehat{U}} - \Pi_{U^*} \Vert,$$
where we used $\Img \widehat{U}^\top = \R^K$ and orthogonality of $\widehat{U}$ for the first equality.
To bound the latter norm of the difference, we rely on the following standard proposition, which is proved 
\begin{Prop}
\label{proposition: projector difference via rho}
For two orthogonal matrices $U_1, U_2 \in \bbO_{a,b}$, $a \ge b$, define the following semidistance
\begin{align*}
    \rho(U_1, U_2) = \inf_{O \in \bbO_{b, b}} \Vert U_1 - U_2 O \Vert.
\end{align*}
Then, we have 
\begin{align*}
    \Vert \Pi_{U_1} - \Pi_{U_2} \Vert \le 2 \cdot \rho(U_1, U_2).
\end{align*}
\end{Prop}

The proposition implies
\begin{align*}
    \Vert \ttm_3(\cE) ((I - \Pi_{U^*}) \widehat{U} \otimes I_{d_2} \Vert \le 2 \sup_{\substack{V \in \R^{d_1 \times J} \\ \Vert V \Vert = 1}} \Vert \ttm_3(\cE) (V \otimes I_{d_2}) \Vert \cdot \rho(\widehat{U}, U^*).
\end{align*}
Combining the above with~\eqref{eq: m_3 bound} and~\eqref{eq: leading term m_3 bound}, we get
\begin{align}
\label{eq: noise projector 1 second-order expansion}
    \Vert \ttm_3(\widehat{U} \times_1 \cE) \Vert \le \Vert \ttm_3(\cE) (U^* \otimes I_{d_2}) \Vert + 2 \sup_{\substack{V \in \R^{d_1 \times J}  \\ \Vert V \Vert = 1}} \Vert \ttm_3(\cE) (V \otimes I_{d_2}) \Vert \cdot \rho(\widehat{U}, U^*).
\end{align}
Analogously, we have
\begin{align}
\label{eq: noise projector 3 second-order expansion}
    \Vert \ttm_1(\widehat{V}_{T - 1} \times_3 \cE) \Vert \le \Vert \ttm_1(\cE) (I_{d_2} \otimes V^*) \Vert + 2 \sup_{\substack{V \in \R^{d_3 \times K }  \\ \Vert V \Vert \le 1}} \Vert \ttm_1(\cE) (I_{d_2} \otimes V) \Vert \cdot \rho(\widehat{V}_{T - 1}, V^*). 
\end{align}
Finally, we bound $\sigma_{\min}(\widehat{V}_{T - 1}^\top V^*)$ below. We have
\begin{align*}
    \sigma_{\min}^2(\widehat{V}^\top_{T - 1} V^*) & = \lambda_{\min}((V^*)^\top \widehat{V} \widehat{V}^\top V^*) = \lambda_K(\Pi_{V^*} \Pi_{\widehat{V}_{T - 1}} \Pi_{V^*}),
\end{align*}
where we used the fact that $V^* A (V^*)^\top$ has the same singular values as $A$ for any Hermitian $A \in \R^{K \times K}$.
Since $\Pi_{V^*}\Pi_{\widehat{V}} \Pi_{V^*} = \Pi_{V^*} - \Pi_{V^*} (I - \Pi_{\widehat{V}_{T - 1}})  \Pi_{V^*}= \Pi_{V^*} - \Pi_{V^*} (\Pi_{V^*} - \Pi_{\widehat{V}_{T - 1}}) \Pi_{V^*}$, the Weyl inequality implies
\begin{align*}
    \lambda_K(\Pi_{V^*} \Pi_{\widehat{V}_{T - 1}} \Pi_{V^*}) \ge \lambda_K(\Pi_{V^*}) - \Vert \Pi_{V^*}(\Pi_{V^*} - \Pi_{\widehat{V}_{T - 1}}) \Pi_{V^*} \Vert\ge 1 - \Vert \Pi_{\widehat{V}_{T - 1}} - \Pi_{V^*} \Vert.
\end{align*}
Then, Proposition~\ref{proposition: projector difference via rho} yields $\Vert \Pi_{\widehat{V}_{T - 1}} - \Pi_{V^*} \Vert \le 2 \rho(\widehat{V}_{T - 1}, V^*)\}$, so
\begin{align}
\label{eq: orthogonal product lower bound}
    \sigma_{\min}(\widehat{V}_{T - 1}^\top V^*) \ge \sqrt{1 - 2 \rho (\widehat{V}_{T - 1}, V^*)},
\end{align}
provided $\rho(\widehat{V}_{T - 1}, V^*) \le 1/2$.

\noindent \textbf{Step 3. Bounding $\rho(\widehat{U}_t, U^*)$, $\rho(\widehat{V}_{t}, V^*)$ recursively.} We provide a recursive bound on $\rho(\widehat{U}_t, U^*)$ and $\rho(\widehat{V}_t, V^*)$. We widely use the following lemma, which is a weaker variant of the Wedin $\sin \Theta$--theorem:

\begin{Prop}
\label{proposition: devis-kahan sin-theta with rho}
Let $A, B$ be matrices, such that $A$ has rank $r$, and denote $B = A + E$. Let $L$ be left singular vectors of $A$ and $\widehat{L}$ be $r$ leading left singular vectors of $B$. Then
\begin{align*}
    \rho(L, \widehat{L}) \le \frac{2\sqrt{2} \Vert E \Vert}{\sigma_r(A)}.
\end{align*}
\end{Prop}

By Proposition~\ref{proposition: devis-kahan sin-theta with rho}, we have
\begin{align}
\label{eq: U0 bound}
    \rho(\widehat{U}_0, U^*) \le  \frac{2 \sqrt{2} \Vert \ttm_1(\cE) \Vert}{\sigma_J(\ttm_1(\cT^*)) }.
\end{align}
To bound $\rho(\widehat{V}_t, V^*)$, we note the following. Since $\widehat{V}_t$ are  leading $K$ left singular vectors of $\ttm_3(\cY \times_1 \widehat{U}_t^\top) = \ttm_3(\cT^* \times_1 \widehat{U}_t^\top) + \ttm_3(\cE \times_1 \widehat{U}_t^\top)$, and there exists an orthogonal matrix $O \in \bbO_{K, K}$ such that $V^* O$ are the left singular vectors of $\ttm_3(\cT^* \times_1 \widehat{U}_t^\top) = V^* \ttm_3(\cW^* \times_1 U^*) (\widehat{U}_t \otimes I_{d_2})$, by the definition of $\rho(\cdot, \cdot)$ and Proposition~\ref{proposition: devis-kahan sin-theta with rho}, we have
\begin{align*}
    \rho(\widehat{V}_0, V^*) \le  \frac{2 \sqrt{2}\Vert \ttm_3(\cE \times_1 \widehat{U}_0) \Vert}{\sigma_K(\ttm_3(\cT^* \times \widehat{U}_0^\top))} \quad \text{and} \quad \rho(\widehat{V}_t, V^*) \le \frac{2 \sqrt{2} \Vert \ttm_3(\cE \times_1 \widehat{U}_t) \Vert}{\sigma_K(\ttm_3(\cT^* \times_1 \widehat{U}^\top_t))}
\end{align*}
for $t = 1, \ldots, T$. Let us bound $\rho(\widehat{V}_t, V^*)$ using $\rho(\widehat{U}_t, U^*)$. First, we have
\begin{align}
    \sigma_{K}(\ttm_3(\cT^* \times_1 \widehat{U}_t^\top)) & = \sigma_K(\ttm_3(\cT^*) (\widehat{U}_t \otimes I_{d_2})) = \sigma_K(\ttm_3(\cT^*) (U^* \otimes I_{d_2}) ((U^*)^\top \widehat{U} \otimes I_{d_2}) ) \label{eq: recursive singular value lower bound for U} \\
    & \ge \sigma_K(\ttm_3(\cT^*) (U^* \otimes I_{d_2})) \sigma_{\min}((U^*)^\top \widehat{U}_t) = \nonumber \\
    & = \sigma_{K}(\ttm_3(\cT^*) (\Pi_{U^*} \otimes I_{d_2})) \sigma_{\min}((U^*)^\top\widehat{U}) \ge \sigma_K(\ttm_3(\cT^*)) \cdot \sqrt{1 - 2 \rho(\widehat{U}_t, U^*)}, \nonumber
\end{align}
provided $\rho(\widehat{U}_t, U^*) < 1/2.$
Second, we bound $\Vert \ttm_3(\cE \times_1 \widehat{U}_t^\top) \Vert$. Following the derivation of~\eqref{eq: noise projector 1 second-order expansion}, we obtain
\begin{align*}
    \Vert \ttm_3(\cE \times_1 \widehat{U}_t^\top) \Vert & = \Vert \ttm_3(\cE) (\widehat{U}_t \otimes I_{d_2}) \Vert \\
    & \le\Vert \ttm_3(\cE) (\Pi_{U^*} \otimes I_{d_2}) (\widehat{U}_t \otimes I_{d_2}) \Vert + \Vert \ttm_3(\cE) ((I - \Pi_{U^*}) \otimes I_{d_1}) (\widehat{U}_t \otimes I_{d_2}) \Vert \\
    & \le \Vert \ttm_3(\cE) (U^* \otimes I_{d_2}) \Vert + \sup_{\substack{U \in \R^{d_1 \times J} \\ \Vert U \Vert \le 1}} \Vert \ttm_3(\cE) (U \otimes I_{d_2}) \Vert \cdot \Vert (I - \Pi_{U^*}) \widehat{U}_t \Vert.
\end{align*}
Since $\widehat{U}_t$ is orthogonal, we have $\Vert (I - \Pi_{U^*}) \widehat{U}_t \Vert = \Vert (I - \Pi_{U^*}) \Pi_{\widehat{U}_t} \Vert$, so
\begin{align*}
    \Vert (I - \Pi_{U^*}) \widehat{U}_t \Vert = \Vert (\Pi_{\widehat{U}_t} - \Pi_{U^*}) \Pi_{\widehat{U}_t} \Vert \le \Vert \Pi_{\widehat{U}_t} - \Pi_{U^*} \Vert \le 2 \rho(\widehat{U}_t, U^*),
\end{align*}
due to Proposition~\ref{proposition: projector difference via rho},
and
\begin{align}
\label{eq: product E and Ut bound}
    \Vert \ttm_3(\cE \times_1 \widehat{U}_t^\top) \Vert \le  \Vert \ttm_3(\cE) (U^* \otimes I_{d_2}) \Vert + 2 \sup_{\substack{U \in \R^{d_1 \times J} \\ \Vert U \Vert \le 1}} \Vert \ttm_3(\cE) (U \otimes I_{d_2}) \Vert \cdot \rho(\widehat{U}_t, U^*).
\end{align}
Following the notation of the theorem, we get
\begin{align}
\label{eq: Vt rho recursive bound}
    \rho(\widehat{V}_t, V^*) \le \frac{2 \sqrt{2} \cdot\left  (\alpha_V + 2 \beta_V \cdot  \rho(\widehat{U}_t, U^*) \right )}{\sigma_K(\ttm_3(\cT^*)) \sqrt{1 - 2 \rho(\widehat{U}_t, U^*)}}.
\end{align}

Next, we will bound $\rho(\widehat{U}_t, U^*)$ using $\rho(\widehat{V}_{t - 1}, V^*)$ for $t \ge 1$. Since $\widehat{U}_t$ are leading $J$ left singular vectors of $\ttm_1(\cY \times_3 \widehat{V}_{t - 1}^\top) = \ttm_1(\cT^* \times_3 \widehat{V}_{t - 1}^\top) + \ttm_1(\cE \times_3 \widehat{V}^\top_{t - 1})$, and there exists an orthogonal matrix $O \in \bbO_{J, J}$ such that $U^* O$ are the left singular vectors of $\ttm_1(\cT^* \times_3 \widehat{V}_{t - 1}^\top) = U^* \ttm_1(\cW^* \times_3 V^*) (I_{d_2} \otimes \widehat{V}_{t - 1})$, by Proposition~\ref{proposition: devis-kahan sin-theta with rho} and the definition of $\rho(\cdot, \cdot)$, we have
\begin{align*}
    \rho(\widehat{U}_{t - 1}, U^*) \le \frac{2 \sqrt{2} \Vert \ttm_1(\cE \times_3 \widehat{V}_{t - 1}^\top) \Vert}{\sigma_J(\ttm_1(\cT^* \times_ 3\widehat{V}_{t - 1}))}.
\end{align*}
Analogously to~\eqref{eq: recursive singular value lower bound for U}, we have
\begin{align*}
    \sigma_J(\ttm_1(\cT^* \times_3 \widehat{V}_{t - 1})) \ge \sigma_J(\ttm_1(\cT^*)) \sqrt{1 - 2 \rho(\widehat{V}_{t - 1}, V^*)},
\end{align*}
provided $\rho(\widehat{V}_{t - 1}, V^*) < 1/2$. Analogously to~\eqref{eq: product E and Ut bound}, we have
\begin{align}
\label{eq: product E and Vt bound}
    \Vert \ttm_1(\cE \times_3 \widehat{V}_{t - 1}) \Vert \le \Vert \ttm_1(\cE) (I_{d_2} \otimes V^*) \Vert + 2 \sup_{\substack{V \in \R^{d_1 \times K} \\ \Vert V \Vert \le 1}} \Vert \ttm_1(\cE) (I_{d_2} \otimes V) \Vert \cdot \rho(\widehat{V}_{t - 1}, V^*).
\end{align}
Thus, using the notation of the theorem, we get
\begin{align}
\label{eq: Ut rho recursive bound}
    \rho(\widehat{U}_t, U^*) \le \frac{2 \sqrt{2} \left (\alpha_U + 2 \beta_U \cdot \rho(\widehat{V}_{t - 1}, V^*)\right )}{\sigma_{J}(\ttm_1(\cT^*)) \sqrt{ 1 - 2 \rho(\widehat{V}_{t - 1}, V^*)}}.
\end{align}

\noindent \textbf{Step 4. Solving the recursion.} We claim that for each $t = 0, \ldots, T$, we have
\begin{align}
\label{eq: 1/4 bound on rho}
    \rho(\widehat{U}_t, U^*) \le 1/4 \quad \text{and} \quad \rho (\widehat{V}_t, V^*) \le 1/4.
\end{align}
Let us prove it by induction. From~\eqref{eq: U0 bound} and conditions of the theorem, we have
\begin{align*}
    \rho(\widehat{U}_0, U^*) \le \frac{3 \Vert \ttm_1(\cE) \Vert}{\sigma_J(\ttm_1(\cT^*))} \le \frac{1}{4}.
\end{align*}
Suppose that we have $\rho(\widehat{U}_t, U^*) \le 1/4$. Let us prove that $\rho(\widehat{V}_t, V^*) \le 1/4$ and $\rho(\widehat{U}_{t + 1}, U^*) \le 1/4$.
First, applying bound~\eqref{eq: Vt rho recursive bound}, we deduce
\begin{align*}
    \rho(\widehat{V}_t, V^*) \le \frac{2 \sqrt{2} (\alpha_V + 2 \beta_V \cdot \rho(\widehat{U}_t, U^*))}{\sigma_K(\ttm_3(\cT^*)) \sqrt{1 - 2 \rho(\widehat{U}_t, U^*)}} \le \frac{4 (\alpha_V + \beta_V /2)}{\sigma_K(\ttm_3(\cT^*))} \le  \frac{6 \beta_V}{\sigma_K(\ttm_3(\cT^*))} \le \frac{1}{4},
\end{align*}
where we used
\begin{align*}
    \alpha_V =   \Vert \ttm_3(\cE) (U^* \otimes I_{d_2}) \Vert \le \sup_{\substack{U \in \R^{d_1 \times J} \\ \Vert U \Vert \le 1}} \Vert \ttm_3(\cE) (U \otimes I_{d_2}) \Vert = \beta_V
\end{align*}
and $\sigma_K(\ttm_3(\cT^*)) \ge 24 \beta_V$ due to conditions of the theorem. Similarly, from~\eqref{eq: Vt rho recursive bound}, we deduce
\begin{align*}
    \rho  (\widehat{U}_{t + 1}, U^*) & \le \frac{2 \sqrt{2} (\alpha_U + 2 \beta_U \cdot \rho(\widehat{V}_t, V^*))}{\sigma_J(\ttm_1(\cT^*)) \sqrt{1 - 2\rho(\widehat{V}_t, V^*)}} \\
    & \le \frac{4 (\alpha_U + \beta_U / 2)}{\sigma_J(\ttm_1(\cT^*))} \le \frac{6 \beta_U}{\sigma_J(\ttm_1(\cT^*))} \le \frac{6 \Vert \ttm_1(\cE) \Vert}{\sigma_J(\ttm_1(\cT^*))} \le \frac{1}{4},
\end{align*}
by the conditions of the theorem and the definition of $\alpha_U, \beta_U$. Hence, for each $t = 0, \ldots, T$, we have $\rho(\widehat{U}_t, U^*) \le 1/4$ and $\rho(\widehat{V}_t, V^*) \le 1/4$.

Hence, we can simplify bounds~\eqref{eq: Vt rho recursive bound},\eqref{eq: Ut rho recursive bound} as follows:
\begin{align*}
    \rho(\widehat{V}_t, V^*) & \le \frac{4\cdot\left  (\alpha_V + 2 \beta_V \cdot  \rho(\widehat{U}_t, U^*) \right )}{\sigma_K(\ttm_3(\cT^*))}, \\
    \rho(\widehat{U}_t, U^*)  & \le \frac{4 \cdot \left (\alpha_U + 2 \beta_U \cdot \rho(\widehat{V}_{t - 1}, V^*)\right )}{\sigma_{J}(\ttm_1(\cT^*)) }.
\end{align*}
We solve these recursive inequalities using the following proposition.
\begin{Prop}
\label{proposition: recursive inequality solution}
Suppose that a sequence of numbers $(\rho_t, \eta_t)$ satisfies
\begin{align*}
    \rho_t & \le x_1 + x_2 \eta_{t} ,\\
    \eta_t & \le y_1 + y_2 \rho_{t - 1}
\end{align*}
for some $x_1, y_1, x_2, y_2$ such that $x_2 y_2 \le 1/2$ and $x_2, y_2 \ge 0$. Then, we have
\begin{align*}
    \rho_t & \le 2(x_1 + x_2 y_1) + x_2 (x_2 y_2)^t \eta_0, \\
    \eta_t & \le 2 (y_1 + x_1 y_2) +(x_2 y_2)^{t} \eta_0.
\end{align*}
\end{Prop}

Applying Proposition~\ref{proposition: recursive inequality solution} to $\rho_t = \rho(\widehat{V}_t, V^*)$, $\eta_t = \rho(\widehat{U}_t, U^*)$, we obtain
\begin{align}
    \rho(\widehat{V}_t, V^*) & \le \frac{8 \alpha_V}{\sigma_K(\ttm_3(\cT^*))} + \frac{16 \beta_V \alpha_U}{\sigma_J(\ttm_1(\cT^*)) \sigma_K(\ttm_3(\cT^*))}  \nonumber \\
    & \quad + \left ( \frac{64\beta_V \beta_U}{\sigma_J(\ttm_1(\cT^*)) \sigma_K(\ttm_3(\cT^*))}\right )^t \times \frac{24 \beta_V  \Vert \ttm_1(\cE) \Vert}{\sigma_K(\ttm_3(\cT^*)) \sigma_J(\ttm_1(\cT^*))}, \label{eq: rho Vt final bound}\\
    \rho(\widehat{U}_t, U^*) & \le \frac{8 \alpha_U}{\sigma_J(\ttm_1(\cT^*))} + \frac{16 \beta_U \alpha_V}{\sigma_J(\ttm_1(\cT^*)) \sigma_K(\ttm_3(\cT^*))} \nonumber \\
    & \quad + \left ( \frac{64\beta_V \beta_U}{\sigma_J(\ttm_1(\cT^*)) \sigma_K(\ttm_3(\cT^*))}\right )^t \times \frac{3 \Vert \ttm_1(\cE) \Vert}{\sigma_J(\ttm_1(\cT^*))}, \label{eq: rho Ut final bound}
\end{align}
where we used~\eqref{eq: U0 bound} to bound $\eta_0 = \rho(\widehat{U}_0, U^*)$.

\noindent \textbf{Step 4. Final bound.} Let us return to the bound~\eqref{eq: step 1 final bound}. Using $\sqrt{\sum_i a_i} \le \sum_i \sqrt{a_i}$ suitable for any positive numbers $a_i$, we get
\begin{align*}
    \Vert \widehat{\cT} - \cT^* \Vert_{\F} & \le \sup_{U \in \bbO_{d_1, J}, V \in \bbO_{d_2, K}} \Vert \cE \times_3 V^\top \times_1 U^\top \Vert_{\F} \\
    & + 2 \sqrt{K} \Vert \ttm_3(\cE \times_1 \widehat{U}^\top) \Vert + \frac{2 \sqrt{J}}{\sigma_{\min}(\widehat{V}_{T - 1}^\top V^*)} \Vert \ttm_1(\cE \times_3 \widehat{V}_{T - 1}^\top) \Vert.
\end{align*}
Combining~\eqref{eq: 1/4 bound on rho} and~\eqref{eq: orthogonal product lower bound}, we obtain
\begin{align*}
    \Vert \widehat{\cT} - \cT^* \Vert_{\F} & \le \sup_{U \in \bbO_{d_1, J}, V \in \bbO_{d_2, K}} \Vert \cE \times_3 V^\top \times_1 U^\top \Vert_{\F} + 2 \sqrt{K} \Vert \ttm_3(\cE \times_1 \widehat{U}^\top) \Vert + 3 \sqrt{J} \Vert \ttm_1(\cE \times_3 \widehat{V}_{T - 1}^\top) \Vert.
\end{align*}
Then, applying~\eqref{eq: noise projector 1 second-order expansion} and~\eqref{eq: noise projector 3 second-order expansion}, we get
\begin{align*}
\Vert \widehat{\cT} - \cT^* \Vert_{\F} & \le \sup_{U \in \bbO_{d_1, J}, V \in \bbO_{d_2, K}} \Vert \cE \times_3 V^\top \times_1 U^\top \Vert_{\F} + 2 \sqrt{K} (\alpha_V + 2 \beta_V \rho(\widehat{U}_T, U^*)) \\
& \quad + 3 \sqrt{J} (\alpha_U + 2 \beta_U \cdot \rho(\widehat{V}_{T - 1}, V^*)).
\end{align*}
Then, we substitute bounds~\eqref{eq: rho Ut final bound},\eqref{eq: rho Vt final bound} into above, and get
\begin{align*}
    \Vert \widehat{\cT} - \cT^* \Vert_{\F} & \le \sup_{U \in \bbO_{d_1, J}, V \in \bbO_{d_2, K}} \Vert \cE \times_3 V^\top \times_1 U^\top \Vert_{\F} + 2 \sqrt{K} (\alpha_V + v_1 + v_2) \\
& \quad + 3 \sqrt{J} (\alpha_U +u_1 + u_2),
\end{align*}
where
\begin{align*}
    v_1 & = 2 \beta_V \cdot \frac{16 \beta_U \alpha_V}{\sigma_J(\ttm_1(\cT^*)) \sigma_K(\ttm_3(\cT^*))}, \\
    v_2 & = \frac{16 \beta_V \alpha_U}{\sigma_J(\ttm_1(\cT^*))} + \frac{6 \beta_V \Vert \ttm_1(\cE) \Vert}{\sigma_J(\ttm_1(\cT^*))} \times \left ( \frac{64\beta_V \beta_U}{\sigma_J(\ttm_1(\cT^*)) \sigma_K(\ttm_3(\cT^*))}\right )^T, \\
    u_1 & = 2 \beta_U \cdot \frac{16 \beta_V \alpha_U}{\sigma_{J}(\ttm_1(\cT^*)) \sigma_K(\ttm_3(\cT^*))} \\
    u_2 & = \frac{16 \beta_U \alpha_V}{\sigma_K(\ttm_3(\cT^*))} + \left ( \frac{64 \beta_U \beta_V}{\sigma_J(\ttm_1(\cT^*)) \sigma_K(\ttm_3(\cT^*))} \right )^T \Vert \ttm_1(\cE) \Vert.
\end{align*}
Since $\sigma_J(\ttm_1(\cT^*)) \ge 24 \Vert \ttm_1(\cE) \Vert \ge 24 \beta_U$ and $\sigma_K(\ttm_3(\cT^*)) \ge 24 \beta_V$, we have $v_1 \le \alpha_V$, $u_1 \le \alpha_U / 3$ and
\begin{align*}
    v_2 \le \frac{16 \beta_V \alpha_U}{\sigma_J(\ttm_1(\cT^*))} +  \left ( \frac{64\beta_V \beta_U}{\sigma_J(\ttm_1(\cT^*)) \sigma_K(\ttm_3(\cT^*))}\right )^T \Vert \ttm_1(\cE) \Vert.
\end{align*}
Combining the above, we obtain
\begin{align*}
    \Vert \widehat{\cT} - \cT^* \Vert_{\F} & \le \sup_{U \in \bbO_{d_1, J}, V \in \bbO_{d_2, K}} \Vert \cE \times_3 V^\top \times_1 U^\top \Vert_{\F} + 4 \sqrt{K} \alpha_V+ 4 \sqrt{J} \alpha_U + \diamondsuit_2 + r_T,
\end{align*}
where $\diamondsuit_2$ and $r_T$ are introduced in the statement of the theorem.
\end{proof}

\subsection{Proof of Proposition~\ref{proposition: projector difference via rho}}

\begin{proof}
    For any matrix $O \in \bbO_{b,b}$, we have
\begin{align*}
    \Vert \Pi_{\widehat{U}} - \Pi_{U^*} \Vert & = \Vert \widehat{U} \widehat{U}^\top - U^* (U^*)^\top \Vert = \Vert \widehat{U} \widehat{U}^\top - \widehat{U} O (U^*)^\top + \widehat{U} O (U^*)^\top - U^* (U^*)^\top \Vert \\
    & \le \Vert \widehat{U} O (\widehat{U} O  - U^*)^\top \Vert + \Vert (\widehat{U} O - U^*) (U^*)^\top \Vert \le 2 \Vert \widehat{U} O - U^* \Vert.
\end{align*}
Taking the infimum over $O \in \bbO_{b,b}$, we obtain the proposition.
\end{proof}

\subsection{Proof of Proposition~\ref{proposition: devis-kahan sin-theta with rho}}
\begin{proof}[Proof of Proposition~\ref{proposition: devis-kahan sin-theta with rho}]
For two subspaces $X, Y$ define:
\begin{align*}
    \Vert \sin \Theta(X, Y) \Vert = \Vert (I - \Pi_{X}) \Pi_{Y} \Vert.
\end{align*}
Then, the following theorem holds.

\begin{Th}[Wedin $\sin \Theta$-theorem~\citep{wedin1972perturbation} ]
    Let $P, Q$ be $\R^{a \times b}$ matrices. Fix $r \le \min\{a, b\}$. Consider the SVD decomposition of $P = U_0 \Sigma_0 V_0^\top + U_1 \Sigma_1 V_1^\top$, $Q = \widetilde{U}_0 \widetilde{\Sigma}_0 \widetilde{V}_0^\top + \widetilde{U}_1 \widetilde{\Sigma}_1 \widetilde{V}_1^\top$, where $\Sigma_0, \widetilde{\Sigma}_0$ corresponds to the first $r$ singular values of $P, Q$ respectively. Suppose that $\sigma_{\min}(\widetilde{\Sigma}_0) - \sigma_{\max}(\Sigma_1) \ge \delta$. Then, we have
    \begin{align*}
        \Vert \sin \Theta(\Img \widetilde{U}_0, \Img U_0) \Vert \le \frac{1}{\delta} \max\{ \Vert (P - Q) V_0^\top\Vert, \Vert U_0^\top (P - Q) \Vert\}. 
    \end{align*}
\end{Th}

To apply the above theorem, consider two cases. If $\sigma_r(A) \ge 2 \Vert E \Vert$, then we apply the above theorem with $\delta =\sigma_r(A)/2$, $P = B$ and $Q = A$, and get
\begin{align*}
    \Vert \sin \Theta(\Img L, \Img \widehat{L}) \Vert \le \frac{2 \Vert E \Vert}{\sigma_r(A)}.
\end{align*}
If $\sigma_r(A) \le 2 \Vert E \Vert$, then
\begin{align*}
    \Vert \sin \Theta(\Img L, \Img \widehat{L}) \Vert \le 1 \le \frac{2 \Vert E \Vert }{\sigma_r(A)}.
\end{align*}
Hence, in either case, we have
\begin{align*}
     \Vert \sin \Theta(\Img L, \Img \widehat{L}) \Vert \le \frac{2 \Vert E \Vert }{\sigma_r(A)}.
\end{align*}
Finally, Lemma 1 of~\citep{cai2018rate} implies that
\begin{align*}
    \rho(L, \widehat{L}) \le \sqrt{2} \Vert \sin \Theta(\Img L, \Img \widehat{L}) \Vert \le \frac{2 \sqrt2 \Vert E \Vert }{\sigma_r(A)},
\end{align*}
and the proposition follows.
\end{proof}

\subsection{Proof of Proposition~\ref{proposition: recursive inequality solution}}
\label{section: proof of recursive inequality solution}

\begin{proof}[Proof of Proposition~\ref{proposition: recursive inequality solution}]
    Combining the initial inequalities, we get
    \begin{align*}
        \eta_t \le y_1 + y_2 x_1 + (x_2 y_2) \eta_{t - 1}.
    \end{align*}
    Iterating the above inequality $t - 1$ times, we get
    \begin{align*}
        \eta_t \le (x_2 y_2)^t \eta_{0} + (y_1 + y_2 x_1)\sum_{i = 0}^{t - 1} (x_2 y_2)^i \le \frac{y_1 + y_2 x_1}{1 - x_2 y_2} + (x_2 y_2)^t \eta_0.
    \end{align*}
    Using $x_2 y_2 \le 1/2$, we obtain
    \begin{align*}
        \eta_t \le 2 (y_1 + y_2 x_1) + (x_2 y_2)^t \rho_0.
    \end{align*}
    Combining the above with the bound $\rho_t \le x_1 + x_2 \eta_t$, we derive
    \begin{align*}
        \rho_t \le x_1 + 2 (y_1 x_2 + x_2 y_2 x_1) + x_2 (x_2 y_2)^t \rho_0 \le 2(x_1 + x_2 y_1) + x_2 (x_2 y_2)^{t} \rho_0,
    \end{align*}
    where we used $x_2 y_2 \le 1/2$ again.
\end{proof}

\section{Proof of Theorem~\ref{theorem: concentration on the empirical process}}
\label{section: proof of main concentration theorem}

\begin{proof}
\noindent \textbf{Step 1. Reduction to the PAC-bayes inequality.} Let us rewrite the core expression, as a supremum of a certain empirical process. We have:



\begin{align*}
    & \sup_{(A_1, A_2, A_3) \in \prod_{i = 1}^3 \bbS_i} \langle \widehat{\cE} \times_3 A_3^\top \times_1 A_1^\top, A_2\rangle  = \sup_{(A_1, A_2, A_3) \in \prod_{i = 1}^3 \bbS_i} \langle A_2 \times_3 A_3^\top \times_1 A_1^\top, \widehat{\cE}\rangle \\ 
    & \qquad =  \sup_{(A_1, A_2, A_3) \in \prod_{i = 1}^3 \bbS_i} \langle A_2 \times_3 A_3 \times_1 A_1, \widehat{\cE} \rangle \\
    & \qquad  = \sup_{(A_1, A_2, A_3) \in \prod_{i = 1}^3 \bbS_i} \left\langle A_2 \times_3 A_3 \times_1 A_1, \sum_{i=1}^n \frac{1}{n}\mathcal{R}(\bX_i\bX_i^\top - \mathbb{E}(\bX\bX^\top)) \right\rangle \\ 
    & \qquad = \sup_{(A_1, A_2, A_3) \in \prod_{i = 1}^3 \bbS_i} \left\langle \cR^{-1} (A_2 \times_3 A_3 \times_1 A_1), \frac{1}{n} \sum_{i=1}^n \bX_i\bX_i^\top - \mathbb{E}(\bX\bX^\top) \right\rangle \\
    & \qquad = \sup_{(A_1, A_2, A_3) \in \prod_{i = 1}^3 \bbS_i} \frac{1}{n}\sum_{i=1}^n \left\{ \bX_i^\top \cR^{-1} (A_2 \times_3 A_3 \times_1 A_1)\bX_i \right. \\
    & \qquad \qquad \qquad \qquad \qquad \qquad \qquad \left. - \E\bX^\top\cR^{-1} (A_2 \times_3 A_3 \times_1 A_1)\bX\right\}.
\end{align*}

Define the following functions:
\begin{align*}
    f_i(A_2 \times_3 A_3 \times_1 A_1) & = \lambda \left \{ \bX_i^\top \cR^{-1}(A_2 \times_3 A_3 \times_1 A_1) \bX_i - \E \bX_i^\top \cR^{-1}(A_2 \times_3 A_3 \times_1 A_1) \bX_i \right \}, \\
    f_{\bX}(A_2 \times_3 A_3 \times_1 A_1) & = \lambda \left \{ \bX^\top \cR^{-1}(A_2 \times_3 A_3 \times_1 A_1) \bX - \E \bX^\top \cR^{-1}(A_2 \times_3 A_3 \times_1 A_1) \bX \right \},
\end{align*}
where the positive factor $\lambda$ will be chosen later.
We will apply Lemma~\ref{lem:pac-bayes} to the empirical process
\begin{align*}
     \sup_{(A_1, A_2, A_3) \in \prod_{i = 1}^s \bbS_i}\frac{1}{n} \sum_{i = 1}^n f_i(A_1, A_2, A_3)
\end{align*}
with the parameter space defined by the target spaces $L_i$ dimensionalities and the prior distribution $\mu$, constructed as a product of independent measures for each subspace separately. Choosing bases in $L_1, L_2, L_3$, we identify $A_1, A_2$ with corresponding matrices and $A_3$ with a corresponding tensor. 
Define linear spaces $\bbL_ 1 = \R^{d_1 \times l_1}, \bbL_2 = \R^{l_1 \times d_2 \times l_3}$ and $\bbL_3 = \R^{d_2 \times l_3}$, and consider distributions $\cD_i$ over $\bbL_i$ defined as follows:
\begin{align*}
    \cD_i = \begin{cases}
        \cN(0, \sigma_i I_{l_i d_i}), & \text{ if } l_i \cdot \ttr_i(\Sigma) \le \log|\bbS_i|, \\
        \operatorname{Uniform}(\bbS_i), & \text{ if } l_i \cdot \ttr_i(\Sigma) > \log |\bbS_i|,
    \end{cases}
\end{align*}
for some $\sigma_1, \sigma_2, \sigma_3$ to be chosen later, assuming that samples from the normal distribution have appropriate shapes.
Then, we put
\begin{align*}
    \mu = \cD_1 \otimes \cD_2 \otimes \cD_3.
\end{align*}
Consider random vectors $P, Q, R$ with mutual distribution $\rho_{A_1, A_2, A_3}$ such that $\E Q \times_3 R \times_1 P = A_2 \times_3 A_3 \times_1 A_1$. Since $f_{i}(A_1, A_2, A_3), f_{\bX}(A_1, A_2, A_3)$ are linear in $A_2 \times_3 A_3 \times_1 A_1$, we have $\E_{\rho_{A_1, A_3, A_2}} f_i(P, Q, R) = f_i(A_1, A_2, A_3)$, so Lemma~\ref{lem:pac-bayes} yields
\begin{equation}
\begin{aligned}
\label{eq: PAC-Bayes upper bound2}
        & \sup_{\substack{A_1 \in \bbS_1, \\ A_2 \in \mathbb{S}_{2}, A_3 \in \bbS_3}}
        \frac{1}{n} \sum_{i = 1}^n f_i(A_1, A_2, A_3)  \\ & \quad \le \sup_{\substack{A_1 \in \bbS_1, \\ A_2 \in \mathbb{S}_{2}, A_3 \in \bbS_3}} \left \{ \E_{\rho_{A_1, A_2, A_3}} \log \E_{\bX} \exp f_{\bX}(P, Q, R) +  \frac{\KL(\rho_{A_1, A_2, A_3}, \mu) + \log(1/\delta)}{n} \right \}
\end{aligned}
\end{equation}
with probability at least $1 - \delta$. Then, we construct $\rho_{A_1, A_2, A_3}$ such that the right-hand side of the above inequality can be controlled efficiently.

\noindent \textbf{Step 2. Constructing $\rho_{A_1, A_2, A_3}$.} Suppose for a while that  $\rho_{A_1, A_2, A_3}$-almost surely we have 
\begin{align}
\label{eq: lambda condition2}
\lambda \Vert \Sigma^{1/2} \cR^{-1}(Q \times_3 R \times_1 P)  \Sigma^{1/2} \Vert_{\F} \le 1/\omega.
\end{align}
Then, Assumption~\ref{as:orlicz} implies
\begin{equation}
\begin{aligned}
& \E_{\rho_{A_1, A_2, A_3}} \log \E_{\bX} \exp f_{\bX}(P, Q, R) \\
    & \qquad = \E_{\rho_{A_1, A_2, A_3}} \log \E_{\bX} \exp \left \{ \lambda \left ( \bX^\top \cR^{-1}(Q \times_3 R \times_1 P) \bX \right. \right.  \\
    & \qquad \qquad \qquad \qquad \qquad \qquad \qquad \left. \left. - \E \bX^\top \cR^{-1}(Q \times_3 R \times_1 P) \bX \right ) \right \} \\
    & \qquad \le \lambda^2 \omega^2  \E_{\rho_{A_1, A_2, A_3}} \Vert \Sigma^{1/2} \cR^{-1}(Q \times_3 R \times_1 P) \Sigma^{1/2} \Vert_{\F}^2.
\label{eq: hanson-wright bound2}
\end{aligned}
\end{equation}
So, to control the above and keep the left-hand side of~\eqref{eq: lambda condition2} bounded, we do the following. Consider random matrices $G_1 \in \R^{d_1 \times l_1}, G_3 \in \R^{d_3 \times l_3}$ and a random tensor $G_3 \in \R^{l_1 \times d_2 \times l_3}$ such that
\begin{align*}
    \rmvec{(G_i)} \sim \begin{cases}
        \cN(0, \sigma_i I_{d_i l_i}), & \text{ if } \ttr_i(\Sigma) \le \log |\bbS_i|, \\
        \delta_0, & \text{ if } l_i \cdot \ttr_i(\Sigma)  > \log |\bbS_i|,
    \end{cases}
\end{align*}
where $\delta_0$ is the delta measure supported on $0 \in \R^{d_i l_i}$.
Then, define a function $g: \R^{d_1 \times l_1} \times \R^{}$
\begin{align}
\label{eq: definition of f}
    g(u', v', w') = \Vert \Sigma^{1/2} \cR^{-1}(v' \times_3 w' \times_1 u')  \Sigma^{1/2} \Vert_{\F}^2.
\end{align}
Sequentially applying the triangle inequality for the Frobenius norm and using $(a + b)^2 \le 2a^2 + 2b^2$, we obtain
\begin{align}
    &f(A_1 + G_1, A_2 + G_2, A_3+ G_3) \le 2g(A_1, A_2 + G_2, A_3+ G_3) + 2g(G_1, A_2 + G_2, A_3+ G_3) \nonumber \\ 
    & \le 4g(A_1, A_2, A_3+ G_3) + 4g(G_1, G_2, A_3+ G_3) \nonumber \\
    & \quad + 4g(A_1, G_2, A_3+ G_3) + 4g(G_1, A_2, A_3+ G_3) \nonumber \\
    & \le 8g(A_1, A_3, A_2) + 8g(A_1, G_2, G_3) + 8g(A_1, A_3, G_3) + 8g(A_1, G_2, A_2) \nonumber \\ & \quad + 8g(G_1, A_3, A_2) + 8g(G_1, G_2, G_3) + 8g(G_1, A_3, G_3) + 8g(G_1, G_2, A_2). \label{eq: triangle upper bound}
\end{align}
Then, we define the distribution $\rho_{A_1, A_2, A_3}$ of the random vector $(P, Q, R)$ as the distribution of $(A_1 + G_1, A_2 + G_2, A_3 + G_3)$ subject to the condition
\begin{align*}
    (G_1, G_2, G_3)  \in \Upsilon & = 
    \begin{Bmatrix}
        & g(G_1, A_2, A_3) \le 8 \E g(G_1', A_2, A_3) \\
        & g(A_1, G_2, A_3) \le 8 \E g(A_1, G_2', A_3) \\
        & g(A_1, A_2, G_3) \le 8 \E g(A_1, A_2, G_3') \\
        (G_1, G_2, G_3) : & g(G_1, G_2, A_3) \le 8 \E g(G_1', G_2', A_3) \\
        & g(G_1, A_2, G_3) \le 8 \E g(G_1', A_2, G_3') \\
        & g(A_1, G_2, G_3) \le 8 \E g(A_1, G_2', G_3') \\
        & g(G_1, G_2, G_3) \le 8 \E g(G_1', G_2', G_3') \\
    \end{Bmatrix},
\end{align*}
where $G_1' \sim \cD_1, G_2' \sim \cD_2, G_3' \sim \cD_3$ are independent copies of $G_1, G_2, G_3$. Put $\Gamma = (\{A_1, G_1\} \times \{A_2, G_2\} \times \{ A_3, G_3\} ) \setminus \{(A_1, A_3, A_2)\}.$
Note that by the union bound and the Markov inequality, we have
\begin{align}
    \Pr \left ( (G_1, G_2, G_3) \not \in \Upsilon \right ) & \le \sum_{(a, b, c) \in \Gamma} \Pr \left (f(a, b, c) > 8 \E f(a, b, c) \right ) \nonumber \\
    & \le \sum_{(a, b, c) \in \Gamma } \frac{1}{8} = \frac{7}{8}. \label{eq: prob of Upsilon complement2}
\end{align}
Combining the definition of Upsilon with upper bound~\eqref{eq: triangle upper bound} implies the following bound on $g(P, Q, R)$:
\begin{align}
    g(P, Q, R) \le 64 \left (g(A_1, A_2, A_3) + \E g(A_1, A_2, G_3) + \E g(A_1, G_2, A_3) + \E g (A_1, G_2, G_3)  \right. \nonumber \\
    \left. + \E g(G_1, A_2, A_3) + \E g(G_1, A_2, G_3) + \E g(G_1, G_2, A_3) + \E g(G_1, G_2, G_3) \right ), \label{eq: almost sure triangle upper bound}
\end{align}
which holds $\rho_{A_1, A_2, A_3}$-almost surely.

Let us check that $\E_{\rho_{A_1, A_3, A_2}} Q \times_3 R \times_1 P = A_2 \times_3 A_3 \times_1 A_1$. Since both the Gaussian distribution and $\delta_0$ are centrally symmetric and the function $f$ does not change its value when multiplying any of its argument by $-1$, we have 
\begin{align}
(P, Q, R) \overset{d}{=} (A_1 + \eps_1 (P - A_1), A_2 + \eps_2 (Q - A_2), A_3 + \eps_3(R-A_3)), \label{eq: Rademacher symmetrization2}
\end{align}
where $\eps_1, \eps_2, \eps_3$ are i.i.d. Rademacher random variables independent of $(P, Q, R)$. Then, we obtain
\begin{align*}
    &\E Q \times_3 R \times_1 P = \E (A_2 + \eps_2 (Q-A_2)) \times_3 (A_3 + \eps_3 (R - A_3)) \times_1 A_1  \\&\qquad + \E (A_2 + \eps_2 (Q-A_2)) \times_3 (A_3 + \eps_3 (R - A_3)) \times_1 \eps_1 (P - A_1) \\&=  \E (A_2 + \eps_2 (Q-A_2)) \times_3 A_3 \times_1 A_1 + \E (A_2 + \eps_2 (Q-A_2)) \times_3 \eps_3 (R - A_3)) \times_1 A_1 \\& =A_2 \times_3 A_3 \times_1 A_1 + \E \eps_2(Q-A_2) \times_3 A_3 \times_1 A_1 = A_2 \times_3 A_3 \times_1 A_1. 
\end{align*}
Hence, to satisfy the assumption~\eqref{eq: lambda condition2} and use~\eqref{eq: hanson-wright bound2}, it is enough to bound expectations $\E f(a, b, c)$ for $(a, b, c) \in \{A_1, G_1\} \times \{A_3, G_3\} \times \{A_2, G_2\}$.

\noindent \textbf{Step 3. Bounding expectations $\E g(\cdot, \cdot, \cdot)$.} Let us start with $g(A_1, A_3, A_2)$. From the definition~\eqref{eq: definition of f}, we have
\begin{align}
        &g(A_1, A_2, A_3)  = \Vert \Sigma^{1/2} \cR^{-1}(A_2 \times_3 A_3 \times_1 A_1)\Sigma^{1/2}\Vert_{\F}^2 \nonumber
        \\ &\le \nS^2 \Vert \cR^{-1}(A_2 \times_3 A_3 \times_1 A_1) \Vert_{\F}^2 = \nS^2 \Vert A_2 \times_3 A_3 \times_1 A_1 \Vert_{\F}^2 = \nS^2, \label{eq: no-gaussians bound}
    \end{align}
    where we used the fact that $A_2$ has unit Frobenius norm and $\Vert A_1 \Vert \le 1,  \Vert A_3 \Vert \le 1$ by the definition of $\bbS_i$. 

    In what follows, it will be useful to rewrite the function $f(A_1, A_2, A_3)$ in different notation. As in the proof of Lemma~\ref{lemma: m1 norm upper bound}, define tensors
    \begin{align*}
        \cS_{p_1 q_1 r_1 p_2 q_2 r_2} & = \Sigma_{(p_1 - 1) qr + (q_1 - 1) r + r_1, (p_2 - 1) qr + (q_2 - 1) r + r_2} \\
        \ttA^{(1)}_{p_2 p_3 j_1} & = (A_1)_{(p_2 - 1) p + p_3, j_1}, \quad \ttA^{(3)}_{r_2 r_3 k_1} = (A_3)_{(r_2 - 1) r + r_3, k_1}, \\
        \ttA^{(2)}_{j_1 q_2 q_3 k_1} & = (A_3)_{j_1, (q_2 - 1) q + q_3, k_1}, \\
        \cG^{(1)}_{p_2 p_3 j_1} & = (G_1)_{(p_2 - 1) p + p_3, j_1}, \quad \cG^{(3)}_{r_2 r_3 k_1} = (G_3)_{(r_2 - 1) r + r_3, k_1}, \\
        \cG^{(2)}_{j_1 q_2 q_3 k_1} &  = (G_3)_{j_1, (q_2 - 1) q + q_3, k_1}.
    \end{align*}
    Then, we obtain
    \begin{align}
        g(A_1, A_2, A_3) & = \Vert \Sigma^{1/2} \cR^{-1}(A_2 \times_3 A_3 \times_1 A_1) \Sigma^{1/2} \Vert_{\F}^2 \nonumber \\ 
        & = \Tr \left (\Sigma \cR^{-1}(A_2 \times_3 A_3 \times_1 A_1) \Sigma \cR^{-\top }(A_2 \times_3 A_3 \times_1 A_1)  \right ) \nonumber \\
        & = \cS_{p_1 q_1 r_1 p_2 q_2 r_2} \ttA^{(1)}_{p_2 p_3 j_1} \ttA^{(2)}_{j_1 q_2 q_3 k_1} \ttA^{(3)}_{r_2 r_3 k_1} \cS_{p_3 q_3 r_3 p_4 q_4 r_4} \ttA^{(1)}_{p_1 p_4 j_2} \ttA^{(2)}_{j_2 q_1 q_4 k_2} \ttA^{(3)}_{r_1 r_4 k_2}. \label{eq: tensor representation 2}
    \end{align}
    Note that the above holds for any $A_i \in \bbL_i$, so the formula remains true when replacing $A_i, \ttA^{(i)}$ with $G_i, \cG^{(i)}$ respectively. 

    Next, we bound $\E g(A_1, A_2, G_3)$. If $\rmvec(G_1) \sim \delta_0$, we have $\E g(A_1, A_2, G_3) = 0$, so it is enough to consider the case $\rmvec(G_3) \sim \cN(0, \sigma_3 I_{d_3 l_3})$. Due to formula~\eqref{eq: tensor representation 2}, it yields
    \begin{align*}
        \E g(A_1, A_2, G_3) & = \E \cS_{p_1 q_1 r_1 p_2 q_2 r_2} \ttA^{(1)}_{p_2 p_3 j_1} \ttA^{(2)}_{j_1 q_2 q_3 k_1} \cG^{(3)}_{r_2 r_3 k_1} \cS_{p_3 q_3 r_3 p_4 q_4 r_4} \ttA^{(1)}_{p_1 p_4 j_2} \ttA^{(2)}_{j_2 q_1 q_4 k_2} \cG^{(3)}_{r_1 r_4 k_2} \\
        & = \sigma_3^2 \delta_{r_2 r_1} \delta_{r_3 r_3} \delta_{k_1k_2} \cS_{p_1 q_1 r_1 p_2 q_2 r_2} \ttA^{(1)}_{p_2 p_3 j_1} \ttA^{(2)}_{j_1 q_2 q_3 k_1}  \cS_{p_3 q_3 r_3 p_4 q_4 r_4 } \ttA^{(1)}_{p_1 p_4 j_2} \ttA^{(2)}_{j_2 q_1 q_4 k_2} \\
        & = \sigma_3^2 \cS_{p_1 q_1 r_1 p_2 q_2 r_1} \ttA^{(1)}_{p_2 p_3 j_1} \ttA^{(2)}_{j_1 q_2 q_3 k_1} \cS_{p_3 q_3 r_3 p_4 q_4r_3} \ttA^{(1)}_{p_1 p_4 j_2} \ttA^{(2)}_{j_2 q_1q_4k_1}.
    \end{align*}
    Define matrices $\widetilde{A}^{(1, j)} \in \R^{p \times p}, \widetilde{A}^{(1, j, k)}$, $i = 1, 2$ and $j = 1, \ldots, J$, by $\widetilde{A}^{(1, j)}_{p_2, p_3} = \ttA^{(1)}_{p_2 p_3 j_1}$ and $\widetilde{A}^{(2, j, k)}_{q_2 ,q_3} = \ttA^{(2)}_{j p_2 p_3 k}$. Then, we have
    \begin{align}
        \E g(A_1, A_2, G_3) & = \sigma_3^2 \cdot \sum_{k_1 \in [l_3]} \Tr \left ( \Tr_3(\Sigma) \sum_{j_1 = 1}^{l_1} \widetilde{A}^{(1, j_1)} \otimes \widetilde{A}^{(2, j_1, k_1)} \right. \nonumber \\
        & \qquad \qquad \qquad \times\left. \Tr_3(\Sigma) \sum_{j_2 = 1}^{l_1} (\widetilde{A}^{(1, j_2)} \otimes \widetilde{A}^{(2, j_2, k_1)})^\top \right ) \nonumber \\
        & \le \sigma_3^2 \sum_{k_1 \in [l_3]}  \left \Vert \Tr_3(\Sigma) \cdot \sum_{j_1 \in [J]} \widetilde{A}^{(1, j_1)} \otimes \widetilde{A}^{(2, j_1, k_1)} \right \Vert_{\F}^2 \nonumber \\
        & \le \sigma_3^2 \Vert \Tr_3(\Sigma) \Vert^2 \cdot \sum_{k_1 \in [l_3]} \Vert \sum_{j_1 \in [l_1]} \widetilde{A}^{(1, j_1)} \otimes \widetilde{A}^{(2, j_1, k_1)} \Vert_{\F}^2, \label{eq: G_3 bound via frobenius tensor product}
    \end{align}
    where we used the Cauchy--Schwartz inequality for the scalar product $\langle A, B \rangle = \Tr(A^\top B) \le \Vert A \Vert_{\F} \Vert B \Vert_{\F}$. Then, we introduce matrices $A'^{(2, k_1)}_{j_1, (q_2 - 1)q + q_3} = \ttA_{j_1 q_2 q_3 k_1}$, $k_1 \in [l_3]$, for which we have
    \begin{align}
        \sum_{k_1 \in [l_3]} \Vert \sum_{j_1 \in [l_1]} \widetilde{A}^{(1, j_1)} \otimes \widetilde{A}^{(2, j_1, k_1)} \Vert_{\F}^2 & = \sum_{k_1 \in [l_3]} \Vert A_1^\top A'^{(2,k_1)} \Vert_{\F}^2 \le \sum_{k_1 \in [l_3]} \Vert A_1^\top \Vert^2 \Vert A'^{(2,k_1)} \Vert^2_{\F} \nonumber \\
        & \le \sum_{k_1 \in [l_3]} \Vert A'^{(2, k_1)} \Vert_{\F}^2 = \Vert A_2 \Vert_{\F}^2 \le 1 \nonumber,
    \end{align}
    where we used $\Vert A_1 \Vert \le 1$ and $\Vert A_2 \Vert_{\F} \le 1$. 
    Substituting the above into~\eqref{eq: G_3 bound via frobenius tensor product} yields
    \begin{align}
        \E g(A_1, A_2, G_3) \le \sigma_3^2 \Vert \Tr_3(\Sigma) \Vert^2. \label{eq: G_3 bound}
    \end{align}
    Analogously, we obtain
    \begin{align}
        \E g(G_1, A_2, A_3) \le
            \sigma_1^2 \Vert \Tr_1(\Sigma) \Vert^2  \label{eq: G_1 bound}
    \end{align}

    Next, we study the term $\E g(A_1, G_2, A_3)$. Obviously, if $\rmvec(G_2) \sim \delta_0$, then $\E g(A_1, G_2, A_3) = 0$, so we consider the case then $\rmvec(G_2) \sim \cN(0, \sigma_3 I_{d_2 l_2})$. Using~\eqref{eq: tensor representation 2} with $G_2$ in place of $A_2$ and defining a matrix $\widetilde{A}^{(3, k_1)} \in \R^{r \times r}$ as $\widetilde{A}^{(3, k_1)}_{r_2 r_3} = \ttA^{(3)}_{r_2 r_3 k_1}$, we obtain
    \begin{align*}
        \E g(A_1, G_2, A_3) & = \E \cS_{p_1 q_1 r_1 p_2 q_2 r_2} \ttA^{(1)}_{p_2 p_3 j_1} \cG^{(2)}_{j_1 q_2 q_3 k_1} \ttA^{(3)}_{r_2 r_3 k_1} \cS_{p_3 q_3 r_3 p_4 q_4 r_4} \ttA^{(1)}_{p_1 p_4 j_2} \cG^{(2)}_{j_2 q_1 q_4 k_2} \ttA^{(3)}_{r_1 r_4 k_2}, \\
        & = \sigma_2^2 \delta_{j_1 j_2} \delta_{q_1 q_2} \delta_{k_1 k_2} \cS_{p_1 q_1 r_1 p_2 q_2 r_2} \ttA^{(1)}_{p_2 p_3 j_1} \ttA^{(3)}_{r_2 r_3 k_1} \cS_{p_3 q_3 r_3 p_4 q_4 r_4} \ttA^{(1)}_{p_1 p_4 j_2} \ttA^{(3)}_{r_1 r_4 k_2} \\
        & = \sigma_2^2 \cS_{p_1 q_1 r_1 p_2 q_1 r_2} \ttA^{(1)}_{p_2 p_3 j_1} \ttA^{(3)}_{r_2 r_3 k_1} \cS_{p_3 q_2 r_3 p_4 q_2 r_4} \ttA^{(1)}_{p_1 p_4 j_1} \ttA^{(3)}_{r_1 r_4 k_1} \\
        & = \sigma_2^2  \sum_{j_1 \in [l_1], k_1 \in [l_3]} \Tr \left ( \Tr_2(\Sigma) \cdot [\widetilde{A}^{(1, j_1)} \otimes \widetilde{A}^{(3, k_1}]\cdot \Tr_2(\Sigma) \cdot [\widetilde{A}^{(1, j_1)} \otimes \widetilde{A}^{(3, k_1}]^\top \right ) \\
        & \le \sigma_2^2 \sum_{j_1 \in [l_1], k_1 \in [l_3]} \Vert \Tr_2(\Sigma)  \cdot [\widetilde{A}^{(1, j_1)} \otimes \widetilde{A}^{(3, k_1)}] \Vert_{\F}^2,
    \end{align*}
    where we used the Cauchy--Schwartz inequality on the last line. It yields
    \begin{align}
        \E g (A_1, G_2, A_3) & \le \sigma_2^2 \Vert \Tr_2(\Sigma) \Vert^2 \sum_{j_1 \in [l_1], k_1 \in [l_3]} \Vert  \widetilde{A}^{(1, j_1)} \otimes \widetilde{A}^{(3, k_1)} \Vert_{\F}^2 \nonumber  \\
        & = \sigma^2_2  \Vert \Tr_2(\Sigma) \Vert^2 \sum_{j_1 \in [l_1], k_1 \in [l_3]} \Vert \widetilde{A}^{(1, j_1)}  \Vert_{\F}^2 \Vert \widetilde{A}^{(3, k_1)} \Vert_{\F}^2 \nonumber \\
        & = \sigma_2^2 \Vert \Tr_2(\Sigma) \Vert^2 \Vert A_1 \Vert_{\F}^2 \Vert A_3 \Vert_{\F}^2 \le \sigma_2^2 l_1 l_3 \Vert \Tr_2(\Sigma) \Vert^2, \label{eq: G_2 bound}
    \end{align}
    where we used $\Vert A_i \Vert_{\F}^2 \le l_i \Vert A_i \Vert^2 \le l_i$ for $i = 1, 3$.

    Next, we bound $\E g(A_1, G_2, G_3)$. If either $\rmvec(G_2) \sim \delta_0$ or $\rmvec(G_3) \sim \delta_0$, then $\E g(A_1, G_2, G_3) = 0$, so we consider the case when both $\rmvec(G_2) \sim \cN(0, \sigma_2^2 I_{d_2 l_2})$ and  $\rmvec(G_3) \sim \cN(0, \sigma_3^2 I_{d_3 l_3})$. Using~\eqref{eq: tensor representation 2} with $G_2, G_3$ in place of $A_2, A_3$, we get
    \begin{align*}
        \E g(A_1, G_2, G_3) & = \E \cS_{p_1 q_1 r_1 p_2 q_2 r_2} \ttA^{(1)}_{p_2 p_3 j_1} \cG^{(2)}_{j_1 q_2 q_3 k_1} \cG^{(3)}_{r_2 r_3 k_1} \cS_{p_3 q_3 r_3 p_4 q_4 r_4} \ttA^{(1)}_{p_1 p_4 j_2} \cG^{(2)}_{j_2 q_1 q_4 k_2} \cG^{(3)}_{r_1 r_4 k_2}, \\
        & = \sigma_2^2 \sigma_3^2 \delta_{k_1 k_1} \cS_{p_1 q_1 r_1 p_2 q_1 r_1} \ttA^{(1)}_{p_2 p_3 j_1} \cS_{p_3 q_3 r_3 p_4 q_3 r_3} \ttA^{(1)}_{p_1 p_4 j_1} \\
        & = \sigma_2^2 \sigma_3^2 l_3 \sum_{j_1 = 1}^{l_1} \Tr \left ( \Tr_{2,3}(\Sigma) \widetilde{A}^{(1, j_1)} \Tr_{2,3}(\Sigma) (\widetilde{A}^{(1, j_1)})^\top \right ) \\
        & \le \sigma_2^2 \sigma_3^2 l_3 \sum_{j_1 = 1}^{l_1} \Vert \Tr_{2,3}(\Sigma) \widetilde{A}^{(1, j_1)} \Vert_{\F}^2 \le \sigma_2^2 \sigma_3^2 l_3 \Vert \Tr_{2,3}(\Sigma) \Vert^2 \sum_{j_1 = 1}^{l_1} \Vert \widetilde{A}^{(1, j_1)} \Vert_{\F}^2 \\
        & = \sigma_2^2 \sigma_3^2 l_3 \Vert \Tr_{2,3}(\Sigma) \Vert^2 \Vert A_1 \Vert_{\F}^2.
    \end{align*}
    Since $\Vert A_1 \Vert_{\F}^2 \le l_1 \Vert A \Vert^2$, we obtain
    \begin{align}
        \E g (A_1, G_2, G_3) \le \sigma_2^2 \sigma_3^2 l_1 l_3 \Vert \Tr_{2,3}(\Sigma) \Vert^2. \label{eq: G_2 G_3 bound}
    \end{align}
    Analogously, we get
    \begin{align}
        \E g (G_1, G_2, A_3) \le \sigma_1^2 \sigma_2^2 l_1 l_3 \Vert \Tr_{1,2}(\Sigma) \Vert^2. \label{eq: G_1 G_2 bound}
    \end{align}
    Then, we bound $\E g (G_1, A_2, G_3)$. Using~\eqref{eq: tensor representation 2} with $G_1, G_3$ in place of $A_1, A_3$, we get
    \begin{align}
        \E g (G_1, A_2, G_3) & = \E \cS_{p_1 q_1 r_1 p_2 q_2 r_2} \cG^{(1)}_{p_2 p_3 j_1} \ttA^{(2)}_{j_1 q_2 q_3 k_1} \cG^{(3)}_{r_2 r_3 k_1} \cS_{p_3 q_3 r_3 p_4 q_4 r_4} \cG^{(1)}_{p_1 p_4 j_2} \ttA^{(2)}_{j_2 q_1 q_4 k_2} \cG^{(3)}_{r_1 r_4 k_2} \nonumber \\
        & = \sigma_1^2 \sigma_3^2 \delta_{p_1 p_2} \delta_{j_1 j_2} \delta_{r_1 r_2} \delta_{k_1 k_2} \delta_{p_3 p_4} \delta_{r_3 r_4}  \nonumber \\
        & \quad \times \cS_{p_1 q_1 r_1 p_2 q_2 r_2} \ttA^{(2)}_{j_1 q_2 q_3 k_1} \cS_{p_3 q_3 r_3 p_4 q_4 r_4} \ttA^{(2)}_{j_2 q_1 q_4 k_2} \nonumber  \\
        & = \sigma_1^2 \sigma_3^2 \cS_{p_1 q_1 r_1 p_1 q_2 r_1} \ttA^{(2)}_{j_1 q_2 q_3 k_1} \cS_{p_3 q_3 r_3 p_3 q_4 r_3} \ttA^{(2)}_{j_1 q_1 q_4 k_1}  \nonumber \\
        & = \sigma_1^2 \sigma_2^2 \sum_{j_1 \in [l_1], k_1 \in [l_3]} \Tr \left ( \Tr_{1,3}(\Sigma) \widetilde{A}^{(2, j_1, k_1)} \Tr_{1,3} (\Sigma) (\widetilde{A}^{(2, j_1, k_1)})^\top \right ). \nonumber
    \end{align}
    By the Cauchy--Schwartz inequality for the matrix product, we obtain
    \begin{align}
        \E g (G_1, A_2, G_3) & \le \sigma_1^2 \sigma_3^2 \sum_{j_1 \in [l_1], k_1 \in [l_3]} \Vert \Tr_{2,3}(\Sigma) \widetilde{A}^{(2, j_1, k_1)} \Vert_{\F}^2 \nonumber \\
        & \le \sigma_1^2 \sigma_3^2 \Vert \Tr_{2,3}(\Sigma) \sum_{j_1 \in [l_1], k_1 \in [l_3]} \Vert \widetilde{A}^{(2, j_1, k_1)} \Vert_{\F}^2 \nonumber \\
        & = \sigma_1^2 \sigma_3^2 \Vert \Tr_{2,3}(\Sigma) \Vert^2 \Vert A_2 \Vert_{\F}^2 = \sigma_1^2 \sigma_3^2 \Vert \Tr_{2,3}(\Sigma) \Vert^2. \label{eq: G_1 G_3 bound}
    \end{align}
    Finally, we bound $\E g (G_1, G_2, G_3)$. If some $G_i$ is distributed according to $\delta_0$, then $\E g(G_1, G_2, G_3) = 0$, so it is enough to consider the case when $G_1, G_2, G_3$ are Gaussian. Using~\eqref{eq: tensor representation 2} with $A_i, \ttA^{(i)}$ replaced by $G_i, \cG^{(i)}$, we obtain
    \begin{align}
        \E g (G_1, G_2, G_3) & = \E \cS_{p_1 q_1 r_1 p_2 q_2 r_2} \cG^{(1)}_{p_2 p_3 j_1} \cG^{(2)}_{j_1 q_2 q_3 k_1} \cG^{(3)}_{r_2 r_3 k_1} \cS_{p_3 q_3 r_3 p_4 q_4 r_4} \cG^{(1)}_{p_1 p_4 j_2} \cG^{(2)}_{j_2 q_1 q_4 k_2} \cG^{(3)}_{r_1 r_4 k_2} \nonumber \\
        & = \sigma_1^2 \sigma_2^2 \sigma_3^2 \delta_{j_1 j_1} \delta_{k_1 k_2} \cS_{p_1 q_1 r_1 p_1 q_1 r_1} \cS_{p_3 q_3 r_3 p_3 q_3 r_3} \nonumber \\
        & = \sigma_1^2 \sigma_2^2 \sigma_3^2 l_1 l_3 \Tr(\Sigma)^2. \label{eq: G_1 G_2 G_3 bound}
    \end{align}
    We summarized obtained bounds on $\E g (\cdot, \cdot, \cdot)$ in Table~\ref{tab: expectations}.
    \begin{table}[ht]
        \centering
        \begin{tabular}{|c|c|c|}
            \hline
            Quantity & Bound &  Ref.  \\
            \hline
            $g (A_1, A_2, A_3)$ & $\Vert \Sigma \Vert^2$ & \eqref{eq: no-gaussians bound} \\
            & & \\
            $\E g (A_1, A_2, G_3)$ & $
                \sigma_3^2 \Vert \Tr_3(\Sigma) \Vert$ & \eqref{eq: G_3 bound} \\
            & & \\
            $\E g (G_1, A_2, A_3)$ & $
                \sigma_1^2 \Vert \Tr_1(\Sigma) \Vert^2$  & \eqref{eq: G_1 bound}  \\
            & & \\
            $\E g (A_1, G_2, A_3)$ & $\sigma_2^2 l_1 l_3 \Vert \Tr_2(\Sigma) \Vert^2$ & \eqref{eq: G_2 bound} \\
            & & \\
            $\E g (A_1, G_2, G_3)$ & $\sigma_2^2 \sigma_3^2 l_1 l_3 \Vert \Tr_{2,3}(\Sigma) \Vert^2$  & \eqref{eq: G_2 G_3 bound} \\
            & & \\ 
            $\E g (G_1, G_2, A_3)$ &  $\sigma_1^2 \sigma_2^2 l_1 l_3 \Vert \Tr_{1,2}(\Sigma) \Vert^2$ & \eqref{eq: G_1 G_2 bound}\\
            & & \\ 
            $\E g (G_1, A_2, G_3)$ & $\sigma_1^2 \sigma_3^2 \Vert \Tr_{2,3}(\Sigma) \Vert^2$  & \eqref{eq: G_1 G_3 bound}\\ 
            & & \\ 
            $\E g (G_1, G_2, G_3)$ & $\sigma_1^2 \sigma_2^2 \sigma_3^2 l_1 l_3 \Tr(\Sigma)^2$ & \eqref{eq: G_1 G_2 G_3 bound}\\ 
            \hline
        \end{tabular}
        \caption{Bounds on $\E g (\cdot, \cdot, \cdot)$.}
        \label{tab: expectations}
    \end{table}

    Combining~\eqref{eq: almost sure triangle upper bound} with bounds~\eqref{eq: no-gaussians bound},\eqref{eq: G_1 bound}-\eqref{eq: G_1 G_2 G_3 bound} implies  the following $\rho_{A_1,A_2,A_3}$-almost surely:
    \begin{align*}
        g (P, Q, R) & \le 64 \left ( \Vert \Sigma \Vert^2 +\sigma_1^2 \sigma_2^2 \sigma_3^2 l_1 l_3 \Tr(\Sigma)^2 \right. \\
        & \left. \quad \quad + 
            \sigma_3^2  \Vert \Tr_3(\Sigma) \Vert^2 + \sigma_2^2 l_1 l_3 \Vert \Tr_2(\Sigma) \Vert^2+ \sigma_1^2 \Vert \Tr_1(\Sigma) \Vert   \right. \\
        & \left. \quad \quad + \sigma_2^2 \sigma_3^2 l_1 l_3 \Vert \Tr_{2,3}(\Sigma) \Vert^2 + \sigma_1^2 \sigma_2^2 l_1 l_3 \Vert \Tr_{1,2}(\Sigma) \Vert^2 + \sigma_1^2 \sigma_3^2 \Vert \Tr_{2,3}(\Sigma) \Vert^2
        \right ).
    \end{align*}
    Finally, we choose $\sigma_1^2, \sigma_2^2, \sigma_3^2$ as follows:
    \begin{align*}
        \sigma_1  = \ttr_1^{-1}(\Sigma), \qquad 
        \sigma_2  = \ttr_2^{-1}(\Sigma) / \sqrt{l_1 l_3}, \qquad
        \sigma_3 = 
            \ttr_3^{-1}(\Sigma).
    \end{align*}
    Then, $\rho_{A_1, A_2, A_3}$-almost surely, we have 
    \begin{align*}
        \Vert \Sigma^{1/2} \cR^{-1}(P \times_1 R \times_3 Q)  \Sigma^{1/2} \Vert_{\F}^2 = f(P, Q, R) \le 2^{12} \Vert \Sigma \Vert^2,
    \end{align*}
    where we used $\Vert \Tr_{S} (\Sigma) \Vert \le \Vert \Sigma \Vert \cdot  \prod_{s \in S} \ttr_s(\Sigma)$ for any non-empty $S$. Hence, if $\lambda$ satisfies
    \begin{align}
        2^6 \lambda \omega \Vert \Sigma \Vert \le 1, \label{eq: lambda final condition2}
    \end{align}
    then~\eqref{eq: lambda condition2} is fulfilled and, due to~\eqref{eq: hanson-wright bound2}, we have
    \begin{align}
        \E_{\rho_{A_1, A_2, A_3}} \log \E_{\bX} \exp f_{\bX}(P, Q, R) \le 2^{12} \lambda^2 \omega^2 \Vert  \Sigma \Vert^2. \label{eq: final bound on the log-moment2}
    \end{align}

        \noindent \textbf{Step 4. Bounding the Kullback-Leibler divergence.} Define $I = \{i \in [3] \mid l_i \ttr_i(\Sigma) > \log |\bbS_i| \}$. Then, for $i \in I$, we have $\cD_i = \operatorname{Uniform}(\bbS_i)$ and the density of $\rho_{A_1, A_2, A_3}$ is given by
        \begin{align*}
            \rho_{A_1, A_2, A_3}(a_1, a_2, a_3) & = \prod_{i \in  I} \delta_0(a_i - A_i) \times \prod_{i \in [3] \setminus I} \frac{\sigma_i^{- l_i d_i}}{(2 \pi)^{l_i d_i/2}} \exp \left \{ - \frac{1}{2 \sigma_i^2} \Vert   a_i - A_i \Vert_{\F}^2 \right \} \\
            & \quad \times \frac{\1 \left \{(a_1 - A_1, a_2 - A_2, a_3 - A_3) \in \Upsilon\right \} }{\Pr((G_1, G_2, G_3) \in \Upsilon)}.
        \end{align*}
        By the definition of $\Upsilon$, $\rho_{A_1, A_2, A_3}$ can be decomposed into product of the truncated Gaussian $\rho_{-I}$ and delta measures $\bigotimes_{i \in I} \delta_{A_i}$. Hence, we have
        \begin{align}
            \KL(\rho_{A_1, A_2, A_3}, \mu) & = \KL(\rho_{-I} \otimes \bigotimes_{i \in I} \delta_{A_i}, \cD_1 \otimes \cD_2 \otimes \cD_3) \nonumber \\
            & = \KL(\rho_{-I}, \bigotimes_{i \in [3] \setminus I} \cD_i) + \sum_{i \in I} \KL(\delta_{A_i}, \operatorname{Uniform}(\bbS_i)) \nonumber \\
            & = \KL(\rho_{-I}, \bigotimes_{i \in [3] \setminus I} \cD_i) + \sum_{i \in I} \log |\bbS_i|. \label{eq: KL decomposition2}
        \end{align}
        Recap that for $i \in [3] \setminus I$, distribution $\cD_i$ is the centered Gaussian with the covariance matrix $\sigma_i^2 I_{d_i l_i}$ up to the reshaping, so the density of $\bigotimes_{i \in I} \cD_i$ is given by
        \begin{align*}
            \mu_{-I}((a_i)_{i \in [3] \setminus I}) = \prod_{i \in [3] \setminus I} \frac{\sigma_i^{-d_i l_i}}{(2\pi)^{d_i l_i/2}} \exp \left ( - \frac{1}{2 \sigma_i^2} \Vert a_i \Vert^2_{\F} \right ).
        \end{align*}
        Hence, we have
        \begin{align*}
            \KL(\rho_{-I}, \otimes_{i \in[3] \setminus I} \cD_i) & = \int_{\prod_{i \in [3] \setminus I} \bbL_i} \rho_{-I}((a_i)_{i \in [3] \setminus I}) \\
            & \qquad  \times \log \left [ \frac{\prod_{i \in [3] \setminus I} \exp \left ( \Vert a_i \Vert_{\F}^2  / 2 \sigma_i^2 - \Vert a_i - A_i \Vert_{\F}^2 / 2 \sigma_i^2\right )}{\Pr((G_1, G_2, G_3) \in \Upsilon)}  \right ] \prod_{i \in [3] \setminus I} \rmd a_i \\
            & = \log \frac{1}{\Pr((G_1, G_2, G_3) \in \Upsilon)} - \sum_{i \in [3] \setminus I} \frac{1}{2 \sigma_i^2}\Vert A_i \Vert_{\F}^2 + \sum_{i \in [3] \setminus I} \frac{1}{\sigma_i^2}\langle \E \bxi^i, A_i \rangle,
        \end{align*}
        where $\bxi^i$ is distributed as the $i$-th marginal of $(P, Q, R) \sim \rho_{A_1, A_2, A_3}$. Using~\eqref{eq: Rademacher symmetrization2} , we get $\E \bxi^i = A_i$, so bound~\eqref{eq: prob of Upsilon complement2} implies
        \begin{align*}
            \KL(\rho_{-I}, \otimes_{i \in [3] \setminus I} \cD_i)  & \le \log 8 + \sum_{i \in [3] \setminus I} \frac{1}{2 \sigma_i^2} \Vert A_i \Vert_{\F}^2 \\
            & \le \log 8 + \frac{1}{2} \sum_{i \in [3] \setminus I} l_i \ttr_i^2(\Sigma),
        \end{align*}
        where we used the definition of $\sigma_i$ and the fact that $\Vert A_i \Vert_{\F}^2 \le l_i \Vert A_i \Vert^2\le l_i$ for $i = 1, 3$. Then, bound~\eqref{eq: KL decomposition2} implies
        \begin{align}
            \KL(\rho_{A_1, A_2, A_3}, \mu) & \le \log 8 + \frac{1}{2} \sum_{i \in [3] \setminus I} l_i \ttr_i^2(\Sigma) + \sum_{i \in I} \log |\bbS_i| \nonumber \\
            & \le \log 8 + \sum_{i = 1}^3 \min\{\ttr_i^2(\Sigma) \cdot l_i, \log |\bbS_i|\}. \label{eq: KL divergence final bound}
        \end{align}

    \textbf{Step 5. Final bound.} Then, we substitute bounds~\eqref{eq: final bound on the log-moment2},\eqref{eq: KL divergence final bound} into~\eqref{eq: PAC-Bayes upper bound2}. It yields
    \begin{align*}
        \sup_{\substack{A_1 \in \bbS_1, \\ A_2 \in \mathbb{S}_{2}, A_3 \in \bbS_3}} \frac{1}{n} \sum_{i = 1}^n \langle \widehat{\cE} \times_3 A_3^\top \times_1 A_1^\top, A_2 \rangle & \le 2^{12} \lambda \omega^2 \Vert \Sigma \Vert^2  \\
        & \quad + \frac{\log 8 + \sum_{i = 1}^3 \min \{\ttr_i(\Sigma) \cdot l_i, \log |\bbS_i|\}  + \log \frac{1}{\delta}}{\lambda n}
    \end{align*}
    with probability at least $1 - \delta$, provided $2^6 \lambda \omega \Vert \Sigma \Vert \le 1$. Since $n \ge \sum_{i = 1}^3 \min\{\ttr_i^2(\Sigma) \cdot l_i, \log |\bbS_i|\} + \log(8/\delta)$, we can choose $\lambda$ as
    \begin{align*}
        \lambda = \frac{1}{2^6 \omega \Vert \Sigma \Vert} \sqrt{\frac{\sum_{i = 1}^3 \min \{\ttr_i^2(\Sigma) \cdot l_i, \log |\bbS_i| \} + \log (8/\delta)}{n}}.
    \end{align*}
    It implies
    \begin{align*}
        \sup_{\substack{A_1 \in \bbS_1, \\ A_2 \in \mathbb{S}_{2}, A_3 \in \bbS_3}} \frac{1}{n} \sum_{i = 1}^n \langle \widehat{\cE} \times_3 A_3^\top \times_1 A_1^\top, A_2 \rangle \le 2^7 \omega \Vert \Sigma \Vert \sqrt{\frac{\sum_{i = 1}^3 \min \{\ttr_i^2(\Sigma) \cdot l_i, \log |\bbS_i| \} + \log (8/\delta)}{n}}
    \end{align*}
    with probability at least $1 - \delta$. This completes the proof.
\end{proof}

\section{Additional Experiments}
\label{section: additional experiments}

\subsection{Tensor-PRLS pseudocode}
\label{section: tensor PRLS pseudocode}

In this section, we give pseudocode for our version of PRLS adopted to order-$3$ tensors. See Algorithm~\ref{algo:tsilikagridis_thresholding}.

\begin{algorithm}[htbp]
\caption{PRLS Thresholding Algorithm}
\label{algo:tsilikagridis_thresholding}
\begin{algorithmic}
\Require Tensor $\cX \in \mathbb{R}^{d_1 \times d_2 \times d_3}$, regularization parameters $\lambda_1$, $\lambda_2$
\Ensure Soft-thresholded tensor $\widehat{\cX}$
\Statex \textbf{Step 1: Mode-1 Unfolding and Thresholding}
\State Reshape initial tensor into matrix: $\cX_{(1)} = \ttm_1(\cX)$
\State Perform SVD of matricization: $U, S, V^\top = \SVD(\cX_{(1)})$
\State Apply soft-thresholding: $S' = \max(S - \lambda_1/2, 0)$
\State Combine soft-thresholded SVD into a matrix: $\widehat{\cX}_{(1)} = U \cdot \diag(S') \cdot V^\top$
\State Reshape back into tensor: $\cX' = \ttm_1^{-1}(\widehat{\cX}_{(1)})$
\Statex \textbf{Step 2: Mode-3 Unfolding and Thresholding}
\State Reshape new approximation into matrix: $\cX_{(3)} = \ttm_3(\cX')$
\State Perform SVD of matricization: $U, S, V^\top = \SVD(\cX_{(3)})$
\State Apply soft-thresholding: $S' = \max(S - \lambda_2/2, 0)$
\State Combine soft-thresholded SVD into a matrix: $\widehat{\cX}_{(3)} = U \cdot \diag(S') \cdot V^\top$
\State Set $\widehat{\cX} = \ttm_3^{-1}(\widehat{\cX}_{(3)})$
\end{algorithmic}
\end{algorithm}

\subsection{Extra experiments on covariance estimation}
\label{subsection:extra_experiments}

Here we study the performance of tensor decomposition algorithms in the setup of Section~\ref{section: experiments}. First, we repeat experiments of Section~\ref{section: experiments} for $n = 4000$, see Table~\ref{tab:covariance_algorithm_comparison-4000 samples}.

\begin{table}[ht]
\centering
\caption{Performance comparison of tensor decomposition algorithms for $n = 4000$. Relative errors were averaged over 16 repeats of the experiment, empirical standard deviation is given after $\pm$ sign. Best results are boldfaced.}
\label{tab:covariance_algorithm_comparison-4000 samples}
\begin{tabular}{lccc}
\toprule
\multirow{2}{*}{Metric} & \multicolumn{3}{c}{Algorithm} \\
\cmidrule(lr){2-4}
 & Sample Mean & TT-HOSVD & HardTTh \\
\midrule
Relative Error &$0.430 \pm 0.007$ & $0.105 \pm 0.008$ & $\mathbf{0.054 \pm 0.002}$ \\
Time (seconds) & $0.0039 \pm 0.0015$ & $0.64 \pm 0.15$ & $3.2\pm 3.3$ \\
\bottomrule
\end{tabular}

\begin{tabular}{lccc}
\toprule
\multirow{2}{*}{Metric} & \multicolumn{3}{c}{Algorithm} \\
\cmidrule(lr){2-4}
 & Tucker & Tucker+HOOI & PRLS \\
\midrule
Relative Error &$0.105 \pm 0.007$ & $\mathbf{0.054 \pm 0.002}$ & $0.217 \pm 0.015$ \\
Time (seconds) & $30.7 \pm 3.9$ & $51.5 \pm 3.9$ & $0.8 \pm 1.1$ \\
\bottomrule
\end{tabular}
\end{table}

Second, we study the dependence of $\sin \Theta$-distance of estimated singular subspaces to singular subspaces of matricizations of $\cT^*$ on the number of iterations $T$ and the sample size $n$. Matrices $\widehat{U}_0, \widehat{U}_T, \widehat{V}_0, \widehat{V}_T$ are defined in Algorithm~\ref{algo: order 3 TT-SVD}. As before, the number of additional iterations is taken $10$. The results are presented in Table~\ref{tab:sin theta study}.

\begin{table}[ht]
    \centering
    \caption{The study of $\sin \Theta$-distance from estimated singular subspaces to singular subspaces of matricizations of $\cR(\Sigma)$. Average errors and standard deviations are obtained after 16 repeats of the experiment. The setup is defined in Section~\ref{section: experiments}.}
    \label{tab:sin theta study}
    \begin{tabular}{l|ccccc}
    \toprule
         & $n = 500$  & $n = 2000$ & $n = 5000$ & $n = 6000$ & $n = 7000$ \\
         \hline
         $\sin \Theta(\Img \widehat{U}_0, \Img U^*)$ &  $1.0 \pm 0.0$ & $1.0 \pm 0.0$ & $0.8 \pm 0.3$ & $0.8 \pm 0.2$ & $0.6 \pm 0.3$ \\
          $\sin \Theta(\Img \widehat{V}_0, \Img V^*)$ & $1.0 \pm 0.0$ & $1.0 \pm 0.0$ & $1.0 \pm 0.0$ & $0.90 \pm 0.14$ & $0.9 \pm 0.2$ \\
         $\sin \Theta(\Img \widehat{U}_T, \Img U^*)$ & $1.0 \pm 0.0$ & $0.33 \pm 0.08$ & $0.17 \pm 0.04$ & $0.13 \pm 0.03$ & $0.13 \pm 0.02$ \\
         $\sin \Theta(\Img \widehat{V}_T, \Img V^*)$ & $1.0\pm  0.0$ & $0.46 \pm 0.17$ & $0.21 \pm 0.03$ & $0.18 \pm 0.05$ & $0.17 \pm 0.02$ \\
    \bottomrule
    \end{tabular}
\end{table}

For scalability study we increase the number of parameters from $10^6$ to $7.4\cdot 10^8$ for $1000$ samples. One can see that our methods scales successfully, even winning comparison with Tucker+HOOI. The results are shown in Table ~\ref{tab:covariance_algorithm_comparison_gigantic}. Next, we increase number of parameters up to $4\cdot 10^9$ for $1000$ and $2000$ samples. Unfortunately, Tucker+HOOI does not show ability for scaling due to enormous time overhead, so results in Table \ref{tab:covariance_algorithm_comparison_enormous} are provided excluding it.

\begin{table}[ht]
\centering
\caption{Performance comparison of tensor decomposition algorithms for $n = 1000$, $p=q=r=30$. Best results are boldfaced.}
\label{tab:covariance_algorithm_comparison_gigantic}
\begin{tabular}{lccc}
\toprule
\multirow{2}{*}{Metric} & \multicolumn{3}{c}{Algorithm} \\
\cmidrule(lr){2-4}
 & Sample Mean & TT-HOSVD & HardTTh \\
\midrule
Relative Error & $4.448$ & $0.216$ & $\mathbf{0.065}$ \\
Time (seconds) & $3.504$ & $867.756$ & $1007.069$ \\
\bottomrule
\end{tabular}

\begin{tabular}{lccc}
\toprule
\multirow{2}{*}{Metric} & \multicolumn{3}{c}{Algorithm} \\
\cmidrule(lr){2-4}
 & Tucker & Tucker+HOOI & PRLS \\
\midrule
Relative Error &$0.192$ & $0.110$ & $4.422$ \\
Time (seconds) & $14601.442$ & $31256.665$ & $703.230$ \\
\bottomrule
\end{tabular}
\end{table}

\begin{table}[ht]
\centering
\caption{Performance comparison of tensor decomposition algorithms for $p=q=r=40$. Best results are boldfaced.}
\label{tab:covariance_algorithm_comparison_enormous}

\begin{tabular}{lccccc}
\toprule
\multirow{2}{*}{Metric (n = 1000)} & \multicolumn{5}{c}{Algorithm} \\
\cmidrule(lr){2-6}
 & Sample Mean & TT-HOSVD & HardTTh & Tucker & PRLS \\
\midrule
Relative Error & 6.86 & 0.21 & \textbf{0.055}& 0.19 & 6.82 \\
Time (seconds) & 20.94 & 5095.54 & 6873.25 & 84360.27 & 5872.09 \\
\bottomrule
\end{tabular}

\begin{tabular}{lccccc}
\toprule
\multirow{2}{*}{Metric (n = 2000)} & \multicolumn{5}{c}{Algorithm} \\
\cmidrule(lr){2-6}
 & Sample Mean & TT-HOSVD & HardTTh & Tucker & PRLS \\
\midrule
Relative Error & 4.87 & 0.19 & \textbf{0.038}& 0.1845 & 4.83 \\
Time (seconds) & 20.63 & 5839.20 & 6889.42 & 84476.38 & 5825.16 \\
\bottomrule
\end{tabular}
\end{table}

We provide ablation study on the effect of ranks on the error rate. We expect that large increase of ranks leads to broken spectral gap condition, thus, models takes part of the noise as vital information. Large decrease leads to loss of vital information, since relevant singular values may be erased. Despite that, small perturbation in ranks may lead to better bias-variance tradeoff, thus, decreasing error overall. See Figure \ref{fig:rank_ablation} for details.

\begin{figure}[ht]
\centering
\includegraphics[width=0.60\textwidth]{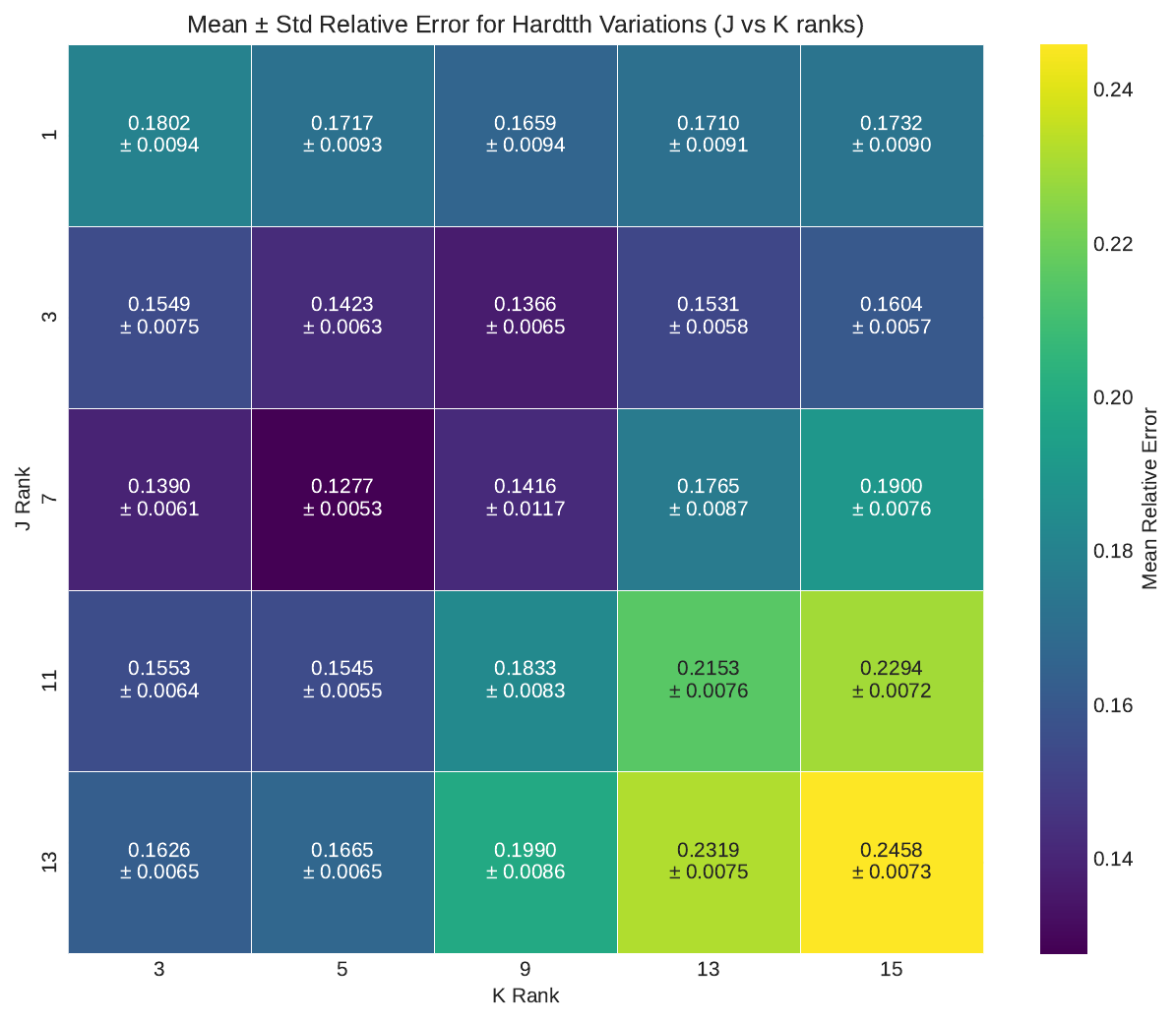}
\caption{Rank ablation study for covariance with parameters $(J, K) = (7, 9)$, $p=q=r=10$, averaged by 32 runs.}
\label{fig:rank_ablation}
\end{figure}

\subsection{Experiments on tensor estimation}

 This section is devoted to experiments that did not have enough space in the main text. In particular, we numerically study the impact of additional iterations of Algorithm~\ref{algo: order 3 TT-SVD} in the tensor estimation problem. We do not consider the misspecified case, and, given $(J, K)$ and $p, q, r$, generate $\cT^*$ as follows. First, we generate matrices $U_j, W_{jk}, V_k$ from model \eqref{eq:kronecker_tt_model} according to the matrix initialize method - random, random symmetric, symmetric with special spectrum decay (i.e. inverse quadratic, exponential, linear, etc.). We will refer to these matrices $U_j, W_{jk}, V_k$ as sub-components of matrix 
\begin{align*}
    S = \sum_{j = 1}^J \sum_{k = 1}^K U_j \otimes W_{jk} \otimes V_k \in \R^{pqr \times p qr }, 
\end{align*}
and reshape it to a tensor $\cT^* = \cR(S)$. It is ease to see that such procedure is equivalent to the direct assignment of TT factors, due to Equation \eqref{eq:kronecker_tt_rearranged}.
Then, choosing a noise level $\sigma$, we generate a noise tensor $\widehat{\cE}$ as a random normal with $\sigma$ as its standard deviation and compute
\begin{align*}
    \cY = \cT^* + \widehat{\cE}.
\end{align*}

Our code supports some other testing regimes:  one can choose the $S$ structure directly (block-Toeplitz, structure \eqref{eq:kronecker_product_model}, etc.) supporting misspecification case, and rank selection method (via hard thresholding, effective rank, absolute error). For more information on rank selection see display \eqref{eq: definition of J estimator}.

For the specific experiment, we vary the algorithms to test, as well as the actual ranks and sizes of the components $U_j, W_{jk}, V_k$. For PRLS algorithm, due to its special setup, we tune $\lambda_1, \lambda_2$ parameters on a log-scale. In the Table \ref{tab:algorithm_comparison} one can see, that our method also shows less variance, compared to the previous algorithms, such as sample mean or Algorithm \ref{algo:tsilikagridis_thresholding} with noise variance equal to 0.3.

\begin{table}[ht]
\centering
\caption{Performance comparison of tensor decomposition algorithms under medium noise conditions. The best results are boldfaced.}
\label{tab:algorithm_comparison}
\begin{tabular}{lccc}
\toprule
\multirow{2}{*}{Metric} & \multicolumn{3}{c}{Algorithm} \\
\cmidrule(lr){2-4}
 & Sample Mean & TT-HOSVD & HardTTh \\
\midrule
Relative Error & $0.3643 \pm 0.0135$ & $0.0449 \pm 0.0018$ & $\mathbf{0.0357 \pm 0.0015}$ \\
Time (seconds) & $0.0204 \pm 0.0096$ & $4.4732 \pm 1.8079$ & $7.5522 \pm 2.1386$ \\
\bottomrule
\end{tabular}

\vspace{0.3cm}

\begin{tabular}{lccc}
\toprule
\multirow{2}{*}{Metric} & \multicolumn{3}{c}{Algorithm} \\
\cmidrule(lr){2-4}
 & Tucker & Tucker+HOOI & PRLS \\
\midrule
Relative Error & $0.0439 \pm 0.0016$ & $\mathbf{0.0357 \pm 0.0015}$ & $0.1130 \pm 0.0037$ \\
Time (seconds) & $56.7830 \pm 16.3132$ & $106.5766 \pm 25.2531$ & $0.7076 \pm 0.1160$ \\
\bottomrule
\end{tabular}
\end{table}

Now consider the case of a low SNR setting (high-noise regime, fast spectrum decay). This case violates the assumptions of Theorem \ref{theorem: Sigma estimator performance}. It can be seen that the methods perform poorly and do not restore the signal (the relative error remains at the level of 0.3), thus, demonstrating the necessity of theorem's conditions. The experiment below was conducted for the case when sub-components of $S$ spectra decrease as inverse square sequence (see Table \ref{tab:quadratic_decay} for details).

\begin{table}[ht]
\centering
\caption{Performance of tensor decomposition algorithms under inverse quadratic decay of spectrum. In case of low SNR we observe that iterative methods perform worse than one-shot and both do not restore signal. The best result is boldfaced.}
\label{tab:quadratic_decay}
\begin{tabular}{lccc}
\toprule
\multirow{2}{*}{Metric} & \multicolumn{3}{c}{Algorithm} \\
\cmidrule(lr){2-4}
 & Sample Mean & TT-HOSVD & HardTTh \\
\midrule
Relative Error & $0.3508 \pm 0.0004$ & $\mathbf{0.0251 \pm 0.0001}$ & $0.0279 \pm 0.0003$ \\
Time (seconds) & $0.0509 \pm 0.0166$ & $13.9748 \pm 4.1845$ & $282.7375 \pm 145.8327$ \\
\bottomrule
\end{tabular}
\end{table}

It may be useful to examine the spectrum of matrix $S$ and matricizations in order to understand how the behavior of algorithms varies in different scenarios. Figure \ref{fig:spectrum_comparison} illustrates this. These plots were constructed for tensor-train rank $(J,K)$ pairs of 7 and 9, respectively, with sub-components having a size of $10\times10$. The total matrix size was $1000 \times 10000$, composed of these sub-components.

\begin{figure}[ht]
\centering
\begin{minipage}[b]{0.4\linewidth}
    \centering
    \includegraphics[width=\linewidth, height=4cm]{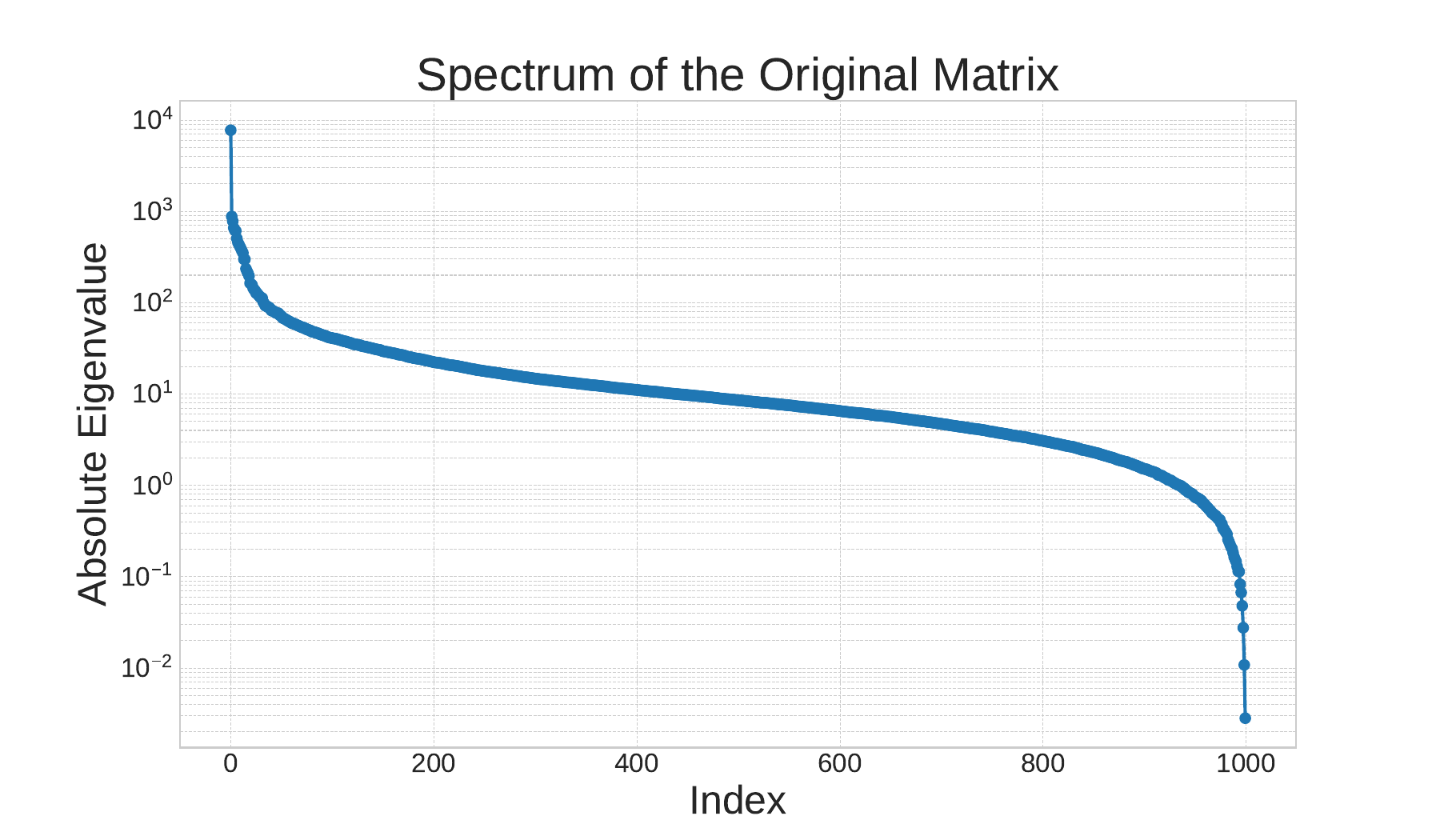} 
    \caption*{(a) Matrix $S$ spectrum}
\end{minipage}
\hfill
\begin{minipage}[b]{0.58\linewidth}
    \centering
    \includegraphics[width=\linewidth]{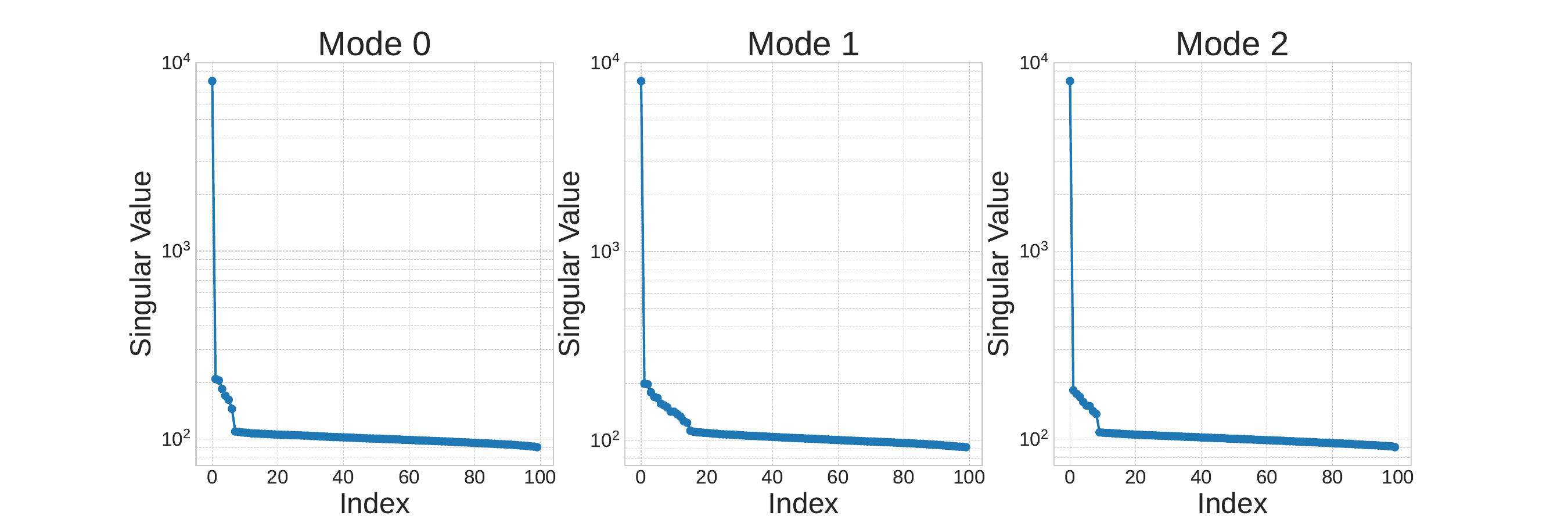}
    \caption*{(b) Singular values of matricizations}
\end{minipage}
\caption{Spectrum of the objectives in case of random sub-components. As one can see, dense spectrum of matrix $S$ with noise become separable for matricizations.}
\label{fig:spectrum_comparison}
\end{figure}

To experimentally confirm the necessity of the conditions of our theorem, we plotted the relationship between singular values and noise levels, as well as the relative error and noise levels. Our findings indicate that, after a certain threshold, our algorithm no longer effectively mitigate noise but instead overfit to it, resulting in inferior performance compared to one-step methods such as TT-HOSVD (see Figure \ref{fig:noise_increasing}).

\begin{figure}[ht]
\centering
\includegraphics[width=0.6\textwidth]{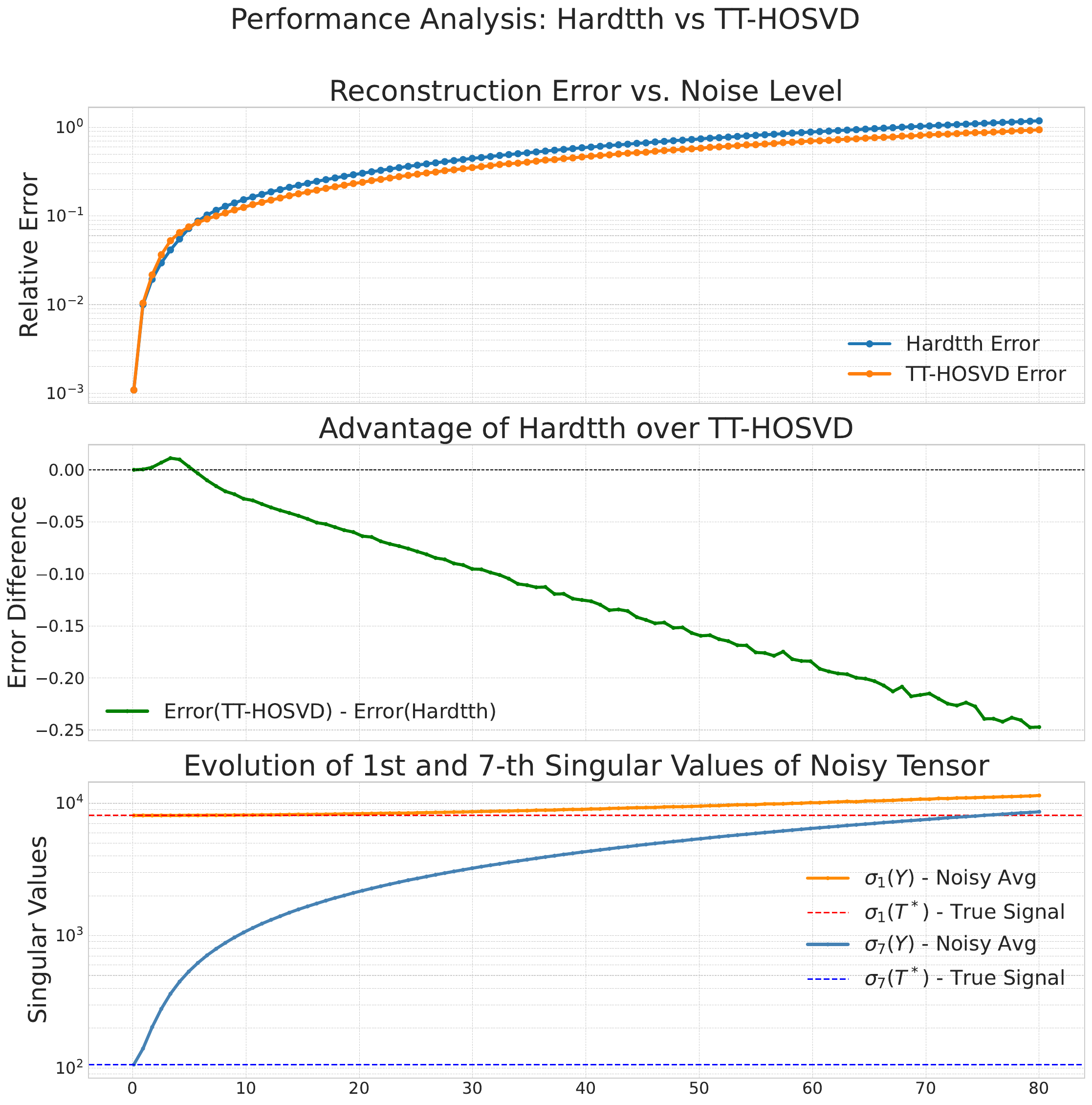}
\caption{Performance of tensor decomposition algorithms and spectrum behavior under noise increase.}
\label{fig:noise_increasing}
\end{figure}

\end{document}